\documentclass{article}

\usepackage{arxiv}
\usepackage{fullpage}
\usepackage[T1]{fontenc}
\usepackage[utf8]{inputenc}
\usepackage[english]{babel}
\usepackage{booktabs}
\usepackage[shortlabels]{enumitem}
\usepackage[nocompress]{cite}

\usepackage[compatibility=false]{caption}
\usepackage{subcaption}
\usepackage{xcolor}
\usepackage{graphics}
\DeclareGraphicsExtensions{.pdf,.png,.jpg} 

\usepackage{algorithm}
\usepackage{algorithmic}
\usepackage{amssymb}
\usepackage{amsthm}
\usepackage{mathtools}
\usepackage{dsfont}
\usepackage{bm}

\usepackage[hypertexnames=false]{hyperref}
\usepackage{cleveref}

\newtheorem{theorem}{Theorem}[section]
\newtheorem{lemma}{Lemma}
\newtheorem{definition}{Definition}
\newtheorem{corollary}{Corollary}
\newtheorem{proposition}{Proposition}
\newtheorem{hypothesis}{Hypothesis}
\newtheorem{remark}{Remark}

\crefname{hypothesis}{Hypothesis}{Hypotheses}
\Crefname{hypothesis}{Hypothesis}{Hypotheses}

\newcommand{\gt}{\bm{x}_{\mathrm{gt}}}
\newcommand{\N}{\mathbb{N}}
\newcommand{\R}{\mathbb{R}}

\newcommand{\norm}[1]{\left\lVert #1 \right\rVert}
\newcommand{\Rreg}{\mathcal{R}}
\newcommand{\Ssol}{\mathcal{S}}
\newcommand{\Ncal}{\mathcal{N}}

\newcommand{\bx}{{\bm{x}}}
\newcommand{\by}{{\bm{y}}}
\newcommand{\bz}{{\bm{z}}}
\newcommand{\bw}{{\bm{w}}}
\newcommand{\bh}{{\bm{h}}}
\newcommand{\bu}{{\bm{u}}}
\newcommand{\bv}{{\bm{v}}}
\newcommand{\be}{{\bm{e}}}
\newcommand{\br}{{\bm{r}}}
\newcommand{\bs}{{\bm{s}}}
\newcommand{\ba}{{\bm{a}}}
\newcommand{\bb}{{\bm{b}}}
\newcommand{\bg}{{\bm{g}}}
\newcommand{\bone}{\mathbf{1}}

\title{Iterated graph Laplacian for image restoration problems}

\author{Stefano Aleotti\thanks{Department of Science and High Technology (DiSAT), University of Insubria, Como 22100, Italy (\email{stefano.aleotti@uninsubria.it}, \email{marco.donatelli@uninsubria.it}, \email{pmaurino@uninsubria.it}).}
\and Davide Bianchi\thanks{School of Mathematics (Zhuhai), Sun Yat-sen University, Zhuhai 519082, China
(\email{bianchid@mail.sysu.edu.cn}). D. Bianchi is supported by the Startup Fund of Sun~Yat-sen~University.}
\and Florian Bo\ss mann\thanks{Department of Mathematics, Harbin Institute of Technology, Harbin 150001, China
(\email{f.bossmann@hit.edu.cn}).}
\and Marco Donatelli\footnotemark[1]
\and Pietro Maurino\footnotemark[1]}

\newcommand{\email}[1]{\href{mailto:#1}{\texttt{#1}}}

\ifpdf
\hypersetup{
  pdftitle={Iterated graph Laplacian for Image Restoration Problems},
  pdfauthor={S. Aleotti, D. Bianchi, F. Bossmann, M. Donatelli, and P. Maurino}
}
\fi

\begin{document}

\maketitle

\begin{abstract}
We study the graph Laplacian operator as a regularizer in a generalized Tikhonov framework for linear ill-posed problems. The Laplacian is updated iteratively from the current reconstruction, so that progressively sharper structural information about the solution is fed into the regularization term. We introduce three schemes: a \emph{standard} one that rebuilds the Laplacian from each new iterate; an \emph{error-equation} scheme that, following the error-based formulation of iterated Tikhonov regularization, builds the Laplacian from an estimate of the reconstruction error rather than of the image itself; and a \emph{mixed} scheme combining the two. We establish convergence of all three schemes for noisy data under a priori parameter and stopping rules as the noise level tends to zero. Numerical experiments in two-dimensional computed tomography and image deblurring show consistent gains in reconstruction quality and sharper recovery of fine details.
\end{abstract}



\section{Introduction}

We are concerned with linear image restoration problems, whose model equation reads
\begin{equation}\label{eq:Model_Equation}\tag{ME}
	K\bx = \by^\delta,
\end{equation}
where $K \colon X \to Y$ is a discretized, inherently ill-posed linear operator from a finite-dimensional space of images $X$ into a (possibly infinite-dimensional) Hilbert space $Y$, modeling a degradation phenomenon such as blurring, tomographic projection, subsampling, or data masking, to name a few representative examples; $\bx\in X$ is the unknown image; and $\by^\delta\in Y$ is the observed data, a noisy perturbation of the exact data $\by$, as we now make precise.

We assume that there is a fixed ground truth $\gt\in X$ that generated the exact data $\by\coloneqq K\gt$, and that only a noisy observation $\by^\delta$ of $\by$ is available, that is,
\begin{equation*}
	\by^\delta \coloneqq \by + \bm{\eta}, \qquad \norm{\by^\delta - \by}_Y\leq \delta,
\end{equation*}
where $\bm{\eta}\in Y$ is a random perturbation, $\delta>0$ is a bound on the noise level, and $\norm{\cdot}_Y$ is the norm of $Y$.
The goal is to recover a reliable approximation of $\gt$ from $\by^\delta$. Solving \cref{eq:Model_Equation} exactly is not a viable way to achieve this: $\by^\delta$ need not belong to the range of $K$, so that \cref{eq:Model_Equation} may have no solution at all; $K$ need not be injective, so that even the unperturbed equation $K\bx=\by$ may have infinitely many solutions; and, although the finite dimensionality of $X$ ensures that the Moore--Penrose pseudo-inverse $K^\dagger\colon Y\to X$ of $K$ is defined on all of $Y$ and provides a least-squares solution $K^\dagger\by^\delta$ for~\cref{eq:Model_Equation}, $K$ discretizes an ill-posed operator, so that the norm of $K^\dagger$ is not bounded uniformly with respect to the discretization level and $K^\dagger\by^\delta$ is highly sensitive even to small perturbations of the data. In other words, the inverse problem associated with~\cref{eq:Model_Equation} is ill-posed and requires regularization: in place of an exact solution of~\cref{eq:Model_Equation}, one looks for a regularized solution, namely a stable approximation of a distinguished solution of the exact equation $K\bx=\by$, selected by the regularization term (see \Cref{def:R-min-solution}). For a comprehensive overview of ill-posed inverse problems and regularization techniques, we refer the reader to \cite{engl1996regularization,hansen1998rank,scherzer2009variational}.

Among the various regularization methods in the literature, we consider here the $\ell^2$--$\ell^q$ regularization of~\cref{eq:Model_Equation} (see, e.g., \cite{CH14,LMRS15,HLMRS17}), which computes a regularized reconstruction $\bx^\delta_\alpha\in X$ as
\begin{equation}\label{model_eq2}
	\bx^\delta_\alpha \in	\underset{\bx \in X}{\arg\min} \left\{ \frac{1}{2}\norm{K\bx - \by^\delta}_Y^2 + \frac{\alpha}{q}\norm{D \bx}_q^q \right\},
\end{equation}
where $D\colon X\to Z$ is the regularization operator, a linear mapping into a finite-dimensional coefficient space $Z$, satisfying $\ker(D)\cap\ker(K)=\{\boldsymbol{0}\}$, see \cite[Chapter 8]{engl1996regularization} and \cite{morozov2012methods}; $\norm{\cdot}_q$, with $q\in[1,2]$, is the $\ell^q$ norm on $Z$; and $\alpha>0$ is the regularization parameter. The first term, $\frac{1}{2}\norm{K\bx-\by^\delta}_Y^2$, quantifies the fidelity of a candidate image $\bx$ to the observed data; the second term, $\frac{1}{q}\norm{D\bx}_q^q$, penalizes candidates that are not compatible with the prior information encoded in $D$; and $\alpha$ balances the trade-off between data fidelity and regularization. For the graph Laplacian operators considered in this paper, $Z=X$, whereas, for instance, a discrete gradient maps into a space $Z$ with several components per pixel. Throughout, $\bx$ denotes a generic candidate image, $\gt$ the ground truth, and reconstructions computed from the noisy data $\by^\delta$ carry the superscript $\delta$, as in $\bx^\delta_\alpha$.

Typical choices for the regularization operator $D$ include linear differential operators, see, e.g., \cite{ROF92,hansen2006deblurring,DR14}.
However, the last decade has seen an increasing interest in nonlocal models and techniques from graph theory \cite{gilboa2007nonlocal,gilboa2009nonlocal,Peyre2008nonlocal,arias2009variational}. This approach has been further investigated in the context of image deblurring and computed tomography (CT) problems in \cite{bianchi2021graph_approximation,buccini2021graph,bianchi2021graph,lou2010image,zhang2010bregmanized}.

In particular, in \cite{bianchi2025data}, the authors investigate a variational method for ill-posed problems, named graphLa\texttt{+}$\Psi$, which
embeds a graph Laplacian operator in the regularization term. The novelty of this method lies in constructing the graph Laplacian based on a preliminary approximation of the solution, which is
obtained using any existing reconstruction method $\Psi$ from the literature. The need for an initial reconstruction is due to the fact that a graph structure should incorporate features such as interfaces and discontinuities of the true solution. However, it was observed that generating the graph directly from the observed,
noisy data $\by^{\delta}$ yields poor results for imaging tasks such as
deblurring or computed tomography. This is unsurprising, since in general
$\by^{\delta}$ does not share the same features as $\gt$, and may not even lie
in the same domain.

This initial preprocessing step plays an important role in the quality of the
final reconstruction: the richer the detail encoded in the initial estimate,
the more effectively the graph Laplacian propagates that information throughout
the optimization. The role of this step has already been examined in depth
within the graphLa\texttt{+}$\Psi$ method, where several reconstruction
techniques, both standard and data-learned, were compared depending on the
underlying inverse problem.

Updating the graph along the iterations is an established technique in nonlocal and graph-based imaging. In~\cite{Peyre2008nonlocal,zhang2010bregmanized,arias2009variational}, the nonlocal weights are recomputed from the current iterate inside an alternating or Bregman-type iteration, and the emphasis lies on the design of the nonlocal functional and of the optimization algorithm. A more recent example, closer in spirit to the present work, is the joint reconstruction and segmentation method of Budd et al.~\cite{budd2023joint}. There, a graph whose weights are computed from features of the current reconstruction carries a graph Ginzburg--Landau segmentation energy; reconstruction and segmentation are updated alternately, with the segmentation steering the reconstruction through the dependence of the graph weights on the image, so that the graph changes from one iteration to the next, as it does in the schemes studied here.

What these approaches share is an iterative refinement principle: rather
than committing to a single initial estimate, one repeatedly improves it and
feeds the result back into the reconstruction. This principle is not specific to
graph-based regularization, and it has long been exploited for inverse problems
more broadly~\cite{nagy2026mixed}. A classical instance is the iterated Tikhonov
method~\cite[Chapter~6]{engl1996regularization}, which has been studied
extensively and extended to nonsmooth regularization terms as
well~\cite{huang2013nonstationary,aleotti2025nested,jin2012nonstationary,jin2014nonstationary,bianchi2023uniformly,bianchi2026iterated,bianchi2015iterated,bianchi2025convergence}.

In this work, we introduce three \emph{iterated graph Laplacian regularization
schemes}. In each of them, the reconstruction from which the graph Laplacian is
built is updated iteratively, so that the regularizer is driven by progressively
more accurate structural information about the solution.

The common building block is one graph Laplacian regularization step. Given a
reference image $\bz\in X$, we construct the associated graph Laplacian
$\Delta_{\bz}$ (see \Cref{sec:graph_theory}) and solve
\begin{equation}\label{eq:graphLa+}
\bx_{\alpha,\bz}^{\delta} \in
\underset{\bx \in X}{\arg\min}
\left\{
\tfrac{1}{2}\norm{K\bx - \by^\delta}_Y^2
+ \tfrac{\alpha}{q}\|\Delta_{\bz}\bx\|_q^q
\right\}.
\end{equation}
Here $\bx_{\alpha,\bz}^{\delta}$ denotes the resulting reconstruction, which depends on the noisy data $\by^\delta$, on the regularization parameter $\alpha>0$, and on the reference image $\bz$. The three schemes differ in how this step is iterated.

The \emph{standard iterative scheme} (\Cref{alg:standard-iter})
feeds each new reconstruction back as the
reference for the next step: the graph Laplacian is recomputed from the current
iterate, \cref{eq:graphLa+} is solved again, and the procedure is repeated. As
the iterates improve, so does the structural prior encoded by the graph
Laplacian.

The \emph{error-equation iterative scheme} (\Cref{alg:error-iter})
takes a different viewpoint. Instead
of refining the solution directly, it refines the initial reconstruction $\bz$
by estimating the error. Following the iterated Tikhonov method, the
data-fidelity term $\tfrac12\norm{K\bx-\by^\delta}_Y^2$ is replaced by
$\tfrac12\norm{K\bh-\br^\delta}_Y^2$, where $\br^\delta=\by^\delta-K\bz$ is the
residual and $\bh$ the error correction to be computed. The graph Laplacian is now built from an estimate of the reconstruction error
rather than of the solution.

The \emph{mixed iterative scheme} (\Cref{alg:mixed-iter})
combines the standard iterative approach with the
error-equation strategy: it begins with a small number of standard iterations to
recover the dominant structure, and then passes to the error-equation
formulation to sharpen the remaining fine-scale details. The motivation for this
combination lies in the complementary behavior of the two approaches.
Empirically, after a few iterations of the standard scheme the performance
metrics reach a plateau, so that further iterations yield no significant gain in
reconstruction quality. We interpret this as an indication that the graph
Laplacian built from the current iterate can no longer extract additional
structural information about the ground-truth solution $\gt$. Once this plateau
is reached, switching to the error-based formulation recovers further detail and
improves the reconstruction.

The standard iterative scheme was first proposed in~\cite{bianchi2024improved}, under the name \texttt{it-graphLa}$\Psi$,
for acoustic impedance inversion in seismic exploration, but only as a heuristic
refinement, without any convergence analysis. A related idea was developed
independently in~\cite{bajpai2026convergence}: there too the graph Laplacian is
recomputed from the current iterate, but the reconstruction is updated by an
explicit Landweber-type step with an extra graph term,
\[
\bx_{k+1}^{\delta}
= \bx_k^{\delta}
- \alpha_k^{\delta}\, K^*\!\left(K\bx_k^{\delta} - \by^\delta\right)
- \beta_k^{\delta}\, \Delta_{\bx_k^{\delta}}\bx_k^{\delta},
\]
where $\bx_k^{\delta}$ is the current iterate, $\alpha_k^{\delta},\beta_k^{\delta}>0$ are step sizes, and $K^*$ denotes the adjoint of $K$, and the iteration is terminated by the Morozov discrepancy principle. Avoiding a
variational subproblem at each update makes that scheme computationally lighter,
since one step costs essentially a forward/adjoint evaluation plus sparse graph
operations. Its convergence theory, however, is established in a fixed
finite-dimensional setting: the monotonicity constants and the admissible
step-size ranges depend on the discretization and on the radius of a ball,
centered at the initial reconstruction, which is assumed to contain the exact
solution.

Our analysis follows a different route. We retain the variational
formulation~\cref{eq:graphLa+} and, beyond the standard scheme, introduce and
analyze the error-equation and mixed schemes, which have no counterpart
in~\cite{bianchi2024improved,bajpai2026convergence}. We prove convergence of all
three schemes for noisy data, as the noise level tends to zero, under an a priori
parameter/stopping rule. The error-equation and mixed schemes
additionally require a uniform bound on the iterates, which holds, for instance,
whenever the iterates remain in a fixed grayscale range. Throughout, all constants are
tracked explicitly and shown to be independent of the discretization level, that
is, of the number of pixels $N$, and of the reference image $\bz$.

In particular, these
results rest on minimal assumptions: the existence of a uniformly bounded exact
solution, and a joint coercivity condition coupling the forward operator and the
graph Laplacian that goes back to
Morozov~\cite{morozov2012methods} (see \Cref{sec:preliminary_estimates}). Both are required only to hold uniformly in
$N$ and $\bz$, and are otherwise standard. 

With respect to the changing-graph methods recalled above, and to the joint reconstruction and segmentation scheme of~\cite{budd2023joint} in particular, the differences are of two kinds. The first concerns the role of the graph. In our schemes the graph Laplacian enters directly as the penalty operator of the convex problem~\cref{eq:graphLa+}, posed on the reconstruction or on an additive error correction, with no segmentation variable and no joint nonconvex energy; the graph is held fixed within each step and rebuilt only between steps, so that each step is a standard $\ell^2$--$\ell^q$ problem for which efficient solvers and parameter-choice rules are available. The second concerns the questions asked of the iteration. Rather than convergence of the iterates to a critical point of an energy for fixed data as in~\cite{budd2023joint}, we study the regularization properties of the iteration: stability for a fixed number of steps and convergence to a solution of the exact equation $K\bx=\by$ as $\delta\to0$, with constants independent of the discretization level $N$ and of the reference image. The error-equation construction, in which the graph is built from an estimate of the error rather than of the solution, and the mixed scheme have, to our knowledge, no counterpart in the literature.

In summary, the proposed schemes provide a graph-based refinement of an available reconstruction that is supported by a regularization theory whose constants do not depend on the discretization level. They are intended for problems in which the forward operator is known and an initial reconstruction $\bx_0^\delta$, produced by a classical method or by a trained network, captures the main structures of the image but still leaves out recoverable details: the graph transfers this structural information into the regularizer, each step re-imposes consistency with the measured data through $K$, and updating the graph lets structure that is invisible in $\by^\delta$, and only partly visible in $\bx_0^\delta$, be exploited.

The paper is organized as follows. \Cref{sec:graph_theory} presents the discrete
image model, the graph construction, and the dimension-free graph estimates used
throughout the analysis. \Cref{sec:preliminary_estimates} introduces the two main assumptions and the 
fixed-reference variational problem, and proves the associated well-posedness and
stability results. \Cref{sec:iterschemes} defines the standard, error-equation,
and mixed iterative graph Laplacian schemes. \Cref{sec:convergence_standard}
analyzes the regularization properties of the standard scheme, and \Cref{sec:convergence_error}
treats the error-equation and mixed schemes. \Cref{sec:numerical_results} reports
extensive numerical experiments on image deblurring and computed tomography.
Finally, \Cref{sec:appendix-coercivity} discusses joint coercivity at fixed
discretization and establishes dimension-free estimates for representative
periodic-blur and full-angle CT models under the stated assumptions.

\section{Graph setting and graph estimates}\label{sec:graph_theory}

The variational approach considered in \cref{eq:graphLa+} and all the iterative strategies introduced later rely on the graph Laplacian operator as a regularization tool. We therefore start by describing the discrete image model, the theoretical framework, and the graph construction.

Let $P$ be a finite set of pixels over a field of view $\Omega\subset \R^2$, with $N\coloneqq|P|$, and let $X$ be the space of real-valued functions on $P$; fixing an ordering of $P$ identifies $X$ with $\R^N$.
A grayscale image is an element $\bx\in X$ with values in $[0,1]$, i.e., a function $\bx\colon P\to[0,1]$, so the set of grayscale images is the box $[0,1]^P\subset X$. For $1\le q<\infty$ we use the normalized discrete norm
\[
\norm{\bx}_q\coloneqq\left(\frac1N\sum_{i\in P}|\bx(i)|^q\right)^{1/q},
\qquad \bx\in X,
\]
and
\[
\norm{\bx}_\infty\coloneqq\max_{i\in P}|\bx(i)|.
\]
The data space $Y$ is assumed to be a Hilbert space, possibly infinite-dimensional, with norm $\norm{\cdot}_Y$ induced by its inner product. The model forward map is a linear operator between the Banach spaces just defined,
\begin{equation}\label{eq:model_operator}
K\colon (X, \|\cdot\|_q) \to (Y, \|\cdot\|_Y).
\end{equation}
We will often omit the space norms, since they will be clear from the context, and we write $\|A\|$ for the operator norm of a linear operator $A$.

In the applications that we consider in \Cref{sec:numerical_results}, $Y$ is a finite-dimensional normalized Euclidean space, as  is $X$, since the measurements are discretized, and $K$ is a matrix. The theoretical analysis developed here, however, remains valid when $Y$ is infinite-dimensional. This is the case, for instance, when only the unknown is discretized: if $\mathcal{K}\colon\mathcal{X}\to Y$ is the continuous forward operator on a space $\mathcal{X}$ of functions on the image domain, such as $L^2(\R^2)$ in \Cref{sec:appendix-CT}, and $\iota\colon X\to\mathcal{X}$ embeds pixel images as piecewise constant functions, then $K\coloneqq\mathcal{K}\circ\iota$ maps the finite-dimensional space $X$ into the infinite-dimensional space $Y$; fully discrete measurements correspond to composing further with a projection onto a finite-dimensional subspace of $Y$. Any such discretization fits the framework, since the analysis of \Cref{sec:preliminary_estimates,sec:convergence_standard,sec:convergence_error} uses only the linearity of $K$ and \Cref{hyp:consistency,hyp:joint-coercivity}; the uniformity in $N$ of the constants in \Cref{hyp:joint-coercivity} is an additional requirement on the discretization, discussed in \Cref{sec:appendix-coercivity} for representative model operators. To keep the notation light, we do not carry the embedding, the projection, or the index $N$ explicitly, and simply write $K\colon X\to Y$ as in \cref{eq:model_operator}.

For a comprehensive introduction to the graph-based framework adopted here, we refer to \cite{keller2021graphs}. 

\begin{definition}[Graph and graph Laplacian]
A graph over $P$ is a symmetric function $\omega\colon P\times P\to[0,\infty)$ such that $\omega(i,i)=0$ for every $i\in P$. Each pair $(i,j)$ such that $\omega(i,j)>0$ is called an \emph{edge}.

The graph Laplacian $\Delta\colon X\to X$ associated with a graph $\omega$ is defined by the action
\begin{equation*}
(\Delta \bu)(i)\coloneqq\sum_{j\in P}\omega(i,j)\bigl(\bu(i)-\bu(j)\bigr),
\qquad i\in P.
\end{equation*}
\end{definition}

We remark that the graph Laplacian does not depend on the ordering chosen for the pixel set $P$, since a different ordering affects its matrix representation only up to a permutation similarity. In the numerical experiments of \Cref{sec:numerical_results} and in the theory below, the graph is built from a reference image using a local Gaussian construction.

We first fix the spatial geometry. We regard $P$ as a finite subset of the integer lattice $\mathbb{Z}^2$, so that each pixel $i\in P$ has integer coordinates $(i_1,i_2)$. For $i,j\in P$,
\[
\operatorname{dist}(i,j)\coloneqq|i_1-j_1|+|i_2-j_2|
\]
denotes the Manhattan distance between the two pixel locations. This spatial distance is only used to select neighbors and is unrelated to the normalized norm $\|\cdot\|_1$ on $X$ introduced above. Different geometric distances may be used; however, for simplicity, we use only the Manhattan distance in this manuscript. The fields of view considered here are rectangular, and so is the discretized
grid that defines the pixels subset of $\mathbb{Z}^2$. However, in principle, the
analysis discussed hereafter can be carried over to fields of view with
different topologies.

Fix now a spatial radius $R\in\N$ and a Gaussian parameter $\sigma>0$. For $i\in P$ define 
\[
\Ncal_R(i)\coloneqq\{j\in P \mid 0< \operatorname{dist}(i,j) \le R\}.
\]
Observe that the \emph{spatial degree} $|\Ncal_R(i)|$ is uniformly bounded:
\begin{equation}\label{eq:spatial-degree_uniform_norm}
d_R\coloneqq\sup_N\max_{i\in P}|\Ncal_R(i)|<\infty .
\end{equation}

Here $P=P_N$ varies with the discretization level $N=|P_N|$, and the supremum is taken over this family of pixel sets. The parameters $R$ and $\sigma$ are fixed independently of $N$. Since there are $4k$ lattice points at Manhattan distance $k$ from a given pixel, we have $d_R\le\sum_{k=1}^{R}4k=2R(R+1)$. For any single reconstruction, $P$ is fixed.

Fix also a contrast cap $B_{\rm c}>0$, independent of $N$, and set
\begin{equation}\label{eq:contrast-cap}
d_{B_{\rm c}}(a,b)\coloneqq\min\{|a-b|,B_{\rm c}\},
\qquad a,b\in\R.
\end{equation}

Given a reference image $\bz\in X$, define
\begin{equation}\label{eq:gaussian-weights}
\omega_\bz(i,j)
\coloneqq\mathds{1}_{\{0<\operatorname{dist}(i,j)\le R\}}
\exp\!\left(-\frac{d_{B_{\rm c}}(\bz(i),\bz(j))^2}{\sigma^2}\right),
\qquad i,j\in P.
\end{equation}
Here $\mathds{1}_{\{0<\operatorname{dist}(i,j)\le R\}}$ is the indicator function, equal to $1$ if $j\in\Ncal_R(i)$ and to $0$ otherwise. The graph Laplacian induced by $\bz$ is
\begin{equation*}
(\Delta_\bz \bu)(i)
\coloneqq\sum_{j\in \Ncal_R(i)} \omega_\bz(i,j)\bigl(\bu(i)-\bu(j)\bigr),
\qquad \bu\in X .
\end{equation*}
We will often call $\bz$ the \emph{graph input} or the \emph{reference image}.

Let us remark that the contrast cap $B_{\rm c}$ acts only on the contrast range, 
namely, the range of absolute differences in grayscale
intensity between pairs of pixels $i$ and $j$, given by
$|\bz(i)-\bz(j)|$, and imposes no constraint on the graph input. It has no effect whenever the values of the graph input lie in an interval of length at most $B_{\rm c}$; in particular, choosing $B_{\rm c}\ge1$ leaves ordinary $[0,1]$-valued grayscale graph inputs unchanged. 

For $q\in[1,2]$, define
\begin{equation}\label{def:graph-regularizer}
\Rreg_\bz^q(\bu)\coloneqq\frac1q\norm{\Delta_\bz \bu}_q^q .
\end{equation}

\subsection{Dimension-free graph estimates}\label{ssec:graph-estimates}

The next estimates are consequences of the local Gaussian construction and the fixed radius $R$. 

\begin{lemma}\label{lem:graph-bounds}
For every $q\in[1,\infty]$, every $N$, every reference image $\bz\in X$, and every $\bu\in X$,
\begin{equation}\label{eq:Delta-uniform-bound}
\norm{\Delta_\bz \bu}_q\le 2d_R\norm{\bu}_q.
\end{equation}
Moreover, for all $\bx,\bz,\bu\in X$,
\begin{equation}\label{eq:Delta-Lipschitz}
\norm{(\Delta_\bx-\Delta_\bz)\bu}_q
\le L_\Delta\,\norm{\bx-\bz}_\infty\,\norm{\bu}_q,
\qquad
L_\Delta\coloneqq\frac{4d_R}{\sigma}\sqrt{\frac{2}{\mathrm{e}}} .
\end{equation}
The constants $2d_R$ and $L_\Delta$ are independent of $N$.
\end{lemma}

\begin{proof}
Let $A_\bz \colon X \to X$ be the adjacency operator
\[
(A_\bz \bu)(i)\coloneqq\sum_{j\in\Ncal_R(i)}\omega_\bz(i,j)\bu(j),
\]
and let $D_\bz \colon X \to X$ be the degree operator
\[
(D_\bz \bu)(i) \coloneqq \operatorname{deg}_\bz(i)\bu(i), \qquad \operatorname{deg}_\bz(i)\coloneqq\sum_{j\in\Ncal_R(i)}\omega_\bz(i,j).
\]
Since $0\le \omega_\bz(i,j)\le1$ and every row and column of the matrix $(\omega_\bz(i,j))_{i,j\in P}$, which represents $A_\bz$ once an ordering of $P$ is fixed, has at most $d_R$ nonzero entries by~\cref{eq:spatial-degree_uniform_norm}, all row sums and column sums of this matrix are at most $d_R$. By the Riesz--Thorin interpolation theorem~\cite[Theorem~E.21]{keller2021graphs}, which yields the bound $\rho^{1-1/q}\gamma^{1/q}$ for the $\ell^q$ operator norm of a matrix, where $\rho$ and $\gamma$ are the largest row sum and the largest column sum of the absolute values of its entries, respectively, we have
\[
\norm{A_\bz}\le d_R,
\qquad
\norm{D_\bz}\le d_R
\]
for every $q\in[1,\infty]$; the bound for the diagonal operator $D_\bz$ is immediate, since $0\le\operatorname{deg}_\bz(i)\le d_R$. Since $\Delta_\bz=D_\bz-A_\bz$, \cref{eq:Delta-uniform-bound} follows.

For the Lipschitz estimate, set
\[
g_{B_{\rm c}}(t)\coloneqq
\exp\!\left(-\frac{\min\{t,B_{\rm c}\}^2}{\sigma^2}\right),
\qquad t\ge0.
\]
The function $g_{B_{\rm c}}$ is globally Lipschitz since the function $g(t) = \exp(-\frac{t^2}{\sigma^2}), \ t\ge0$ has bounded derivative. Furthermore, since $|g'(t)|=\frac{2t}{\sigma^2}\exp(-\frac{t^2}{\sigma^2})$ attains its maximum, equal to $\frac1\sigma\sqrt{2/\mathrm{e}}$, at $t=\frac{\sigma}{\sqrt{2}}$ and since the function $t\mapsto\min\{t,B_{\rm c}\}$ is 1-Lipschitz, we have that
\[
\operatorname{Lip}(g_{B_{\rm c}})
\le\frac1\sigma\sqrt{\frac{2}{\mathrm{e}}}.
\]
For every edge $(i,j)$,
\[
    \omega_{\bz}(i,j) = \exp\left(-\frac{\min\{|\bz(i)-\bz(j)|,B_{\rm c}\}^2}{\sigma^2}\right)= g_{B_{\rm c}}(|\bz(i)-\bz(j)|).
\]
Hence, we have that
\[
\begin{aligned}
|\omega_\bx(i,j)-\omega_\bz(i,j)|
&\le \frac1\sigma\sqrt{\frac{2}{\mathrm{e}}}
\bigl||\bx(i)-\bx(j)|-|\bz(i)-\bz(j)|\bigr|  \\
&\le \frac1\sigma\sqrt{\frac{2}{\mathrm{e}}}
\bigl(|\bx(i)-\bz(i)|+|\bx(j)-\bz(j)|\bigr) \\
&\le \frac{2}{\sigma}\sqrt{\frac{2}{\mathrm{e}}}\,\norm{\bx-\bz}_\infty.
\end{aligned}
\]
Finally, write $\Delta_\bx-\Delta_\bz=(D_\bx-D_\bz)-(A_\bx-A_\bz)$. The matrix of $A_\bx-A_\bz$ has entries $\omega_\bx(i,j)-\omega_\bz(i,j)$; every row and every column has at most $d_R$ nonzero entries, each bounded in absolute value by $\frac{2}{\sigma}\sqrt{2/\mathrm{e}}\,\norm{\bx-\bz}_\infty$, by the above inequality. By the Riesz--Thorin interpolation theorem~\cite[Theorem~E.21]{keller2021graphs}, applied to the row and column sums of the absolute values of the entries, we therefore obtain, for every $q\in[1,\infty]$,
\[
\norm{A_\bx-A_\bz}\le \frac{2d_R}{\sigma}\sqrt{\frac{2}{\mathrm{e}}}\,\norm{\bx-\bz}_\infty,
\qquad
\norm{D_\bx-D_\bz}\le \frac{2d_R}{\sigma}\sqrt{\frac{2}{\mathrm{e}}}\,\norm{\bx-\bz}_\infty,
\]
where the second bound holds because $D_\bx-D_\bz$ is the diagonal operator whose entries are the row sums of the matrix of $A_\bx-A_\bz$. Adding the two bounds gives \cref{eq:Delta-Lipschitz}.
\end{proof}

\begin{corollary}\label{cor:R-bounds}
For every $q\in[1,2]$,
\begin{equation}\label{eq:R-uniform-bound}
0\le \Rreg_\bz^q(\bu)\le \frac{(2d_R)^q}{q}\norm{\bu}_q^q .
\end{equation}
Furthermore, there exists a constant
\[
C_{R,q}\coloneqq2(2d_R)^{q-1}L_\Delta
\]
independent of $N$ such that, for all $\bx,\bz,\bu\in X$,
\begin{equation}\label{eq:R-Lipschitz-guide}
\bigl|\Rreg_\bx^q(\bu)-\Rreg_\bz^q(\bu)\bigr|
\le C_{R,q}\,\norm{\bx-\bz}_\infty\,\norm{\bu}_q^q .
\end{equation}
Consequently, if $\bz_m\to \bz$ in $\norm{\cdot}_\infty$ and $\bu_m\to \bu$ in $X$, then
\begin{equation}\label{eq:R-continuity}
\Rreg_{\bz_m}^q(\bu_m)\to \Rreg_\bz^q(\bu).
\end{equation}
\end{corollary}

\begin{proof}
Estimate \cref{eq:R-uniform-bound} follows immediately from \cref{eq:Delta-uniform-bound}.

Let $\ba\coloneqq\Delta_\bx \bu$ and $\bb\coloneqq\Delta_\bz \bu$. Let $q\in[1,2]$ and $s,t\geq 0$, and consider the power function $\tau\mapsto\tau^q$. By the mean value theorem, there exists $\xi$ between $s$ and $t$ such that $|s^q-t^q|=q\xi^{q-1}|s-t|$.
Consequently,
\[
    \frac{1}{q}|s^q-t^q|
    \leq \max\{s,t\}^{q-1}|s-t|
    \leq (s^{q-1}+t^{q-1})|s-t|.
\]
Setting $s=\|\ba\|_q$ and $t=\|\bb\|_q$ and applying the reverse triangle inequality, we obtain
\[
\frac1q\bigl|\norm{\ba}_q^q-\norm{\bb}_q^q\bigr|
\le \bigl(\norm{\ba}_q^{q-1}+\norm{\bb}_q^{q-1}\bigr)
\norm{\ba-\bb}_q.
\]
 Using \Cref{lem:graph-bounds},
\[
\norm{\ba}_q^{q-1}+\norm{\bb}_q^{q-1}
\le 2(2d_R)^{q-1}\norm{\bu}_q^{q-1},
\qquad
\norm{\ba-\bb}_q\le L_\Delta\norm{\bx-\bz}_\infty\norm{\bu}_q,
\]
and \cref{eq:R-Lipschitz-guide} follows.

For \cref{eq:R-continuity}, write
\[
\begin{aligned}
|\Rreg_{\bz_m}^q(\bu_m)-\Rreg_\bz^q(\bu)|
&\le |\Rreg_{\bz_m}^q(\bu_m)-\Rreg_\bz^q(\bu_m)|
+|\Rreg_\bz^q(\bu_m)-\Rreg_\bz^q(\bu)|.
\end{aligned}
\]
The first term tends to zero by \cref{eq:R-Lipschitz-guide}, since $(\|\bu_m\|_q)$ is bounded. The second term tends to zero by the continuity (uniform in $N$) of the fixed linear map $\Delta_\bz$ (see~\cref{eq:Delta-uniform-bound}) and of the map $\bv\mapsto\norm{\bv}_q^q$.
\end{proof}

\section{Single iteration step and well-posedness}\label{sec:preliminary_estimates}
This section sets up the regularization framework and the main hypotheses before the iterative schemes are introduced. All assumptions required for each statement are explicitly stated, and no standing assumptions are imposed. We then refine several stability estimates for the fixed-reference, or ``single-step'', variational problem involving the graph Laplacian, first established in~\cite{bianchi2025data}. These facts will be used repeatedly in the convergence analysis.

At discretization level $N$, given the model operator $K$ in~\cref{eq:model_operator}, we recall the model equation~\cref{eq:Model_Equation}
\begin{equation*}
K\bx=\by^\delta,
\end{equation*}
where $\by^\delta$ is a noisy perturbation of the exact data $\by$ such that
\begin{equation*}
\norm{\by^\delta-\by}_Y\le \delta.
\end{equation*}

We remark that each object described here, like for instance the forward operator $K$
and the norm in $X$ of elements $\bx$, may depend on the discretization level $N$. To avoid
unnecessarily heavy notation, we suppress the index $N$. When we
say that a property holds uniformly with respect to the discretization level,
we mean that it holds for every $N$.

\begin{hypothesis}[Existence of solutions]\label{hyp:consistency}
The exact data are generated by applying the forward
operator $K$ to the ground truth, i.e., $\by=K\gt$, and the ground truth is uniformly bounded across discretization levels: there is a constant $M_{\rm gt}>0$, independent of $N$, such that
\[
K\gt=\by,
\qquad
\norm{\gt}_2\le M_{\rm gt} .
\]
\end{hypothesis}

By Jensen's inequality, it is easy to check that the normalized norms defined as
\[
\|\bx\|_q = \left(\frac{1}{N}\sum_{i\in P}|\bx(i)|^q\right)^{\frac{1}{q}}
\]
are nondecreasing in the exponent, that is, $\norm{\bx}_q \leq \norm{\bx}_p$ for $1\leq q\leq p$; it follows that
$\norm{\gt}_q\le M_{\rm gt}$ for every $q\in[1,2]$.
The restriction $q\le2$ is used here and in the proof of \Cref{thm:error-noisy-convergence} to obtain $\ell^q$ bounds from $\ell^2$ bounds uniformly in $N$.

In particular, for every discretization level $N$, the solution set
\[
\Ssol(\by)\coloneqq\{\bx\in X:K\bx=\by\}
\]
is nonempty since it contains $\gt$, which is uniformly bounded: its normalized $\ell^2$ norm is bounded by the constant $M_{\rm gt}$, independent of $N$.
For images in a fixed grayscale range, this bound holds automatically in the
normalized norms.

\begin{hypothesis}[Uniform joint coercivity]\label{hyp:joint-coercivity}
There exist constants $C_J>0$ and $M_K>0$, independent of $N$ and of the reference image $\bz$, such that
\begin{equation*}
\norm{K\bu}_Y\le M_K\norm{\bu}_2
\qquad\forall \bu\in X,
\end{equation*}
and
\begin{equation}\label{eq:joint-coercivity}
\norm{\bu}_2
\le C_J\Bigl(\norm{K\bu}_Y+\norm{\Delta_\bz \bu}_q\Bigr)
\qquad\forall \bu,\bz\in X .
\end{equation}
\end{hypothesis}

Condition~\cref{eq:joint-coercivity} is the graph Laplacian $\ell^q$ analogue of the classical \emph{completion condition} of Morozov~\cite[equation~(3), p.~1, and Theorem~1(c), p.~3]{morozov2012methods}, also called \emph{complementation condition} in~\cite[Chapter~8, equation~(8.3)]{engl1996regularization}. In the Hilbert-space case $q=2$, it is usually stated in the equivalent squared
form: there exists $\gamma>0$ such that
\[
\norm{K\bu}_Y^2+\norm{\Delta_\bz\bu}_2^2 \;\geq\; \gamma\,\norm{\bu}_2^2
\qquad\forall\,\bu\in X .
\]
The two formulations, for $q=2$, are equivalent, since $\tfrac{1}{\sqrt2}(a+b)\leq\sqrt{a^2+b^2}
\leq a+b$ for $a,b\geq0$: one passes from the squared form to
\cref{eq:joint-coercivity} with $C_J=\gamma^{-1/2}$, and back with
$\gamma=(2C_J^2)^{-1}$.

A necessary requirement is that
the forward operator and the regularizer share no nontrivial kernel,
\[
\ker K\cap\ker\Delta_\bz=\{\boldsymbol 0\}  .
\]
 For fixed $N$, $\bz$, and $q\geq 1$, this condition is also sufficient. On the compact unit sphere $\{\bu \in X \colon \norm{\bu}_2=1\}$, the continuous function
\[
f(\bu)=\norm{K\bu}_Y+\norm{\Delta_\bz\bu}_q
\]
has a minimum $c_{N,\bz,q}$. If $\ker K\cap\ker\Delta_\bz=\{\boldsymbol 0\}$, then $c_{N,\bz,q}>0$. Homogeneity of $f$ and rescaling give
\[
\norm{\bu}_2\le C_{N,\bz,q}
\bigl(\norm{K\bu}_Y+\norm{\Delta_\bz\bu}_q\bigr)
\quad\forall \bu\in X, \qquad C_{N,\bz,q}\coloneqq c_{N,\bz,q}^{-1}.
\]
What \Cref{hyp:joint-coercivity} demands
beyond this pointwise statement is that the bound holds \emph{uniformly} in the
discretization level $N$ and in the reference image $\bz$, that is, with a single constant $C_J$ in place of the constants $C_{N,\bz,q}$. This uniformity is what yields the dimension-free constants in
the convergence analysis.
The coercivity condition in \Cref{hyp:joint-coercivity} is then verified in \Cref{sec:appendix-coercivity} for some representative model operators~$K$.

Now we introduce the definition of solutions for the exact equation $K\bx=\by$ underlying~\cref{eq:Model_Equation} with respect to the graph regularizer. It is standard in variational problems; see, for example, \cite[Definition~3.24]{scherzer2009variational}.

\begin{definition}[Regularizing solutions]\label{def:R-min-solution}
Fix $\bz\in X$ and define $\Rreg_\bz^q$ as per~\cref{def:graph-regularizer}. A point $\bx^\star\in \Ssol(\by)$ is called an $\Rreg_\bz^q$-minimizing solution if
\begin{equation*}
\Rreg_\bz^q(\bx^\star)
=\min\{\Rreg_\bz^q(\bs):\bs\in\Ssol(\by)\}.
\end{equation*}
\end{definition}

For the standard iterative scheme below (\Cref{alg:standard-iter}), the possible
limiting solutions belong to the collection of minimizers obtained as the
reference image ranges over the exact solutions. We make this precise.

\begin{definition}[Graph-minimizing solutions]\label{def:Sgm}
The set of \emph{graph-minimizing solutions} of the exact equation $K\bx=\by$ is
\begin{equation}\label{eq:Sgm-def}
\mathcal S_{\rm gm}(\by)
\coloneqq
\bigcup_{\bz\in\Ssol(\by)}
\operatorname*{argmin}_{\bs\in\Ssol(\by)}\Rreg_\bz^q(\bs),
\end{equation}
that is, $\bu\in\mathcal S_{\rm gm}(\by)$ if and only if $\bu$ is an
$\Rreg_\bz^q$-minimizing solution in the sense of \Cref{def:R-min-solution} for
some reference $\bz\in\Ssol(\by)$.
\end{definition}

\begin{remark}\label{rem:Sgm}
By construction, $\mathcal S_{\rm gm}(\by)\subseteq\Ssol(\by)$.
Under \Cref{hyp:consistency,hyp:joint-coercivity},
the minimum in \Cref{def:R-min-solution} is attained for every fixed
$\bz\in X$. Indeed, by \Cref{hyp:consistency}, $\Ssol(\by)$ is nonempty.
Fix $\bz\in X$ and choose any $\bs_0\in\Ssol(\by)$. Consider the set
\[
    A\coloneqq\left\{
    \bs\in X\colon K\bs=\by,\;
    \Rreg_\bz^q(\bs)\leq\Rreg_\bz^q(\bs_0)
    \right\}.
\]
Clearly, $A$ is nonempty since $\bs_0\in A$. Moreover, $A$ is closed
because $K$ and $\Rreg_\bz^q$ are continuous.
For every $\bs\in A$, we have $\|K\bs\|_Y=\|\by\|_Y$ and
$\|\Delta_\bz\bs\|_q=(q\Rreg_\bz^q(\bs))^{1/q}
\leq(q\Rreg_\bz^q(\bs_0))^{1/q}$.
Therefore, by the uniform joint coercivity condition
\cref{eq:joint-coercivity}, every $\bs\in A$ satisfies
$\|\bs\|_2\leq C_J\bigl(\|\by\|_Y+
(q\Rreg_\bz^q(\bs_0))^{1/q}\bigr)$.
Hence, $A$ is bounded.
As $X$ is a finite-dimensional normed space, $A$ is therefore compact.
By continuity, $\Rreg_\bz^q$ attains its minimum on $A$.
Finally, every $\bs\in\Ssol(\by)\setminus A$ satisfies
$\Rreg_\bz^q(\bs)>\Rreg_\bz^q(\bs_0)$, while the minimum
over $A$ is at most $\Rreg_\bz^q(\bs_0)$.
Consequently, the minimum over $A$ coincides with the minimum
over $\Ssol(\by)$, proving that the corresponding
$\operatorname*{argmin}$ is nonempty.
In particular, taking $\bz\in\Ssol(\by)$ shows that $\mathcal S_{\rm gm}(\by)$ is nonempty.
The set reduces to a single point when $K$ is injective, since then
$\Ssol(\by)=\{\gt\}$.
\end{remark}

\subsection{Single-step variational problem}\label{ssec:fixed-graph}
With the framework and hypotheses in place, we now study the single-step
problem~\cref{eq:graphLa+}, in which the reference image $\bz$ is held fixed.
This is the building block of all three iterative schemes. Each iteration updates
the reference $\bz$ (through the recomputed graph Laplacian), and the
error-equation and mixed schemes update the data $\by$ as well (through the
residual). We therefore need the single step to be well defined and to depend stably on both
$\by$ and $\bz$, with constants independent of the discretization level $N$.

\Cref{prop:existence} below establishes that the fixed-reference functional is convex
and coercive, with sublevel sets bounded uniformly in $N$ and $\bz$; a minimizer
always exists, and is unique when $q>1$. \Cref{cor:one-step-stability} then establishes
stability with respect to the data $\by$ and the reference image $\bz$, with
continuous dependence when $q>1$ and dimension-independent bounds on the
minimizers in terms of the data bounds. Both facts are used repeatedly in the convergence
analysis in \Cref{sec:convergence_standard,sec:convergence_error}.

For $q\in[1,2]$, $\alpha>0$, a reference image $\bz\in X$, and data $\by\in Y$,
define
\begin{subequations}\label{eq:fixed-graph-functional}
\begin{align}
\Phi_{\alpha,\bz}^q(\bx;\by)
&\coloneqq\frac12\norm{K\bx-\by}_Y^2+\alpha\Rreg_\bz^q(\bx),
\qquad \bx\in X ,
\label{eq:fixed-graph-functional-a}\\[6pt]
T^q_{\alpha,\bz}(\by)
&\coloneqq\operatorname*{argmin}_{\bx\in X}\;\Phi_{\alpha,\bz}^q(\bx;\by).
\label{eq:fixed-graph-functional-b}
\end{align}
\end{subequations}
This notation will be used with either exact data $\by$ or noisy data
$\by^\delta$.

\begin{proposition}\label{prop:existence}
Assume \Cref{hyp:joint-coercivity}. For every $N$, $\bz\in X$,
$\by\in Y$, $\alpha>0$, and $q\in[1,2]$, the functional
$\Phi_{\alpha,\bz}^q(\cdot;\by)$ is convex and coercive. More precisely, if
\[
\Phi_{\alpha,\bz}^q(\bx;\by)\le M,
\]
then
\begin{equation}\label{eq:sublevel-bound}
\norm{\bx}_2
\le C_J\left((2M)^{1/2}+\norm{\by}_Y
+\left(\frac{qM}{\alpha}\right)^{1/q}\right).
\end{equation}
Consequently, $\Phi_{\alpha,\bz}^q(\cdot;\by)$ admits at least one minimizer.
If $q>1$, this minimizer is unique.
\end{proposition}

\begin{proof}
Convexity follows from the convexity of the data term $\bx \mapsto \frac12\norm{K\bx - \by}_Y^2$ and of the regularization term
$\bx\mapsto\norm{\Delta_\bz \bx}_q^q$. Assume $\Phi_{\alpha,\bz}^q(\bx;\by)\le M$. Then
\[
\norm{K\bx-\by}_Y\le (2M)^{1/2},
\qquad
\norm{\Delta_\bz \bx}_q\le \left(\frac{qM}{\alpha}\right)^{1/q}.
\]
Hence $\norm{K\bx}_Y\le (2M)^{1/2}+\norm{\by}_Y$. Applying
\cref{eq:joint-coercivity} gives \cref{eq:sublevel-bound}. Thus the sublevel
sets of the functional $\Phi_{\alpha,\bz}^q(\cdot;\by)$ are bounded with a bound independent of $N$ and $\bz$. In particular, the sublevel set $L_M\coloneqq\{\bx\in X:\Phi_{\alpha,\bz}^q(\bx;\by)\le M\}$ with $M=\Phi_{\alpha,\bz}^q(0;\by)$ is nonempty, since it contains $\bx=0$; it is bounded, by \cref{eq:sublevel-bound}; and it is closed, by the continuity of the functional. Hence it is compact, since $X$ is finite-dimensional, and the functional attains its minimum over $L_M$, which is also its minimum over $X$. Alternatively, existence follows from convexity, continuity and coercivity; see \cite[Proposition 11.15]{bauschke2017convex}.

For uniqueness when $q>1$, let $\bx_1\ne\bx_2$ and set
$\bu\coloneqq\bx_1-\bx_2$. By \cref{eq:joint-coercivity}, $K\bu$ and
$\Delta_\bz\bu$ cannot both vanish. If $K\bu\ne0$, the data term is strictly
convex on the segment joining $\bx_1$ and $\bx_2$. If $\Delta_\bz\bu\ne0$, then
strict convexity of $\bv\mapsto\norm{\bv}_q^q$ for $q>1$ implies that the regularization term is
strictly convex on that segment. Therefore
$\Phi_{\alpha,\bz}^q(\cdot;\by)$ is strictly convex along every nontrivial
segment. Uniqueness then follows.
\end{proof}

\begin{corollary}
\label{cor:one-step-stability}
Assume \Cref{hyp:joint-coercivity}, and fix $q\in[1,2]$ and $\alpha>0$,
with $\Phi^q_{\alpha,\bz}$ and $T^q_{\alpha,\bz}$ as in~\cref{eq:fixed-graph-functional}. Let
\[
\by_m\to\by \;\text{ in } Y,
\qquad
\norm{\bz_m-\bz}_\infty\to 0,
\]
and choose any
\[
\bx_m\in T^q_{\alpha,\bz_m}(\by_m).
\]
Then the following hold.

\begin{enumerate}[(i),leftmargin=*]
\item The family $(\bx_m)$ is bounded, with a bound independent of $N$ whenever
$\sup_m\norm{\by_m}_Y$ is bounded independently of $N$; every cluster point
of $(\bx_m)$ is a minimizer of $\Phi_{\alpha,\bz}^q(\cdot;\by)$, and consequently
$(\bx_m)$ has a subsequence converging in $X$ to such a minimizer.
\item If $q>1$, the minimizer is unique, so the operator in
\cref{eq:fixed-graph-functional-b} is single-valued, $\bx_m=T^q_{\alpha,\bz_m}(\by_m)$,
and the whole sequence converges,
\[
\bx_m\to T^q_{\alpha,\bz}(\by)\quad\text{in }X.
\]
\end{enumerate}
\end{corollary}
\begin{proof}
By minimality, testing against $\bx=0$ in~\cref{eq:fixed-graph-functional-a} gives
\[
\Phi_{\alpha,\bz_m}^q(\bx_m;\by_m)
\le \Phi_{\alpha,\bz_m}^q(0;\by_m)
=\tfrac12\norm{\by_m}_Y^2 .
\]
Since $(\by_m)$ converges, it is bounded, so the right-hand side is bounded
uniformly in $m$, and \Cref{prop:existence} yields a bound on
$(\bx_m)$ depending only on this data bound, $\alpha$, $q$, and $C_J$.
In particular, it is independent of $N$ whenever the data bound is.
As $X$ is finite-dimensional,
$(\bx_m)$ therefore admits cluster points. Let $\bx_{m_k}\to\bar\bx$ in $X$.
For every fixed $\bx\in X$, minimality of $\bx_{m_k}$ gives
\[
\Phi_{\alpha,\bz_{m_k}}^q(\bx_{m_k};\by_{m_k})
\le \Phi_{\alpha,\bz_{m_k}}^q(\bx;\by_{m_k}).
\]
The data-fidelity term in~\cref{eq:fixed-graph-functional-a} is continuous in
$(\bx,\by)$ by the bound $\|K\bu\|_Y\le M_K\|\bu\|_2$ of \Cref{hyp:joint-coercivity}, and by~\cref{eq:R-continuity} the regularization term is
jointly continuous in the reference image and in its argument, so that $\Rreg_{\bz_{m_k}}^q(\bx_{m_k})\to\Rreg_\bz^q(\bar\bx)$ and $\Rreg_{\bz_{m_k}}^q(\bx)\to\Rreg_\bz^q(\bx)$; passing to the limit along $k$ on both
sides (using $\by_{m_k}\to\by$ and $\norm{\bz_{m_k}-\bz}_\infty\to0$) yields
\[
\Phi_{\alpha,\bz}^q(\bar\bx;\by)
\le \Phi_{\alpha,\bz}^q(\bx;\by)
\qquad\forall\,\bx\in X .
\]
Hence $\bar\bx$ minimizes $\Phi_{\alpha,\bz}^q(\cdot;\by)$, which proves~(i).

If $q>1$, \Cref{prop:existence} shows each functional is strictly
convex, hence has a unique minimizer; the $\operatorname*{argmin}$
in~\cref{eq:fixed-graph-functional-b} is then single-valued, so
$\bx_m=T^q_{\alpha,\bz_m}(\by_m)$ and the limit problem has the unique minimizer
$T^q_{\alpha,\bz}(\by)$. By~(i), every subsequence of $(\bx_m)$ has a further
subsequence converging to this same limit, so by standard topological arguments the whole sequence converges; see, for example,~\cite[Lemma~8.2]{scherzer2009variational}.
This proves~(ii).
\end{proof}

\section{Iterated graph Laplacian regularization schemes}\label{sec:iterschemes}
In this section, we formally introduce the new iterative schemes based on the graph Laplacian defined in \Cref{sec:graph_theory}. These algorithms include the standard graph-based update, the error-equation strategy, and the mixed iterative method combining the previous two.

Since $T^q_{\alpha,\bz}$ is set-valued when minimizers are not unique, each time we write $\bu\in T^q_{\alpha,\bz}(\by)$ we mean that one admissible minimizer is selected.

\subsection{Standard iterative scheme}
Given a preliminary reconstruction $\bx_0^\delta \in X$ and a positive sequence $\{\alpha_n\}_{n\ge1}$, we define the iterative scheme
\begin{equation}\label{eq:standard-iter}
\bx_n^\delta \in T^q_{\alpha_n,\bx_{n-1}^\delta}(\by^\delta),
\qquad n \geq 1.
\end{equation}
Equivalently, at each iteration $\bx_n^\delta$ minimizes
\[
\frac12\norm{K\bx-\by^\delta}_Y^2+\alpha_n\Rreg_{\bx_{n-1}^\delta}^q(\bx)
\qquad\text{over }\bx\in X.
\]
Thus, the graph Laplacian operator is recomputed from the previous iterate $\bx_{n-1}^\delta$, and the regularization functional varies with $n$. The proposed strategy is described in \Cref{alg:standard-iter}.

\begin{algorithm}[h]
\caption{Standard iterative scheme}\label{alg:standard-iter}
\begin{algorithmic}[1]
\REQUIRE $K$, $\by^\delta$, $\bx_0^\delta$, $\{\alpha_n\}_{n\ge1}$
\FOR{$n = 1,2,\ldots$}
\STATE $\bx_n^\delta \in T^q_{\alpha_n,\bx_{n-1}^\delta}(\by^\delta)$
\ENDFOR
\end{algorithmic}
\end{algorithm}

The rationale behind this strategy is as follows. The graph Laplacian can effectively encode prior information about the ground truth solution provided that the reconstruction used to build it is a sufficiently accurate approximation of $\gt$. Even when the initial reconstruction $\bx_0^\delta$ is a poor approximation, the first iterate $\bx_1^\delta$ is expected to improve upon it. Constructing the graph Laplacian using $\bx_1^\delta$ may therefore yield a linear operator that is better suited as a regularizer. As the iterations proceed, the graph Laplacian can incorporate more accurate structural information about the solution, potentially improving the final reconstruction.

\subsection{Error-equation iterative scheme}\label{sec:erreq}
A second approach is based on iterative refinement through the error equation.
For the unknown initial error $\be_0^\delta\coloneqq\gt-\bx_0^\delta$, the noise model gives
\[
\br_0^\delta\coloneqq\by^\delta-K\bx_0^\delta
=K\be_0^\delta+\bm\eta.
\]
Thus the residual is noisy data for this error. We estimate the error by solving the regularized problem
\[
\underset{\bh\in X}{\operatorname{argmin}}
\left\{
\tfrac12\norm{K\bh-\br_0^\delta}_Y^2
+\alpha_1\Rreg_{\bh_0^\delta}^q(\bh)
\right\},
\]
where $\bh_0^\delta\in X$ is a prescribed image-space approximation of the first error, used only to construct the graph Laplacian in the first error step. 

We build the graph from an error estimate to encode still missing features rather than the structure of the whole image. After initialization, each correction serves as the reference image for the next step. The error-equation iteration therefore updates both the reconstruction and the error graph input. With
\[
\br_n^\delta\coloneqq\by^\delta-K\bx_n^\delta,
\]
the scheme is
\[
\left\{
\begin{aligned}
\bh_{n+1}^\delta
&\in T^q_{\alpha_{n+1},\bh_n^\delta}(\br_n^\delta),\\
\bx_{n+1}^\delta
&=\bx_n^\delta+\bh_{n+1}^\delta,\\
\br_{n+1}^\delta
&=\br_n^\delta-K\bh_{n+1}^\delta,
\end{aligned}
\right.
\qquad n\ge0.
\]
The algorithm is summarized in \Cref{alg:error-iter}.

\begin{algorithm}[h!]
\caption{Error-equation iterative scheme}\label{alg:error-iter}
\begin{algorithmic}[1]
\REQUIRE $K$, $\by^\delta$, $\bx_0^\delta$, $\bh_0^\delta$, $\{\alpha_n\}_{n\ge1}$
\STATE $\br_0^\delta=\by^\delta-K\bx_0^\delta$
\FOR{$n = 0,1,2,\ldots$}
\STATE $\bh_{n+1}^\delta \in T^q_{\alpha_{n+1},\bh_n^\delta}(\br_n^\delta)$
\STATE $\bx_{n+1}^\delta=\bx_n^\delta+\bh_{n+1}^\delta$
\STATE $\br_{n+1}^\delta=\br_n^\delta-K\bh_{n+1}^\delta$
\ENDFOR
\end{algorithmic}
\end{algorithm}

\subsection{Mixed iterative scheme}\label{sec:mix}
The mixed iterative scheme combines the standard approach in~\Cref{alg:standard-iter} with the error-equation strategy in~\Cref{alg:error-iter}. The motivation for combining these two approaches lies in their complementary behavior. Empirically, after a few iterations of the standard scheme, we observe that the performance metrics reach a plateau, meaning that further iterations do not lead to any significant improvement in the reconstruction quality. This phenomenon can be interpreted as an indication that the graph Laplacian, constructed from the current iterate, is no longer able to extract additional structural information about the ground truth solution $\gt$. Once this plateau is reached, switching to the error-based formulation can further improve the reconstruction by recovering additional details. 

\begin{algorithm}[t]
\caption{Mixed iterative scheme}\label{alg:mixed-iter}
\begin{algorithmic}[1]
\REQUIRE $K$, $\by^\delta$, $\bx_0^\delta$, $\{\alpha_n\}_{n\ge1}$, $n_{\mathrm{std}}$
\STATE \textbf{Standard phase}
\FOR{$n = 1$ \TO $n_{\mathrm{std}}$}
\STATE $\bx_n^\delta \in T^q_{\alpha_n,\bx_{n-1}^\delta}(\by^\delta)$
\ENDFOR
\STATE $\br_{n_{\mathrm{std}}}^\delta=\by^\delta-K\bx_{n_{\mathrm{std}}}^\delta$
\STATE Choose an image-space initialization $\bg_0^\delta$
\STATE \textbf{Error phase}
\FOR{$k = 0,1,2,\ldots$}
\STATE $\bg_{k+1}^\delta\in T^q_{\alpha_{n_{\mathrm{std}}+k+1},\bg_k^\delta}(\br_{n_{\mathrm{std}}+k}^\delta)$
\STATE $\bx_{n_{\mathrm{std}}+k+1}^\delta=\bx_{n_{\mathrm{std}}+k}^\delta+\bg_{k+1}^\delta$
\STATE $\br_{n_{\mathrm{std}}+k+1}^\delta=\br_{n_{\mathrm{std}}+k}^\delta-K\bg_{k+1}^\delta$
\ENDFOR
\end{algorithmic}
\end{algorithm}

More precisely, fix the number of standard iterations
$n_{\mathrm{std}}\in\mathbb{N}_{>0}$. We first run $n_{\mathrm{std}}$ steps
of~\Cref{alg:standard-iter}, producing $\bx_{n_{\mathrm{std}}}^\delta$, and then
continue with the error-equation iteration of~\Cref{alg:error-iter} initialized at
this point. We set the residual
\[
\br_{n_{\mathrm{std}}}^\delta\coloneqq\by^\delta-K\bx_{n_{\mathrm{std}}}^\delta ,
\]
and choose an initial error-stage graph input $\bg_0^\delta\in X$. In \Cref{alg:mixed-iter}, the error-phase corrections, which also serve as graph inputs for the following step, are denoted by $\bg_k^\delta$, $k\ge1$, instead of $\bh_k^\delta$, to emphasize that this phase is initialized by the graph input $\bg_0^\delta$ and by the residual $\br_{n_{\mathrm{std}}}^\delta$ of the standard phase rather than by $\bh_0^\delta$ and $\br_0^\delta$.
Since the choice of $\bg_0^\delta$ is problem-dependent, in the numerical experiments we use the
following strategies:
\begin{itemize}
    \item \emph{Deblurring problem.} Since the residual belongs to the same image
    space as the reconstruction, we choose
    \begin{equation}\label{eq:g0res}
    \bg_0^\delta=\br_{n_{\mathrm{std}}}^\delta
    =\by^\delta-K\bx_{n_{\mathrm{std}}}^\delta .
    \end{equation}
    \item \emph{CT problem.} We apply filtered backprojection (FBP) to the
    residual $\br_{n_{\mathrm{std}}}^\delta$ to obtain an image-space approximation
    of the error.
    \item \emph{General approach.} A problem-independent alternative computes the
    initialization via Tikhonov regularization,
    \[
    \bg_0^\delta \in
    \underset{\bw \in X}{\operatorname{argmin}}
    \left\{
    \tfrac{1}{2}\norm{K\bw-\br_{n_{\mathrm{std}}}^\delta}_Y^2
    + \lambda \norm{\bw}_2^2
    \right\},
    \]
    where $\lambda>0$ can be selected by an automatic rule such as generalized
    cross-validation (GCV)~\cite{gcv}.
\end{itemize}
Finally, we perform a prescribed number of error-equation steps as
in~\Cref{alg:error-iter}, building the graph Laplacian from
$\bg_0^\delta$ at the first error step. Usually a single iteration
suffices to recover a few additional sharp edges. The overall procedure is summarized in~\Cref{alg:mixed-iter}.

\section{Regularization properties of \texorpdfstring{\Cref{alg:standard-iter}}{the standard iterative scheme}}\label{sec:convergence_standard}
In this section, we present the theoretical results concerning the standard iterative scheme.

\subsection{Finite-iteration stability}
We now show that the one-step stability of \Cref{cor:one-step-stability}
propagates, by induction, through any fixed number of iterations of the standard
scheme: if the data and the initial reconstruction converge, the first $J$
iterates converge along subsequences to a trajectory of the limiting scheme.
When the minimizer is unique at every limiting step, the whole sequence
converges. The bounds on the iterates are independent of the discretization
level $N$ provided that the data bounds are independent of $N$.
In the statement, $m$ indexes the sequence of perturbed inputs $(\by_m,\bx_{0,m})$ and is the index that tends to infinity, whereas $n\in\{1,\ldots,J\}$ labels the iterates generated from each input and ranges over a fixed finite set.
\begin{theorem}\label{thm:finite-stability}
Assume \Cref{hyp:joint-coercivity}. Fix $J\in\N$ and positive parameters $\alpha_1,\ldots,\alpha_J$. Let
\[
\by_m\to \by \quad\text{in }Y,
\qquad
\bx_{0,m}\to \bx_0 \quad\text{in }\norm{\cdot}_\infty.
\]
For each $m$, let $\bx_{n,m}$ be generated by \cref{eq:standard-iter} with data $\by_m$ and initial point $\bx_{0,m}$, for $n=1,\ldots,J$.

Then the family of all such iterates is bounded, with bounds independent of $N$
whenever $\sup_m\norm{\by_m}_Y$ is bounded independently of $N$. Moreover, every subsequence indexed by $(m_k)$ has a further subsequence indexed by $(m_{k_j})$ such that, for each $n=1,\ldots,J$,
\[
\bx_{n,m_{k_j}}\to \bar \bx_n,
\]
where $(\bar \bx_n)_{n=1}^J$ is a sequence generated by the exact limiting data $\by$ and initial point $\bx_0$, namely
\[
\bar \bx_n\in\operatorname*{argmin}_{\bx\in X}
\left\{
\frac12\norm{K\bx-\by}_Y^2
+\alpha_n\Rreg_{\bar \bx_{n-1}}^q(\bx)
\right\},
\qquad \bar \bx_0=\bx_0.
\]
If the minimizer is unique at every limiting step, in particular if $q>1$, then the whole sequence converges:
\[
\bx_{n,m}\to \bar \bx_n,
\qquad n=1,\ldots,J.
\]
\end{theorem}

\begin{proof}
We argue by induction on $n$. For $n=1$, the claim is exactly \Cref{cor:one-step-stability}, applied with data $\by_m\to\by$, reference images $\bz_m=\bx_{0,m}$, and limit reference $\bz=\bx_0$.

Now assume the claim holds for $n-1$, so along any convergent subsequence $\bx_{n-1,m_k}\to\bar\bx_{n-1}$. Apply \Cref{cor:one-step-stability} once more, this time with reference images $\bz_{m_k}=\bx_{n-1,m_k}$ and limit reference $\bz=\bar\bx_{n-1}$. It follows that the sequence $\bx_{n,m_k}$ is bounded, and every cluster point of it solves the $n$-th problem in the limit.

So along every subsequence, the cluster points are minimizers. This proves the stability statement. Moreover, if the limiting problem has a unique minimizer at each step, there is only one possible cluster point, so the whole sequence converges.
\end{proof}

\subsection{Convergence for noisy data}

As in any iterative regularization method, the iteration has to be stopped at an index that depends on the noise level. For data $\by^\delta$ with noise level $\delta$, we denote by $n(\delta)\ge1$ the number of steps of the scheme \cref{eq:standard-iter} that are performed, and we call it the \emph{stopping index}: the reconstruction returned by the method is $\bx_{n(\delta)}^\delta$. Together with the sequence $(\alpha_n)$, the map $\delta\mapsto n(\delta)$ constitutes the parameter choice of the scheme. It is an a priori choice, since it depends on $\delta$ only.

Classical regularization theory identifies accumulation points of regularized
solutions as regularizer-minimizing solutions; uniqueness of such a solution
yields whole-sequence convergence. Here the regularizer
$\Rreg_\bz^q$ carries a reference point $\bz$ that is itself produced by the
iteration. Therefore, distinct
subsequences may converge to distinct reference images $\bar\bx$, and each reference image
can select a different minimizer. The natural limit object is thus not a
point but the set $\mathcal S_{\rm gm}(\by)$ of graph-minimizing solutions
introduced in \Cref{def:Sgm}.

Accordingly, our main convergence statement is set-valued: the iterates stopped at $n=n(\delta)$
approach this set,
\[
\operatorname{dist}_2\bigl(\bx_{n(\delta)}^\delta,\mathcal S_{\rm gm}(\by)\bigr)
\to0
\qquad\text{as }\delta\downarrow0,
\]
where
\[
\operatorname{dist}_2(\bx,U)
\coloneqq
\inf_{\bu\in U}\norm{\bx-\bu}_2,
\]
with no uniqueness assumption and no source condition. It reduces to the usual
single-point convergence when $\mathcal S_{\rm gm}(\by)$ is a singleton,
for instance when $K$ is injective. 
In any case, the method returns the single reconstruction $\bx_{n(\delta)}^\delta$. The set $\mathcal S_{\rm gm}(\by)$ only describes where the reconstructions can accumulate as $\delta\downarrow0$.

Before proceeding further with the convergence result, we need a couple of preliminary estimates.

\medskip
\noindent\textit{A priori estimates.}
Assume \Cref{hyp:consistency,hyp:joint-coercivity}, and let
$\norm{\by^\delta-\by}_Y\le\delta$. Let $(\bx_n^\delta)$ be generated by
\cref{eq:standard-iter}. With $\gt\in\Ssol(\by)$ the bounded ground truth from
\Cref{hyp:consistency}, the inequality~\cref{eq:R-uniform-bound} gives
\begin{equation*}
\Rreg_{\bz}^q(\gt)\le C_0,
\qquad
C_0\coloneqq\frac{(2d_R)^q}{q}M_{\rm gt}^q ,
\qquad \bz\in X .
\end{equation*}
Therefore, by minimality of $\bx_n^\delta$, for every $n\ge1$,
\begin{equation*}
\begin{aligned}
\frac12\norm{K\bx_n^\delta-\by^\delta}_Y^2
+\alpha_n\Rreg_{\bx_{n-1}^\delta}^q(\bx_n^\delta)
&\le
\frac12\norm{K\gt-\by^\delta}_Y^2
+\alpha_n\Rreg_{\bx_{n-1}^\delta}^q(\gt)\\
&\le \frac12\delta^2+\alpha_n C_0 .
\end{aligned}
\end{equation*}
Consequently,
\begin{equation}\label{eq:noisy-res-bound}
\norm{K\bx_n^\delta-\by^\delta}_Y^2
\le \delta^2+2\alpha_n C_0,
\end{equation}
and
\begin{equation}\label{eq:noisy-reg-bound}
\Rreg_{\bx_{n-1}^\delta}^q(\bx_n^\delta)
\le C_0+\frac{\delta^2}{2\alpha_n}.
\end{equation}

\begin{theorem}
\label{thm:standard-vanishing-noise}
Assume \Cref{hyp:consistency,hyp:joint-coercivity}.
Let $(\alpha_n)$ be positive, nonincreasing, $\alpha_n\downarrow0$, and let the stopping
index $n(\delta)$ be chosen so that
\begin{equation}\label{eq:apriori-rule}
\alpha_{n(\delta)}\to0,
\qquad
\frac{\delta^2}{\alpha_{n(\delta)}}\to0
\qquad\text{as }\delta\downarrow0.
\end{equation}
Let $\norm{\by^\delta-\by}_Y\le\delta$, and let $(\bx_n^\delta)$ be generated by
\cref{eq:standard-iter} with data $\by^\delta$. Then
\begin{equation}\label{eq:apriori-residual-convergence}
\norm{K\bx_{n(\delta)}^\delta-\by}_Y\to0
\qquad\text{as }\delta\downarrow0.
\end{equation}
Moreover, the family $(\bx_{n(\delta)}^\delta)$ is bounded in $X$, with
bounds independent of $N$, and
\begin{equation}\label{eq:set-convergence-Sgm}
\operatorname{dist}_2
\bigl(\bx_{n(\delta)}^\delta,\mathcal S_{\rm gm}(\by)\bigr)
\to0
\qquad\text{as }\delta\downarrow0.
\end{equation}
In particular,
\begin{equation}\label{eq:set-convergence-S}
\operatorname{dist}_2
\bigl(\bx_{n(\delta)}^\delta,\Ssol(\by)\bigr)
\to0
\qquad\text{as }\delta\downarrow0.
\end{equation}
\end{theorem}

The sequence $(\alpha_n)$ can be chosen in several ways. For instance, for an explicit a priori choice, fix $\alpha_0>0$, $0<\rho<1$, and $0<\beta<2$.
Set $\alpha_n=\alpha_0\rho^n$. For $0<\delta^\beta\leq\alpha_1$, let
\[
  n(\delta)=\max\{n\in\N:n\geq1,\ \alpha_n\geq\delta^\beta\}.
\]
Then $\alpha_{n(\delta)+1}<\delta^\beta\leq\alpha_{n(\delta)}$, so
\[
  \delta^\beta\leq\alpha_{n(\delta)}<\frac{\delta^\beta}{\rho},
  \qquad
  \frac{\delta^2}{\alpha_{n(\delta)}}\leq\delta^{2-\beta}.
\]
Thus $\alpha_{n(\delta)}\to0$ and $\delta^2/\alpha_{n(\delta)}\to0$ as
$\delta\downarrow0$, as required by~\eqref{eq:apriori-rule}.

\begin{proof}
By \Cref{rem:Sgm} the set $\mathcal S_{\rm gm}(\by)$ is nonempty, so the distance
in \cref{eq:set-convergence-Sgm} is well defined.
With $\alpha_n>0$ for every $n$ and $(\alpha_n)$ nonincreasing, $\alpha_{n(\delta)}\to0$ forces $n(\delta)\to\infty$. Indeed, if $n(\delta_j)\le M$ along some sequence $\delta_j\downarrow0$, then $\alpha_{n(\delta_j)}\ge\alpha_M>0$ for all $j$, contradicting $\alpha_{n(\delta_j)}\to0$.

In particular, $n(\delta)\ge2$ for all sufficiently small $\delta$, so that the indices $n(\delta)-1$ and $n(\delta)-2$ used below are admissible. Moreover,
\[
\alpha_{n(\delta)-1}\to0,
\qquad
\frac{\delta^2}{\alpha_{n(\delta)-1}}
\le
\frac{\delta^2}{\alpha_{n(\delta)}}
\to0 ,
\]
where the inequality uses the monotonicity of $(\alpha_n)$.

By \cref{eq:noisy-res-bound}, applied with $n=n(\delta)$,
\[
\norm{K\bx_{n(\delta)}^\delta-\by}_Y
\le
\norm{K\bx_{n(\delta)}^\delta-\by^\delta}_Y
+\norm{\by^\delta-\by}_Y
\le
\bigl(\delta^2+2\alpha_{n(\delta)}C_0\bigr)^{1/2}
+\delta .
\]
Therefore
\[
\norm{K\bx_{n(\delta)}^\delta-\by}_Y\to0 ,
\]
which proves \cref{eq:apriori-residual-convergence}. The same argument with
$n=n(\delta)-1$ gives
\begin{equation}\label{eq:predecessor-residual-convergence}
\norm{K\bx_{n(\delta)-1}^\delta-\by}_Y\to0 .
\end{equation}

Next, \cref{eq:noisy-reg-bound}, used with $n=n(\delta)$ and $n=n(\delta)-1$,
gives uniform bounds on
\[
\Rreg_{\bx_{n(\delta)-1}^\delta}^q(\bx_{n(\delta)}^\delta),
\qquad
\Rreg_{\bx_{n(\delta)-2}^\delta}^q(\bx_{n(\delta)-1}^\delta).
\]
Together with \cref{eq:apriori-residual-convergence},
\cref{eq:predecessor-residual-convergence}, and the joint coercivity condition
\cref{eq:joint-coercivity}, this gives uniform bounds on
$\bx_{n(\delta)}^\delta$ and $\bx_{n(\delta)-1}^\delta$, with constants
independent of $N$.

We now prove \cref{eq:set-convergence-Sgm}. Suppose, by contradiction, that it
fails. Then there exist $\varepsilon>0$ and a sequence $\delta_j\downarrow0$
such that, with $n_j\coloneqq n(\delta_j)$,
\begin{equation}\label{eq:dist-contradiction}
\operatorname{dist}_2
\bigl(\bx_{n_j}^{\delta_j},\mathcal S_{\rm gm}(\by)\bigr)
\ge\varepsilon
\qquad\forall j .
\end{equation}
By the uniform bounds just proved, finite-dimensional compactness gives, after
passing to a subsequence,
\[
\bx_{n_j}^{\delta_j}\to\bx^\star,
\qquad
\bx_{n_j-1}^{\delta_j}\to\bar\bx .
\]
The residual convergence gives
\[
\bx^\star,\bar\bx\in\Ssol(\by).
\]

It remains to identify $\bx^\star$. Let $\bs\in\Ssol(\by)$ be arbitrary. By
minimality of $\bx_{n_j}^{\delta_j}$ and $K\bs=\by$,
\[
\Rreg_{\bx_{n_j-1}^{\delta_j}}^q(\bx_{n_j}^{\delta_j})
\le
\Rreg_{\bx_{n_j-1}^{\delta_j}}^q(\bs)
+\frac{\delta_j^2}{2\alpha_{n_j}} .
\]
Passing to the limit, using \Cref{cor:R-bounds} and
\cref{eq:apriori-rule}, yields
\[
\Rreg_{\bar\bx}^q(\bx^\star)
\le
\Rreg_{\bar\bx}^q(\bs)
\qquad
\forall \bs\in\Ssol(\by).
\]
Thus
\[
\bx^\star\in
\operatorname*{argmin}_{\bs\in\Ssol(\by)}
\Rreg_{\bar\bx}^q(\bs).
\]
Since $\bar\bx\in\Ssol(\by)$, the definition \cref{eq:Sgm-def} gives
$\bx^\star\in\mathcal S_{\rm gm}(\by)$.

But then
\[
\operatorname{dist}_2
\bigl(\bx_{n_j}^{\delta_j},\mathcal S_{\rm gm}(\by)\bigr)
\le
\norm{\bx_{n_j}^{\delta_j}-\bx^\star}_2
\to0,
\]
contradicting \cref{eq:dist-contradiction}. Hence
\cref{eq:set-convergence-Sgm} holds. Since
$\mathcal S_{\rm gm}(\by)\subseteq\Ssol(\by)$,
\cref{eq:set-convergence-S} follows immediately.
\end{proof}

\section{Regularization properties of \texorpdfstring{\Cref{alg:error-iter,alg:mixed-iter}}{the error-equation and mixed iterative schemes}}\label{sec:convergence_error}

We now analyze the error-equation and mixed iterations. Throughout this section, the graph regularizer is still
\(\Rreg_\bz^q\) from \cref{def:graph-regularizer}, and the solution operator
\(T_{\alpha,\bz}^q\) is the one defined in
\cref{eq:fixed-graph-functional}. As in \Cref{sec:iterschemes}, when
\(T_{\alpha,\bz}^q\) is set-valued, one admissible minimizer is selected.

For \Cref{alg:error-iter}, we use the residual convention
\[
\br_n^\delta \coloneqq \by^\delta-K\bx_n^\delta .
\]
Then the error-equation scheme can be written compactly as
\begin{equation}\label{eq:error-scheme-shifted}
\left\{
\begin{aligned}
\br_n^\delta &= \by^\delta-K\bx_n^\delta,\\
\bh_{n+1}^\delta
&\in T_{\alpha_{n+1},\bh_n^\delta}^q(\br_n^\delta),\\
\bx_{n+1}^\delta
&=\bx_n^\delta+\bh_{n+1}^\delta,
\end{aligned}
\right.
\qquad n\ge0,
\end{equation}
where \(\bh_0^\delta\in X\) is the prescribed image-space initialization used to
construct the first error graph. Equivalently,
\[
\br_{n+1}^\delta
=
\br_n^\delta-K\bh_{n+1}^\delta
=
\by^\delta-K\bx_{n+1}^\delta .
\]

\subsection{One correction step}\label{sec:error-one-step}

For a residual \(\br\in Y\), a reference \(\bz\in X\), and
\(\alpha>0\), the correction step is
\[
\bh\in T_{\alpha,\bz}^q(\br),
\]
that is,
\[
\bh\in
\operatorname*{argmin}_{\bu\in X}
\left\{
\frac12\norm{K\bu-\br}_Y^2+\alpha\Rreg_\bz^q(\bu)
\right\}.
\]
This is exactly the fixed-reference problem
\(\Phi_{\alpha,\bz}^q(\cdot;\br)\) from
\cref{eq:fixed-graph-functional}.

\begin{proposition}\label{prop:error-one-step}
Assume \Cref{hyp:joint-coercivity}. Let \(q\in[1,2]\). For every
\(\br\in Y\), every \(\bz\in X\), and every \(\alpha>0\), the set
\(T_{\alpha,\bz}^q(\br)\) is nonempty. More precisely, if
\[
\Phi_{\alpha,\bz}^q(\bh;\br)\le M,
\]
then
\begin{equation}\label{eq:error-one-step-sublevel}
\norm{\bh}_2
\le
C_J\left(
(2M)^{1/2}+\norm{\br}_Y
+\left(\frac{qM}{\alpha}\right)^{1/q}
\right).
\end{equation}
The bound is independent of the reference image \(\bz\) and of the
discretization level \(N\). If \(q>1\), then \(T_{\alpha,\bz}^q(\br)\) is a
singleton.
\end{proposition}

\begin{proof}
This is \Cref{prop:existence} applied with data \(\by=\br\) and minimization
variable \(\bx=\bh\). The estimate \cref{eq:error-one-step-sublevel} is exactly
\cref{eq:sublevel-bound} with this notation.
\end{proof}

\begin{proposition}\label{prop:error-one-step-stability}
Assume \Cref{hyp:joint-coercivity}. Fix \(q\in[1,2]\) and \(\alpha>0\). Let
\[
\br_m\to\br \quad\text{in }Y,
\qquad
\norm{\bz_m-\bz}_\infty\to0,
\]
and choose
\[
\bh_m\in T_{\alpha,\bz_m}^q(\br_m).
\]
Then \((\bh_m)\) is bounded. Every cluster point of \((\bh_m)\) belongs to
\(T_{\alpha,\bz}^q(\br)\). If the limiting minimizer is unique, in particular if
\(q>1\), then the whole sequence converges to that minimizer.
\end{proposition}

\begin{proof}
This is \Cref{cor:one-step-stability} with \(\by_m=\br_m\),
\(\by=\br\), and \(\bx_m=\bh_m\).
\end{proof}

\subsection{Finite-iteration stability of \texorpdfstring{\Cref{alg:error-iter}}{the error-equation scheme}}\label{sec:error-finite-stability}
As for the standard scheme, one-step stability propagates by induction through
any fixed number of iterations of the error-equation scheme
\cref{eq:error-scheme-shifted}. The difference is that now we need to track the pair
$(\bx_n,\bh_n)$ of current reconstruction and graph input reference, so convergence is
required not only of the data and the initial reconstruction but also of the
initial graph input $\bh_0$. Under these assumptions, the first $J$ iterates
converge along subsequences to a trajectory of the limiting scheme. When the
minimizer is unique at every limiting correction step, the whole sequence
converges. The bounds on the iterates are independent of the discretization
level $N$ provided that the data and the initial families have bounds
independent of $N$.
\begin{theorem}\label{thm:error-finite-stability}
Assume \Cref{hyp:joint-coercivity}. Fix \(J\in\N\) and positive parameters
\(\alpha_1,\ldots,\alpha_J\). Let
\[
\by_m\to \by \quad\text{in }Y,
\qquad
\bx_{0,m}\to \bx_0 \quad\text{in }\norm{\cdot}_2,
\qquad
\norm{\bh_{0,m}-\bh_0}_\infty\to0 .
\]
For each \(m\), let \((\bx_{n,m},\bh_{n,m})\) be generated by
\cref{eq:error-scheme-shifted} with data \(\by_m\), initial reconstruction
\(\bx_{0,m}\), and initial graph input \(\bh_{0,m}\), for
\(n=0,\ldots,J\).

Then, at each fixed discretization level, the family
\[
\{\bx_{n,m}:0\le n\le J,\ m\ge1\}\cup\{\bh_{n,m}:0\le n\le J,\ m\ge1\}
\]
is bounded. If the data and the initial families are bounded by constants
independent of \(N\), then the same is true for the family of generated iterates.

Moreover, every subsequence has a further subsequence such that, for each
\(n=0,\ldots,J\),
\[
\bx_{n,m_k}\to \bar\bx_n \quad\text{in } X,
\qquad
\norm{\bh_{n,m_k}-\bar\bh_n}_\infty\to0,
\]
where \((\bar\bx_n,\bar\bh_n)_{n=0}^J\) is generated by
\cref{eq:error-scheme-shifted} with limiting data \(\by\), initial reconstruction
\(\bx_0\), and initial graph input \(\bh_0\). Equivalently,
\[
\bar\br_n=\by-K\bar\bx_n,
\qquad
\bar\bh_{n+1}\in T_{\alpha_{n+1},\bar\bh_n}^q(\bar\br_n),
\qquad
\bar\bx_{n+1}=\bar\bx_n+\bar\bh_{n+1}.
\]
If the minimizer is unique at every limiting correction step, in particular if
\(q>1\), then the whole sequence converges at each fixed step.
\end{theorem}

\begin{proof}
We argue by induction. The convergence of \(\bx_{0,m}\) and \(\bh_{0,m}\) is
assumed. Suppose that, along a subsequence,
\[
\bx_{n,m}\to \bar\bx_n,
\qquad
\norm{\bh_{n,m}-\bar\bh_n}_\infty\to0 .
\]
Then
\[
\br_{n,m}\coloneqq \by_m-K\bx_{n,m}
\to
\bar\br_n\coloneqq \by-K\bar\bx_n
\quad\text{in }Y .
\]
Since
\[
\bh_{n+1,m}\in T_{\alpha_{n+1},\bh_{n,m}}^q(\br_{n,m}),
\]
\Cref{prop:error-one-step-stability} gives boundedness of
\((\bh_{n+1,m})\), and every cluster point solves the limiting correction
problem
\[
\bar\bh_{n+1}\in T_{\alpha_{n+1},\bar\bh_n}^q(\bar\br_n).
\]
After extracting a further subsequence,
\[
\bh_{n+1,m}\to \bar\bh_{n+1}
\]
in \(X\), hence also in \(\norm{\cdot}_\infty\) at fixed \(N\). The update
\[
\bx_{n+1,m}=\bx_{n,m}+\bh_{n+1,m}
\]
then gives
\[
\bx_{n+1,m}\to \bar\bx_n+\bar\bh_{n+1}\eqqcolon \bar\bx_{n+1}.
\]
This proves the induction step. If each limiting correction step has a unique
minimizer, no subsequence extraction is needed, and the whole sequence converges.
\end{proof}

\subsection{Convergence of \texorpdfstring{\Cref{alg:error-iter}}{the error-equation scheme} for noisy data}
\label{sec:error-noisy}

In contrast to \Cref{alg:standard-iter}, the regularization term in
\Cref{alg:error-iter} acts on the correction \(\bh_{n+1}^\delta\), not directly
on the cumulative reconstruction \(\bx_n^\delta\). Consequently, without an
additional selection principle, the natural convergence statement is convergence
of the stopped reconstructions to the exact-solution set \(\Ssol(\by)\), rather
than to the graph-minimizing set \(\mathcal S_{\rm gm}(\by)\).

Let \(m(\delta)\ge1\) be an a priori stopping index satisfying
\begin{equation}\label{eq:error-apriori-rule}
\alpha_{m(\delta)}\to0,
\qquad
\frac{\delta^2}{\alpha_{m(\delta)}}\to0
\qquad\text{as }\delta\downarrow0 .
\end{equation}
The stopping index is denoted here by $m(\delta)$, rather than by $n(\delta)$ as in \Cref{sec:convergence_standard}, for consistency with the mixed scheme of \Cref{sec:mixed-convergence}, where the standard phase takes the first $n_{\mathrm{std}}$ steps and the stopped reconstruction has index $n_{\mathrm{std}}+m(\delta)$.
We assume that the predecessor of the stopped reconstruction is uniformly
bounded for all sufficiently small noise levels:
\begin{equation}\label{eq:error-noisy-bounded-x}
\norm{\bx_{m(\delta)-1}^\delta}_2
\le M_X
\qquad\text{for all sufficiently small }\delta>0,
\end{equation}
with \(M_X\) independent of \(N\) and of $\delta$.

\begin{theorem}
\label{thm:error-noisy-convergence}
Assume \Cref{hyp:consistency,hyp:joint-coercivity}. Let
\(\norm{\by^\delta-\by}_Y\le\delta\), and let
\((\bx_n^\delta,\bh_n^\delta)\) be generated by
\cref{eq:error-scheme-shifted}. Suppose that
\cref{eq:error-apriori-rule,eq:error-noisy-bounded-x} hold. Then the stopped
family \((\bx_{m(\delta)}^\delta)\) is uniformly bounded in \(\norm{\cdot}_2\)
for all sufficiently small \(\delta>0\), and
\begin{equation}\label{eq:error-residual-convergence}
\norm{K\bx_{m(\delta)}^\delta-\by}_Y\to0
\qquad\text{as }\delta\downarrow0.
\end{equation}
Moreover,
\begin{equation}\label{eq:error-set-convergence}
\operatorname{dist}_2
\bigl(\bx_{m(\delta)}^\delta,\Ssol(\by)\bigr)
\to0
\qquad\text{as }\delta\downarrow0.
\end{equation}
\end{theorem}

\begin{proof}
Set \(m=m(\delta)\). Let \(\gt\) be the bounded exact solution from
\Cref{hyp:consistency}, and set
\[
\be_{m-1}^\delta\coloneqq \gt-\bx_{m-1}^\delta .
\]
Since \(K\gt=\by\) and
\(\br_{m-1}^\delta=\by^\delta-K\bx_{m-1}^\delta\),
\[
K\be_{m-1}^\delta-\br_{m-1}^\delta
=
\by-\by^\delta,
\qquad
\norm{K\be_{m-1}^\delta-\br_{m-1}^\delta}_Y\le\delta .
\]
Moreover, because \(q\le2\),
\[
\norm{\be_{m-1}^\delta}_q
\le
\norm{\gt}_q+\norm{\bx_{m-1}^\delta}_q
\le
M_{\rm gt}+M_X .
\]
Hence, by \cref{eq:R-uniform-bound},
\[
\Rreg_{\bh_{m-1}^\delta}^q(\be_{m-1}^\delta)
\le
C_E,
\qquad
C_E\coloneqq
\frac{(2d_R)^q}{q}(M_{\rm gt}+M_X)^q .
\]
Minimality of \(\bh_m^\delta\) gives
\[
\frac12\norm{K\bh_m^\delta-\br_{m-1}^\delta}_Y^2
+
\alpha_m\Rreg_{\bh_{m-1}^\delta}^q(\bh_m^\delta)
\le
\frac12\delta^2+\alpha_m C_E .
\]
Since
\[
K\bh_m^\delta-\br_{m-1}^\delta
=
-\br_m^\delta,
\]
we obtain
\begin{equation*}
\norm{\br_m^\delta}_Y^2
\le
\delta^2+2\alpha_m C_E .
\end{equation*}
By \cref{eq:error-apriori-rule}, the right-hand side tends to zero. Therefore
\[
\norm{K\bx_m^\delta-\by}_Y
\le
\norm{K\bx_m^\delta-\by^\delta}_Y
+\norm{\by^\delta-\by}_Y
=
\norm{\br_m^\delta}_Y
+\norm{\by^\delta-\by}_Y
\to0,
\]
which proves \cref{eq:error-residual-convergence}.

It remains to show the boundedness of the stopped reconstructions. Define
\[
M_\delta\coloneqq \frac12\delta^2+\alpha_m C_E .
\]
The bound in \cref{eq:error-noisy-bounded-x} and \Cref{hyp:joint-coercivity} give
\[
\norm{\br_{m-1}^\delta}_Y
\le
\norm{\by}_Y+\delta+M_KM_X
\]
for all sufficiently small $\delta>0$; in particular, $\norm{\br_{m-1}^\delta}_Y$ stays bounded as $\delta\downarrow0$. Applying the sublevel estimate
\cref{eq:error-one-step-sublevel} to \(\bh_m^\delta\) yields
\[
\norm{\bh_m^\delta}_2
\le
C_J\left(
(2M_\delta)^{1/2}
+\norm{\br_{m-1}^\delta}_Y
+\left(\frac{qM_\delta}{\alpha_m}\right)^{1/q}
\right).
\]
The first term is bounded, the second term is bounded by the previous estimate,
and
\[
\frac{M_\delta}{\alpha_m}
=
\frac{\delta^2}{2\alpha_m}+C_E
\]
is bounded for small \(\delta\) by \cref{eq:error-apriori-rule}. Hence
\((\bh_m^\delta)\) is bounded in \(\norm{\cdot}_2\). Together with
\cref{eq:error-noisy-bounded-x}, this gives boundedness of
\[
\bx_m^\delta=\bx_{m-1}^\delta+\bh_m^\delta .
\]

We now prove \cref{eq:error-set-convergence}. Suppose, by contradiction, that it
fails. Then there exist \(\varepsilon>0\) and a sequence
\(\delta_j\downarrow0\) such that
\[
\operatorname{dist}_2
\bigl(\bx_{m(\delta_j)}^{\delta_j},\Ssol(\by)\bigr)
\ge\varepsilon
\qquad\forall j.
\]
By the boundedness just proved and finite-dimensional compactness, after passing
to a subsequence,
\[
\bx_{m(\delta_j)}^{\delta_j}\to\bx^\star .
\]
The residual convergence \cref{eq:error-residual-convergence} gives
\(K\bx^\star=\by\), hence \(\bx^\star\in\Ssol(\by)\). Therefore
\[
\operatorname{dist}_2
\bigl(\bx_{m(\delta_j)}^{\delta_j},\Ssol(\by)\bigr)
\le
\norm{\bx_{m(\delta_j)}^{\delta_j}-\bx^\star}_2
\to0,
\]
contradicting the choice of the sequence. This proves
\cref{eq:error-set-convergence}.
\end{proof}

\begin{remark}
The boundedness condition \cref{eq:error-noisy-bounded-x} is the only additional condition required by the error-equation scheme but not by the standard scheme. In
\Cref{alg:standard-iter}, joint coercivity controls the same variable that is
being reconstructed, because the functional contains
\(\Rreg_{\bx_{n-1}^\delta}^q(\bx_n^\delta)\). In
\Cref{alg:error-iter}, the functional controls the correction
\(\bh_m^\delta\), while the reconstruction is the cumulative quantity
\(\bx_m^\delta=\bx_{m-1}^\delta+\bh_m^\delta\). Thus joint coercivity alone does
not control the previous reconstruction \(\bx_{m-1}^\delta\); some boundedness or stopping
selection condition is needed unless additional structure is imposed. If the reconstruction iterates remain in a fixed grayscale range, this bound follows immediately in the normalized norm.
\end{remark}

\subsection{Convergence of \texorpdfstring{\Cref{alg:mixed-iter}}{the mixed scheme} for noisy data}
\label{sec:mixed-convergence}

The mixed scheme consists of a fixed standard phase followed by an error-equation
phase. Since the standard phase has fixed length, its finite-iteration stability
follows from \Cref{thm:finite-stability}; after that point, the error phase is
exactly \Cref{alg:error-iter} with shifted indices. Therefore the vanishing-noise
argument is inherited from \Cref{thm:error-noisy-convergence}.

Fix \(n_{\mathrm{std}}\in\mathbb N_{>0}\). After the standard phase, write the
chosen initial graph input for the error stage as \(\bg_0^\delta\), and set
\[
\br_{n_{\mathrm{std}}}^\delta
=
\by^\delta-K\bx_{n_{\mathrm{std}}}^\delta .
\]
For \(k\ge0\), the error stage of \Cref{alg:mixed-iter} is
\begin{equation}\label{eq:mixed-error-stage}
\left\{
\begin{aligned}
\br_{n_{\mathrm{std}}+k}^\delta
&=
\by^\delta-K\bx_{n_{\mathrm{std}}+k}^\delta,\\
\bg_{k+1}^\delta
&\in
T_{\alpha_{n_{\mathrm{std}}+k+1},\bg_k^\delta}^q
\bigl(\br_{n_{\mathrm{std}}+k}^\delta\bigr),\\
\bx_{n_{\mathrm{std}}+k+1}^\delta
&=
\bx_{n_{\mathrm{std}}+k}^\delta+\bg_{k+1}^\delta .
\end{aligned}
\right.
\end{equation}

Let the number \(m(\delta)\ge1\) of error-stage iterations satisfy
\begin{equation}\label{eq:mixed-apriori-rule}
\alpha_{n_{\mathrm{std}}+m(\delta)}\to0,
\qquad
\frac{\delta^2}{\alpha_{n_{\mathrm{std}}+m(\delta)}}\to0
\qquad\text{as }\delta\downarrow0 .
\end{equation}
Assume that the predecessor of the stopped mixed reconstruction is uniformly
bounded for all sufficiently small noise levels:
\begin{equation}\label{eq:mixed-noisy-bounded}
\norm{\bx_{n_{\mathrm{std}}+m(\delta)-1}^\delta}_2
\le M_X
\qquad\text{for all sufficiently small }\delta>0,
\end{equation}
with \(M_X\) independent of \(N\) and of $\delta$.

\begin{theorem}
\label{thm:mixed-noisy-convergence}
Assume \Cref{hyp:consistency,hyp:joint-coercivity}. Let
\(\norm{\by^\delta-\by}_Y\le\delta\). For each \(\delta\), compute
\(n_{\mathrm{std}}\) standard iterations and then apply the error stage
\cref{eq:mixed-error-stage}. Suppose that
\cref{eq:mixed-apriori-rule,eq:mixed-noisy-bounded} hold. Then the stopped mixed
family
\[
\bigl(\bx_{n_{\mathrm{std}}+m(\delta)}^\delta\bigr)
\]
is uniformly bounded in \(\norm{\cdot}_2\) for all sufficiently small
\(\delta>0\), and
\begin{equation}\label{eq:mixed-residual-convergence}
\norm{K\bx_{n_{\mathrm{std}}+m(\delta)}^\delta-\by}_Y\to0
\qquad\text{as }\delta\downarrow0.
\end{equation}
Moreover,
\begin{equation}\label{eq:mixed-set-convergence}
\operatorname{dist}_2
\bigl(\bx_{n_{\mathrm{std}}+m(\delta)}^\delta,\Ssol(\by)\bigr)
\to0
\qquad\text{as }\delta\downarrow0 .
\end{equation}
\end{theorem}

\begin{proof}
Apply \Cref{thm:error-noisy-convergence} to the shifted variables
\[
\widetilde\bx_k^\delta
\coloneqq
\bx_{n_{\mathrm{std}}+k}^\delta,
\qquad
\widetilde\bh_k^\delta
\coloneqq
\bg_k^\delta,
\qquad
\widetilde\alpha_k
\coloneqq
\alpha_{n_{\mathrm{std}}+k}.
\]
Then \cref{eq:mixed-error-stage} is exactly
\cref{eq:error-scheme-shifted} for
\((\widetilde\bx_k^\delta,\widetilde\bh_k^\delta)\). The parameter rule
\cref{eq:mixed-apriori-rule} is \cref{eq:error-apriori-rule} for
\(\widetilde\alpha_{m(\delta)}\), and the boundedness condition
\cref{eq:mixed-noisy-bounded} is \cref{eq:error-noisy-bounded-x} for the shifted
predecessor
\[
\widetilde\bx_{m(\delta)-1}^\delta
=
\bx_{n_{\mathrm{std}}+m(\delta)-1}^\delta .
\]
Therefore \Cref{thm:error-noisy-convergence} gives boundedness of
\[
\widetilde\bx_{m(\delta)}^\delta
=
\bx_{n_{\mathrm{std}}+m(\delta)}^\delta,
\]
as well as \cref{eq:mixed-residual-convergence} and
\cref{eq:mixed-set-convergence}.
\end{proof}

\begin{remark}
The fixed standard phase does not need to converge to an exact solution for
\Cref{thm:mixed-noisy-convergence}; it only supplies the initial reconstruction
and initial residual for the shifted error-equation scheme. 
\end{remark}

\section{Numerical tests}\label{sec:numerical_results}

In this section, we present numerical results for image deblurring and CT problems
obtained with the proposed iterative schemes described in \Cref{sec:iterschemes}.
For the graph Laplacian iterations, we consider $\ell^2$--$\ell^1$ regularization, i.e., $q=1$. The approximate solution of the resulting $\ell^2$--$\ell^1$ minimization problems is computed by the majorization-minimization generalized Krylov subspace (MM-GKS) method with the adaptive majorant proposed in \cite{LMRS15} as implemented in \cite{buccini2024software}.
This implementation employs a restarting strategy for the generalized Krylov subspace to reduce the computational cost. Moreover, at each graph iteration, GCV selects the regularization parameter of the quadratic subproblems at every inner MM--GKS iteration; this practical parameter choice is not guaranteed to satisfy the a priori rules of Theorems~\ref{thm:standard-vanishing-noise}, \ref{thm:error-noisy-convergence}, and~\ref{thm:mixed-noisy-convergence}. The schemes do not depend on this choice of solver: the minimization steps of \Cref{alg:standard-iter,alg:error-iter,alg:mixed-iter} can be carried out by any method for $\ell^2$--$\ell^q$ problems, for instance by the alternating direction method of multipliers (ADMM)~\cite{boyd2011distributed} or by the variable projected augmented Lagrangian method~\cite{aleotti2026variable}.

For every fixed deblurring or CT matrix used below,
\Cref{rem:finite-dim-coercivity} gives the joint coercivity needed for the
$q=1$ numerical problem, although its constant may depend on $N$. At the
dimension-free level, \Cref{prop:dim-free-deblurring} treats the periodic blur
model for $q=2$, whereas \Cref{prop:dim-free-CT} treats the full-angle CT model
for every $q\in[1,2]$.

Unless otherwise stated, the parameters defining the graph Laplacian are set to $R=5$ and $\sigma=10^{-3}$ for the standard iterative scheme \cref{eq:standard-iter}; for the error-equation stage of the mixed iterative scheme in \Cref{alg:mixed-iter}, we keep $R=5$ and set $\sigma=10^{-4}$. These values, as well as the numbers of standard and error-equation iterations used in each test, were set by hand, without any specific fine-tuning; a practical stopping rule for the numerical implementation is left for future research.
The numerical graph implementations use Gaussian weights without an explicit
contrast cap. For any fixed finite run, introducing a cap $B_{\rm c}$ as in
\cref{eq:contrast-cap} larger than every pixel-intensity difference encountered
in the graph inputs would leave the computed weights unchanged.

We compare the reconstruction quality of our iterated graph Laplacian methods with that obtained after the first iteration, i.e., the graphLa\texttt{+}$\Psi$ method proposed in \cite{bianchi2025data}.
In Tests~1 and~3 we also consider two reconstruction methods from the recent literature. The first is the iterative regularization method assisted by the graph Laplacian (IRMGL) of~\cite{bajpai2026convergence}, recalled in the Introduction: a Landweber-type iteration with an additional graph Laplacian term, in which the graph is recomputed from the current iterate and the iteration is stopped early. The second is the plug-and-play ADMM method of~\cite{Chan2017-pb}, in which the proximal step associated with the regularizer in ADMM is replaced by a denoiser; with the BM3D denoiser~\cite{Dabov2007-bm3d}, we refer to it as ADMM-BM3D. Both methods are used in two ways: to compute the initial reconstruction $\bx_0^\delta$ of our schemes, and as terms of comparison.

To assess the quality of the reconstructions, we consider the standard imaging performance metrics: the Relative Reconstruction Error (RRE), defined for a reconstruction $\bx$ as $\mathrm{RRE}(\bx)\coloneqq\norm{\bx-\gt}_2/\norm{\gt}_2$, the Peak Signal-to-Noise Ratio (PSNR), the Structural Similarity Index Measure (SSIM), and the Gradient Magnitude Similarity Deviation (GMSD) \cite{Xue2014-wt}, which is closely related to SSIM and assesses the similarity between two images through the standard deviation of the gradient magnitude similarity map. Lower values of RRE and GMSD, and higher values of PSNR and SSIM, indicate a better reconstruction.
When reporting performance metrics for a nonrectangular image region, we compute SSIM and GMSD on the smallest rectangular subimage containing that region.
In the tables and figure captions, boldface indicates the best value for each metric on the same image and region. In the tables for Test~1, reconstructions are compared separately for each initialization.

Within each iteration block of \Cref{Table: test4,Table: test_CT_REAL,Table:network_test}, a horizontal line separates the standard iterates $\bx_n^\delta$ from the mixed reconstructions $\bx_{n_{\mathrm{std}}+k}^\delta$ after $k$ error-equation steps, consistent with~\cref{eq:mixed-error-stage}.

\subsection[Test 1: Satellite deblurring problem]{Test 1: Satellite deblurring  problem}
We consider the images in \Cref{fig:test4}.
The test image is a $238 \times 238$ grayscale image, and the point spread function (PSF) has a support of $9 \times 9$ pixels. Gaussian noise at level $1\%$ has been added to the blurred image to obtain the observed corrupted image.

\begin{figure}
        \centering
		\begin{subfigure}{0.22\textwidth}
			\includegraphics[width=\textwidth]{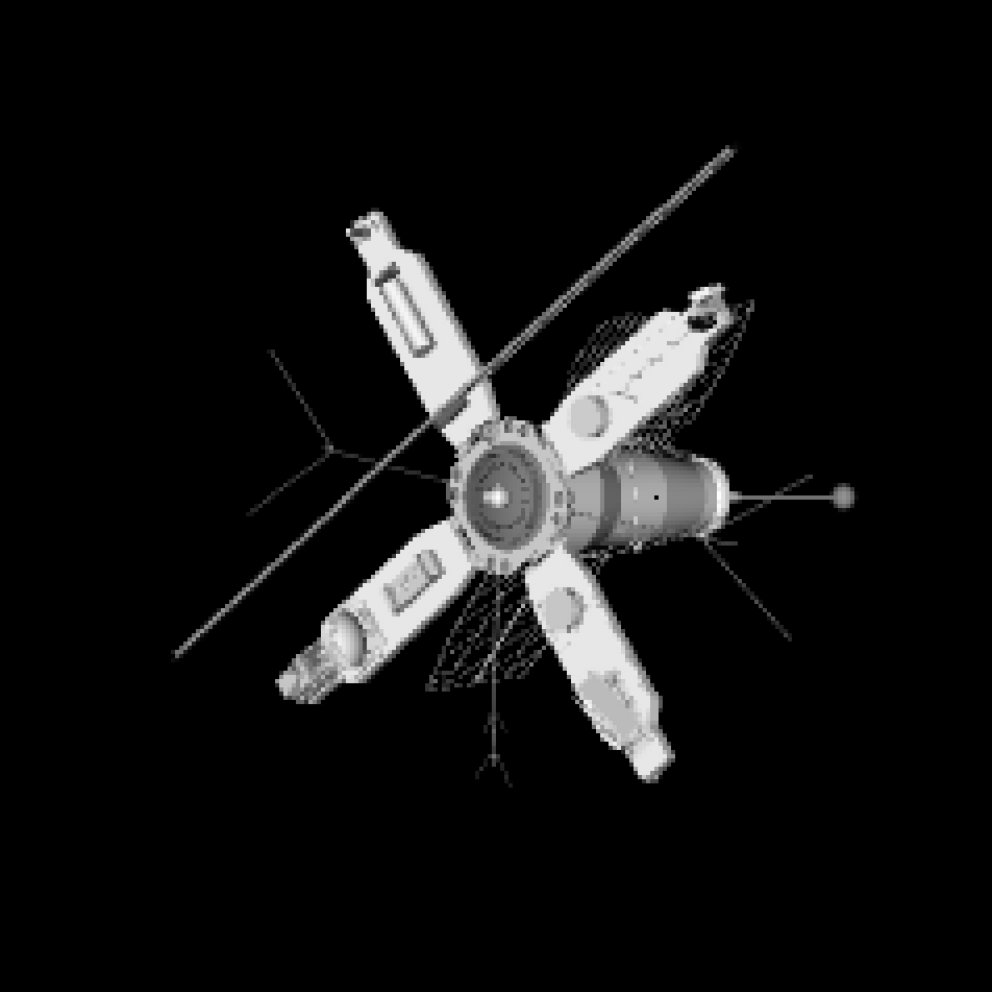}
			\caption{}
		\end{subfigure}\hfill
		\begin{subfigure}{0.22\textwidth}
			\includegraphics[width=\textwidth]{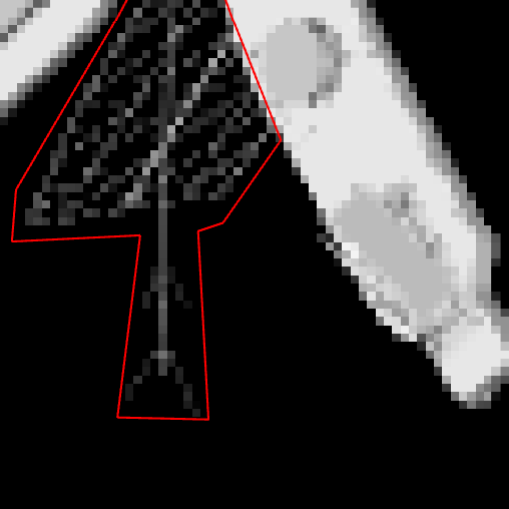}
			\caption{}
		\end{subfigure}\hfill
		\begin{subfigure}{0.22\textwidth}
			\includegraphics[clip,trim={2cm 1.5cm 2cm 1.5cm}, width=\textwidth, height=\textwidth]{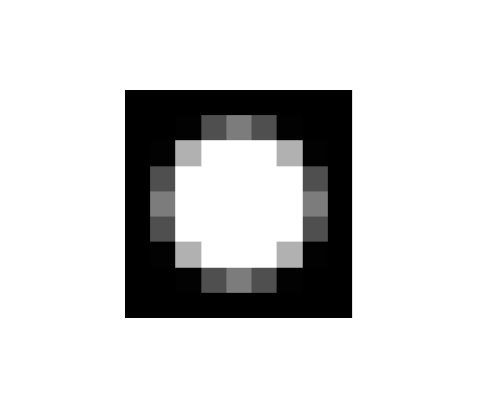}
			\caption{}
		\end{subfigure}\hfill
		\begin{subfigure}{0.22\textwidth}
			\includegraphics[width=\textwidth]{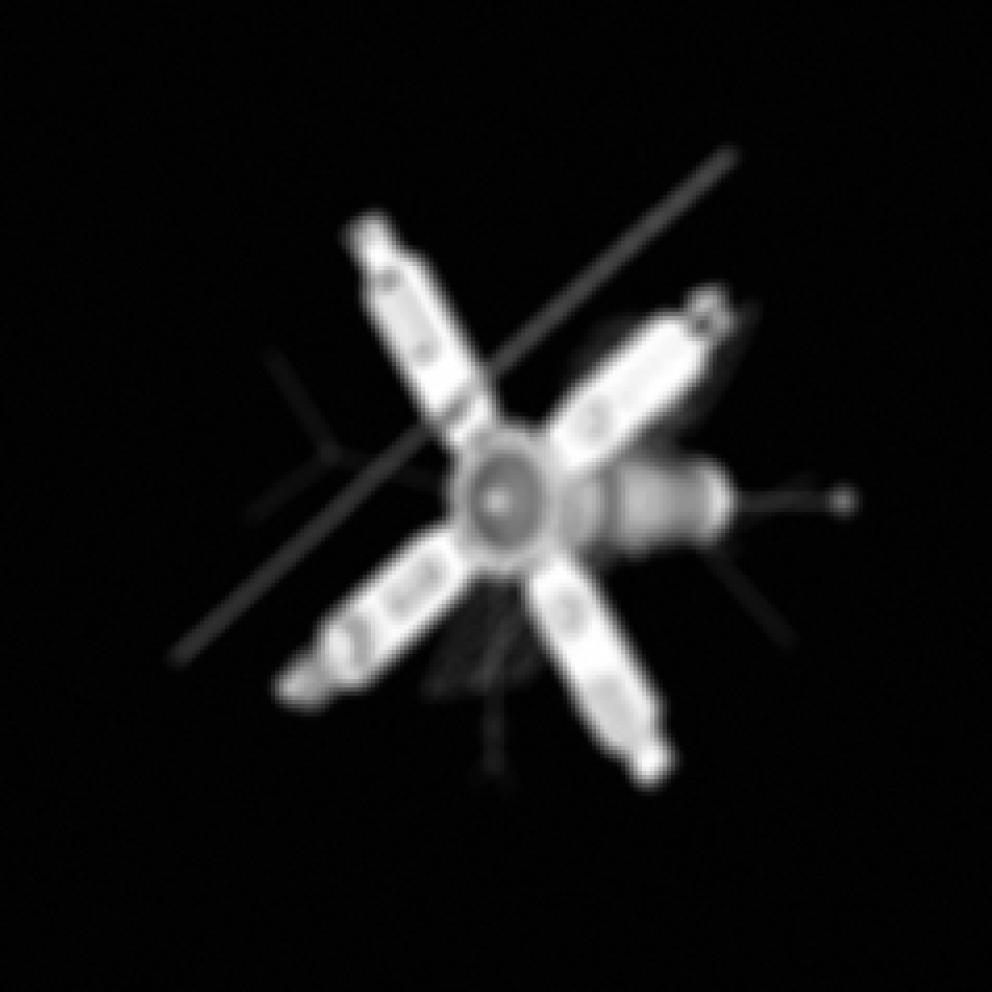}
			\caption{}
		\end{subfigure}
		\caption{Satellite deblurring test: (a) original image, (b) detail of the original image, with the region in the red box used for the computed metrics, (c) PSF with a $9\times9$-pixel support, (d) corrupted image.}
		\label{fig:test4}
	\end{figure}

To show the importance of a good initial approximation $\bx_0^\delta$, we consider four initial reconstructions: a reconstruction obtained by $\ell^2$--$\ell^2$ Tikhonov regularization; a total variation (TV) reconstruction, corresponding to the minimization problem \cref{model_eq2} with $q=1$ and a finite-difference approximation of the gradient as the operator $D$; IRMGL; and ADMM-BM3D.
Note that the solution computed after the first iteration, i.e., $\bx_1^\delta$, is exactly the one computed by the graphLa\texttt{+}$\Psi$ method proposed in \cite{bianchi2025data}, with $\Psi$ the method that computes $\bx_0^\delta$.

The ADMM-BM3D reconstruction is computed with 100 ADMM iterations, and its parameters are chosen by an oracle, i.e., they are tuned using the ground truth. This is a favorable setting for ADMM-BM3D that cannot be reproduced in practice, so its results should be read as a best-case reference. The same holds for IRMGL: the iteration is started from the adjoint reconstruction $K^*\by^\delta$, from which the initial graph is built, and is stopped by an oracle at the iterate that minimizes the RRE ($k=2157$). For Tikhonov and TV, instead, the ground truth is not used: the regularization parameter is estimated automatically by GCV for Tikhonov and by the discrepancy principle for TV according to the implementation in~\cite{buccini2024software}.

For the mixed iteration, in the notation of \cref{eq:mixed-error-stage}, the initial error-stage graph input is $\bg_0^\delta=\br_{n_{\mathrm{std}}}^\delta$, consistent with the choice in \cref{eq:g0res}.
The results for the general approach, which computes $\bg_0^\delta$ using Tikhonov regularization to approximate the solution of the error equation $K\bg=\br_{n_{\mathrm{std}}}^\delta$, are not reported since they are indistinguishable from those obtained with the simple choice $\bg_0^\delta=\br_{n_{\mathrm{std}}}^\delta$.
The mixed scheme is run with $n_{\mathrm{std}}=5$ standard iterations and six error-equation steps for Tikhonov, with $n_{\mathrm{std}}=4$ and one error-equation step for TV, and with $n_{\mathrm{std}}=2$ and one error-equation step for IRMGL and ADMM-BM3D.

\begin{table}
\centering
\footnotesize
\setlength{\tabcolsep}{4pt}
\begin{tabular}{l rrrr rrrr}
\toprule
 & \multicolumn{4}{c}{Tikhonov} & \multicolumn{4}{c}{TV} \\
\cmidrule(lr){2-5}\cmidrule(l){6-9}
 & RRE\(\downarrow\)& PSNR\(\uparrow\) & SSIM\(\uparrow\) & GMSD\(\downarrow\) & RRE\(\downarrow\) & PSNR\(\uparrow\)  &  SSIM\(\uparrow\) &  GMSD\(\downarrow\) \\
\midrule
$\bx_{0}^{\delta}$ & 0.2178 & 26.23 & 0.8949 & 0.0739 & 0.1418 & 29.96 & 0.9077 & 0.0278 \\
$\bx_{1}^{\delta}$ & 0.1434 & 29.87 & 0.9565 & 0.0290 & 0.1076 & 32.36 & 0.9759 & 0.0124 \\
$\bx_{2}^{\delta}$ & 0.1222 & 31.25 & 0.9679 & 0.0209 & 0.0975 & 33.22 & 0.9814 & 0.0097 \\
$\bx_{3}^{\delta}$ & 0.1155 & 31.74 & 0.9721 & 0.0197 & 0.0941 & 33.53 & 0.9832 & 0.0093 \\
$\bx_{4}^{\delta}$ & 0.1109 & 32.10 & 0.9744 & 0.0193 &  &  &  &  \\
$\bx_{n_{\mathrm{std}}}^{\delta}$ & 0.1099 & 32.17 & \textbf{0.9756} & 0.0195 & 0.0939 & 33.54 & \textbf{0.9834} & 0.0091 \\
\midrule
$\bx_{n_{\mathrm{std}}+1}^{\delta}$ & 0.1093 & 32.22 & 0.9737 & 0.0176 & \textbf{0.0937} & \textbf{33.56} & 0.9811 & \textbf{0.0087} \\
$\bx_{n_{\mathrm{std}}+6}^{\delta}$ & \textbf{0.1067} & \textbf{32.43} & 0.9579 & \textbf{0.0110} &  &  &  &  \\
\midrule[\heavyrulewidth]
 & \multicolumn{4}{c}{IRMGL} & \multicolumn{4}{c}{ADMM-BM3D} \\
\cmidrule(lr){2-5}\cmidrule(l){6-9}
 & RRE\(\downarrow\) &  PSNR\(\uparrow\) & SSIM\(\uparrow\) &  GMSD\(\downarrow\) &  RRE\(\downarrow\) & PSNR\(\uparrow\) & SSIM\(\uparrow\) & GMSD\(\downarrow\) \\
\midrule
$\bx_{0}^{\delta}$ & 0.1140 & 31.86 & 0.8686 & 0.0164 & 0.0816 & 34.76 & 0.9870 & 0.0060 \\
$\bx_{1}^{\delta}$ & 0.0883 & 34.08 & 0.9838 & 0.0095 & 0.0775 & 35.21 & 0.9894 & \textbf{0.0045} \\
$\bx_{n_{\mathrm{std}}}^{\delta}$ & \textbf{0.0842} & \textbf{34.49} & \textbf{0.9867} & 0.0082 & \textbf{0.0768} & \textbf{35.28} & \textbf{0.9896} & 0.0047 \\
\midrule
$\bx_{n_{\mathrm{std}}+1}^{\delta}$ & \textbf{0.0842} & \textbf{34.49} & 0.9842 & \textbf{0.0081} & 0.0769 & 35.27 & 0.9872 & 0.0046 \\
\bottomrule
\end{tabular}
\caption{Performance metrics for Test 1 along the iterations of the proposed scheme, for $\bx_{0}^{\delta}$ computed by Tikhonov and TV (top) and by IRMGL and ADMM-BM3D (bottom). Rows above the horizontal line report the standard iterates, with $n_{\mathrm{std}}=5$ for Tikhonov, $n_{\mathrm{std}}=4$ for TV, and $n_{\mathrm{std}}=2$ for IRMGL and ADMM-BM3D; rows below it report the error-equation stage of the mixed scheme. Rows that do not apply to an initialization are left empty.}
\label{Table: test4}
\end{table}

The performance metrics reported in \Cref{Table: test4} confirm the significant improvement achieved already after the first iteration, in agreement with the results in~\cite{bianchi2025data}, and show that the subsequent iterations improve the graphLa\texttt{+}$\Psi$ reconstruction $\bx_1^\delta$ further. For Tikhonov, the standard iterations reduce the RRE from $0.1434$ to $0.1099$, and the error-equation stage reduces it to $0.1067$; for TV, the RRE decreases from $0.1076$ to $0.0939$ and then to $0.0937$.
In particular, the GMSD values confirm the effectiveness of the error-equation stage of the mixed iterative scheme: for Tikhonov, the GMSD decreases from $0.0195$ to $0.0110$. In this stage the SSIM of the whole image decreases slightly, from $0.9756$ to $0.9579$. The two indices measure different features: GMSD assesses the fidelity of the gradient magnitude, that is, of edges and fine structures, which the error-equation corrections are designed to recover, whereas SSIM averages local statistics over the whole image, most of which consists here of homogeneous regions, where the corrections are not expected to help. The same behavior, on a smaller scale, is observed for the other three initial reconstructions: the error-equation step slightly lowers the SSIM of the whole image, while RRE and PSNR remain essentially unchanged and GMSD decreases slightly.

The improvement in the quality of the reconstructions is more evident when looking at the details in \Cref{fig:reconstructions}, which shows the region of \Cref{fig:test4}(b) for each initial reconstruction, for $\bx_{n_{\mathrm{std}}}^{\delta}$, and for the final reconstruction computed by the mixed iterative scheme; the metrics for the initial and final reconstructions are reported in \Cref{tab:details}. For Tikhonov and TV, the initial reconstruction is less accurate, and the iterations improve it markedly: on the whole image, the RRE decreases from $0.2178$ to $0.1067$ and from $0.1418$ to $0.0937$, respectively. The Tikhonov case, in particular, motivates the mixed scheme: once the standard iterations have reached a plateau, the error-equation stage recovers further details, and on the detail it improves all the indices, SSIM included (PSNR from $20.64$ to $21.56$, SSIM from $0.7836$ to $0.8364$, GMSD from $0.0820$ to $0.0414$). For IRMGL and ADMM-BM3D, thanks also to the oracle estimation of the parameters, the initial reconstructions are already accurate, and the improvement provided by the additional iterations is smaller. Nevertheless, the proposed iteration still enhances the image details: on the whole image, the RRE decreases from $0.1140$ to $0.0842$ for IRMGL and from $0.0816$ to $0.0769$ for ADMM-BM3D, and on the detail all four indices improve.

\begin{figure}
\centering
\begin{subfigure}[t]{0.19\linewidth}
  \centering
  \includegraphics[width=\linewidth]{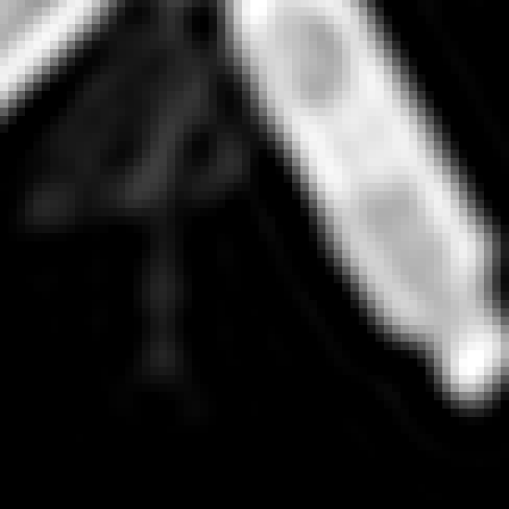}
  \caption{}
\end{subfigure}
\enspace
\begin{subfigure}[t]{0.19\linewidth}
  \centering
  \includegraphics[width=\linewidth]{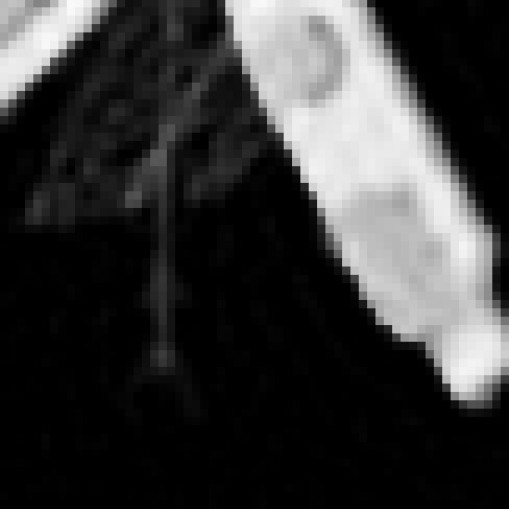}
  \caption{}
\end{subfigure}
\enspace
\begin{subfigure}[t]{0.19\linewidth}
  \centering
  \includegraphics[width=\linewidth]{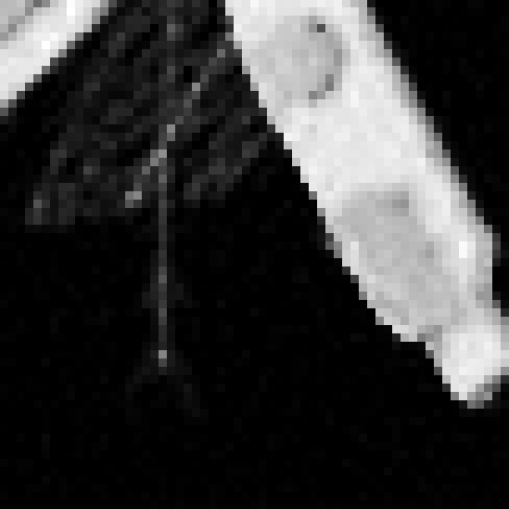}
  \caption{}
\end{subfigure}
\enspace
\begin{subfigure}[t]{0.19\linewidth}
  \centering
  \includegraphics[width=\linewidth]{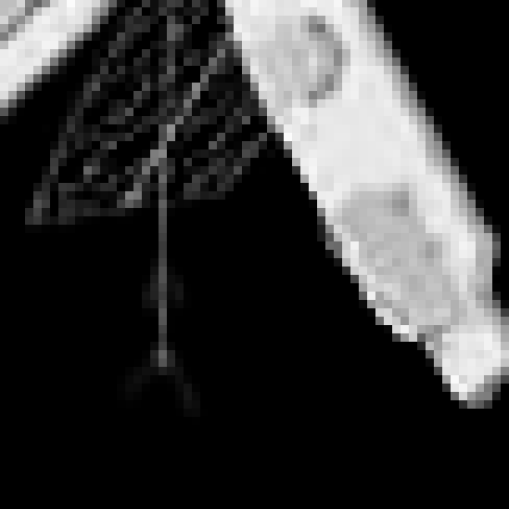}
  \caption{}
\end{subfigure}

\begin{subfigure}[t]{0.19\linewidth}
  \centering
  \includegraphics[width=\linewidth]{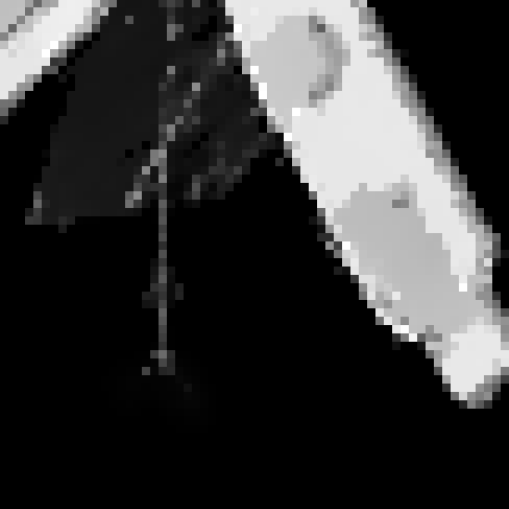}
  \caption{}
\end{subfigure}
\enspace
\begin{subfigure}[t]{0.19\linewidth}
  \centering
  \includegraphics[width=\linewidth]{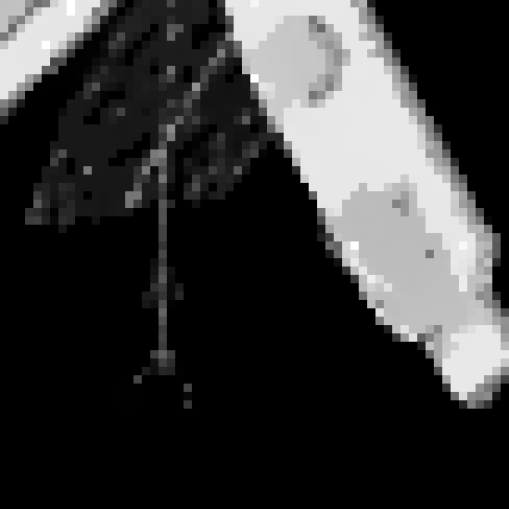}
  \caption{}
\end{subfigure}
\enspace
\begin{subfigure}[t]{0.19\linewidth}
  \centering
  \includegraphics[width=\linewidth]{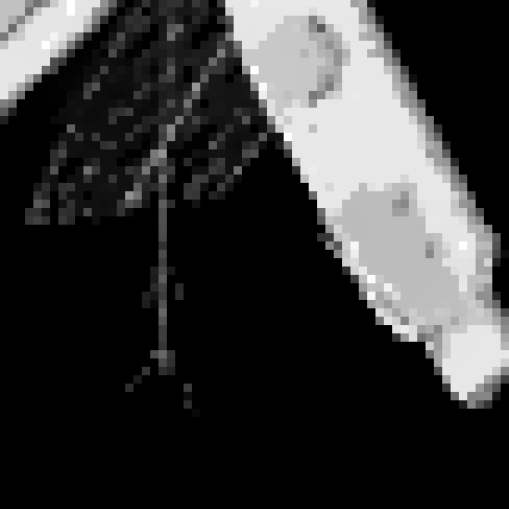}
  \caption{}
\end{subfigure}
\enspace
\begin{subfigure}[t]{0.19\linewidth}
  \centering
  \includegraphics[width=\linewidth]{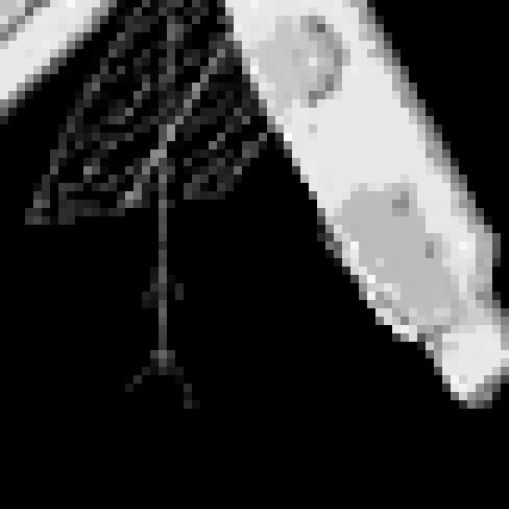}
  \caption{}
\end{subfigure}

\begin{subfigure}[t]{0.19\linewidth}
  \centering
  \includegraphics[width=\linewidth]{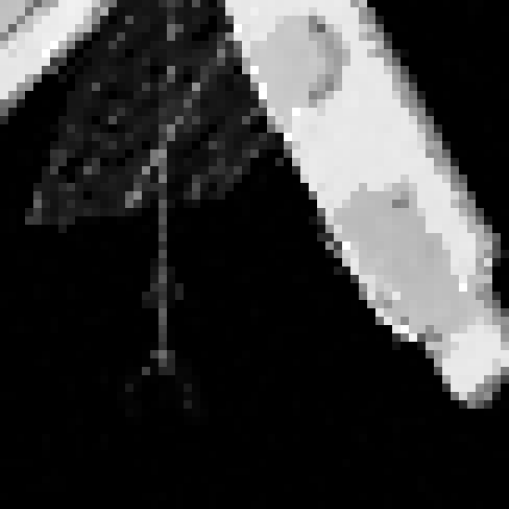}
  \caption{}
\end{subfigure}
\enspace
\begin{subfigure}[t]{0.19\linewidth}
  \centering
  \includegraphics[width=\linewidth]{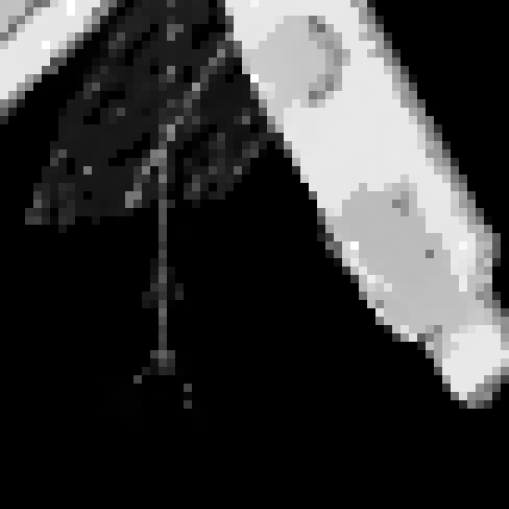}
  \caption{}
\end{subfigure}
\enspace
\begin{subfigure}[t]{0.19\linewidth}
  \centering
  \includegraphics[width=\linewidth]{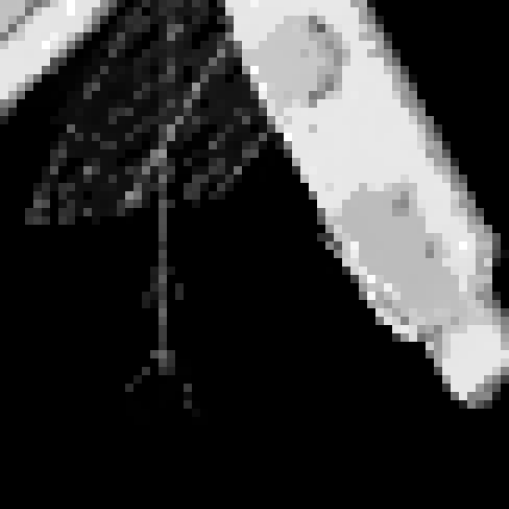}
  \caption{}
\end{subfigure}
\enspace
\begin{subfigure}[t]{0.19\linewidth}
  \centering
  \includegraphics[width=\linewidth]{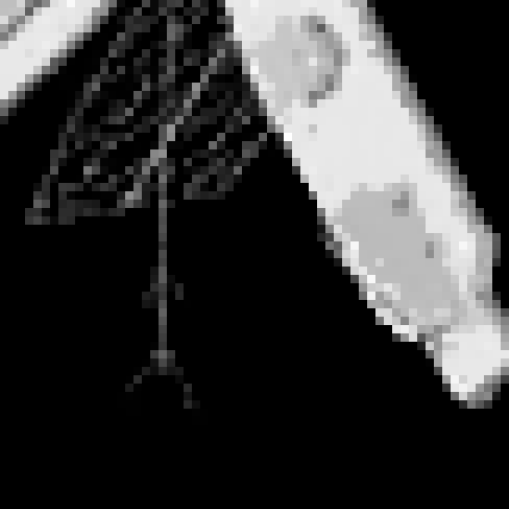}
  \caption{}
\end{subfigure}
\caption{Details of the reconstructions for Test 1, in the region shown in \Cref{fig:test4}(b): (a)--(d) initial reconstructions computed by Tikhonov, TV, IRMGL, and ADMM-BM3D, respectively; (e)--(h) corresponding reconstructions after the standard iterations ($\bx_{n_{\mathrm{std}}}^{\delta}$); (i)--(l) corresponding final reconstructions of the mixed scheme. The quality metrics for the initial and final reconstructions are reported in \Cref{tab:details}.}
\label{fig:reconstructions}
\end{figure}

\begin{table}[ht]
\centering
\small
\begin{tabular}{l rrrr rrrr}
\toprule
 & \multicolumn{4}{c}{Initial reconstruction, (a)--(d)} & \multicolumn{4}{c}{Mixed iteration, (i)--(l)} \\
\cmidrule(lr){2-5}\cmidrule(l){6-9}
Initialization & {RRE\(\downarrow\)} & {PSNR\(\uparrow\)} & {SSIM\(\uparrow\)} & {GMSD\(\downarrow\)} & {RRE\(\downarrow\) } & {PSNR\(\uparrow\) } & {SSIM\(\uparrow\)} & {GMSD\(\downarrow\)} \\
\midrule
Tikhonov         & 0.6806     & 18.27    & 0.5814     & 0.1088     & \textbf{0.4662} & \textbf{21.56} & \textbf{0.8364} & \textbf{0.0414} \\
TV               & 0.5519 & 20.10 & 0.6985 & 0.0744 & \textbf{0.4632} & \textbf{21.62} & \textbf{0.8469} & \textbf{0.0440} \\
IRMGL            & 0.4610 & 21.66 & 0.7614 & 0.0540 & \textbf{0.3939} & \textbf{23.03} & \textbf{0.8849} & \textbf{0.0315} \\
ADMM-BM3D        & 0.3349 & 24.43 & 0.9169 & 0.0181 & \textbf{0.3064} & \textbf{25.21} & \textbf{0.9412} & \textbf{0.0158} \\
\bottomrule
\end{tabular}
\caption{Quality metrics computed on the details shown in
\Cref{fig:reconstructions}: initial reconstructions, panels (a)--(d), and final reconstructions of the mixed scheme, panels (i)--(l).}
\label{tab:details}
\end{table}

\subsection{Test 2: a computed tomography example}\label{ssec:CT}

In this numerical test, the image is a $256 \times 256$ grayscale image from a real CT scan, taken from the AAPM Low Dose CT Grand Challenge dataset~\cite{Moen2021-me}. The sinogram is computed using AIR Tools II~\cite{Hansen2017-bb}. We consider a simulated CT acquisition with 60 equispaced projection angles in the interval $[0,\pi]$ and 362 detector bins. To simulate real-world scenarios, the sinogram is corrupted with additive Gaussian noise at a noise level of $1\%$. The original image and the noisy sinogram are shown in \Cref{fig:test_CT_REAL}(a)--(b).

The parameter $\sigma$ used for the construction of the graph Laplacian is set to $\sigma=5\times10^{-3}$ for the standard iteration and to $\sigma=10^{-3}$ for the error-equation stage of the mixed scheme.
We apply the standard iterative scheme \cref{eq:standard-iter}, with the initial approximation $\bx_{0}^{\delta}$ computed via FBP. The standard iterative scheme is performed for two iterations. Subsequently, one error-equation step in the mixed scheme is applied, where the Laplacian is constructed using the back-projected residual obtained by FBP with Hann filtering.

The performance metrics reported in \Cref{Table: test_CT_REAL}, together with the restored images shown in \Cref{fig:test_CT_REAL}(c)--(e), support the effectiveness of the proposed iterative schemes. In particular, the close-up views in the insets of \Cref{fig:test_CT_REAL} highlight the ability of the error-equation stage to attenuate noise while preserving relevant structural details. With respect to graphLa\texttt{+}FBP ($\bx_1^\delta$), the RRE decreases from $0.1414$ to $0.1352$. As in Test~1, the error-equation step slightly lowers the SSIM of the whole image (from $0.7828$ to $0.7644$), while it markedly improves the GMSD (from $0.0586$ to $0.0393$); on the detail, it improves all the indices, SSIM included (from $0.7273$ to $0.7604$).

\begin{figure}
        \centering
        \newcommand{\ctpanel}[2]{%
          \setlength{\unitlength}{\linewidth}%
          \begin{picture}(1,1)
            \put(0,0){\includegraphics[width=\unitlength,height=\unitlength]{#1}}%
            \put(0.1914,0.1992){\color{yellow}\linethickness{0.5pt}\framebox(0.1602,0.1602){}}%
            \put(0.53,0.53){{\color{yellow}\setlength{\fboxsep}{0pt}\setlength{\fboxrule}{0.5pt}\fbox{\includegraphics[width=0.46\unitlength]{#2}}}}%
          \end{picture}}
		\begin{subfigure}{0.215\textwidth}
			\ctpanel{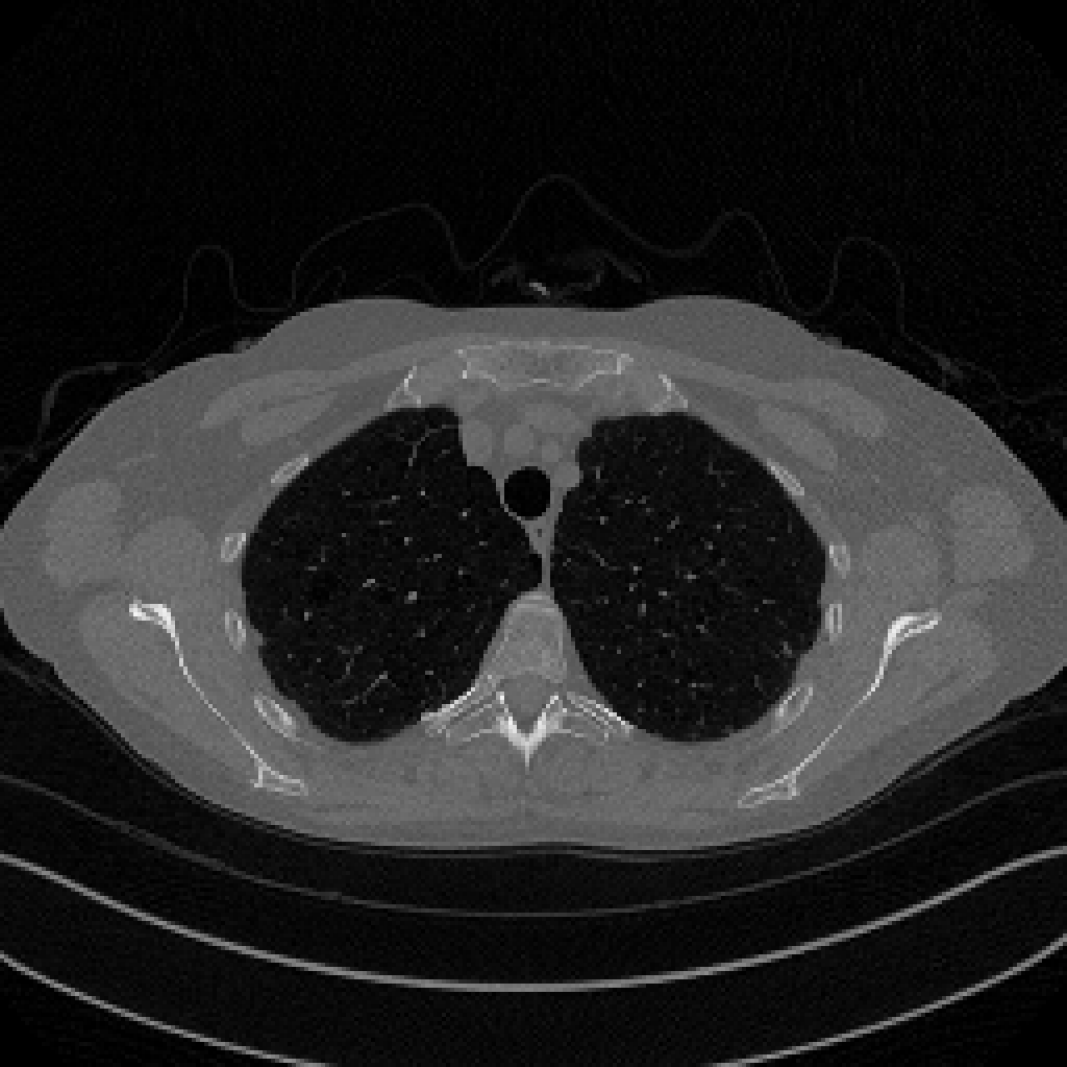}{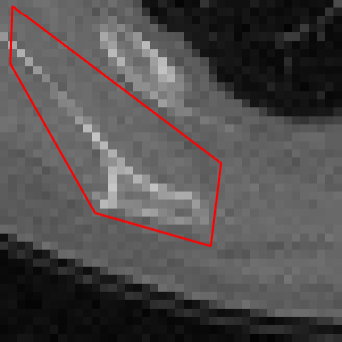}
			\caption{}
		\end{subfigure}\hfill
		\begin{subfigure}{0.10\textwidth}
			\includegraphics[clip,trim={2cm 1.5cm 2cm 1.5cm}, width=\textwidth, height=2.15\textwidth]{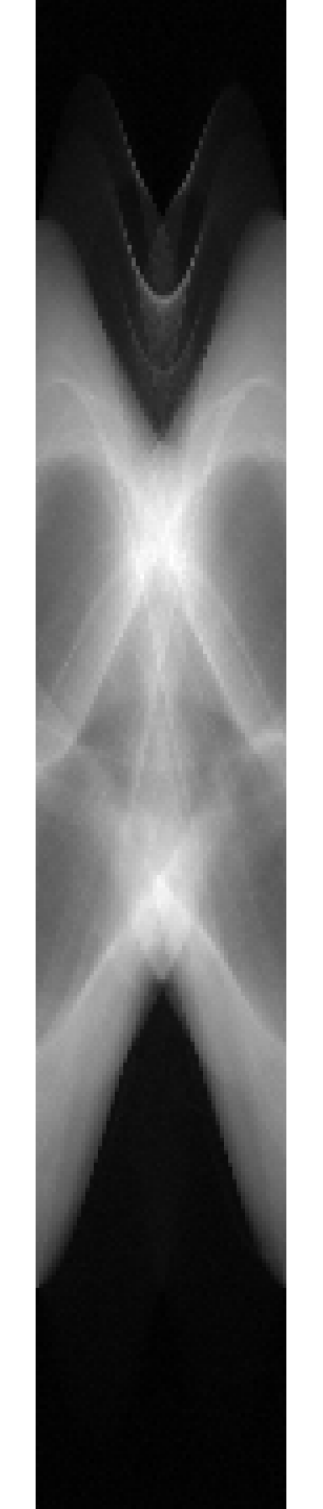}
			\caption{}
		\end{subfigure}\hfill
		\begin{subfigure}{0.215\textwidth}
			\ctpanel{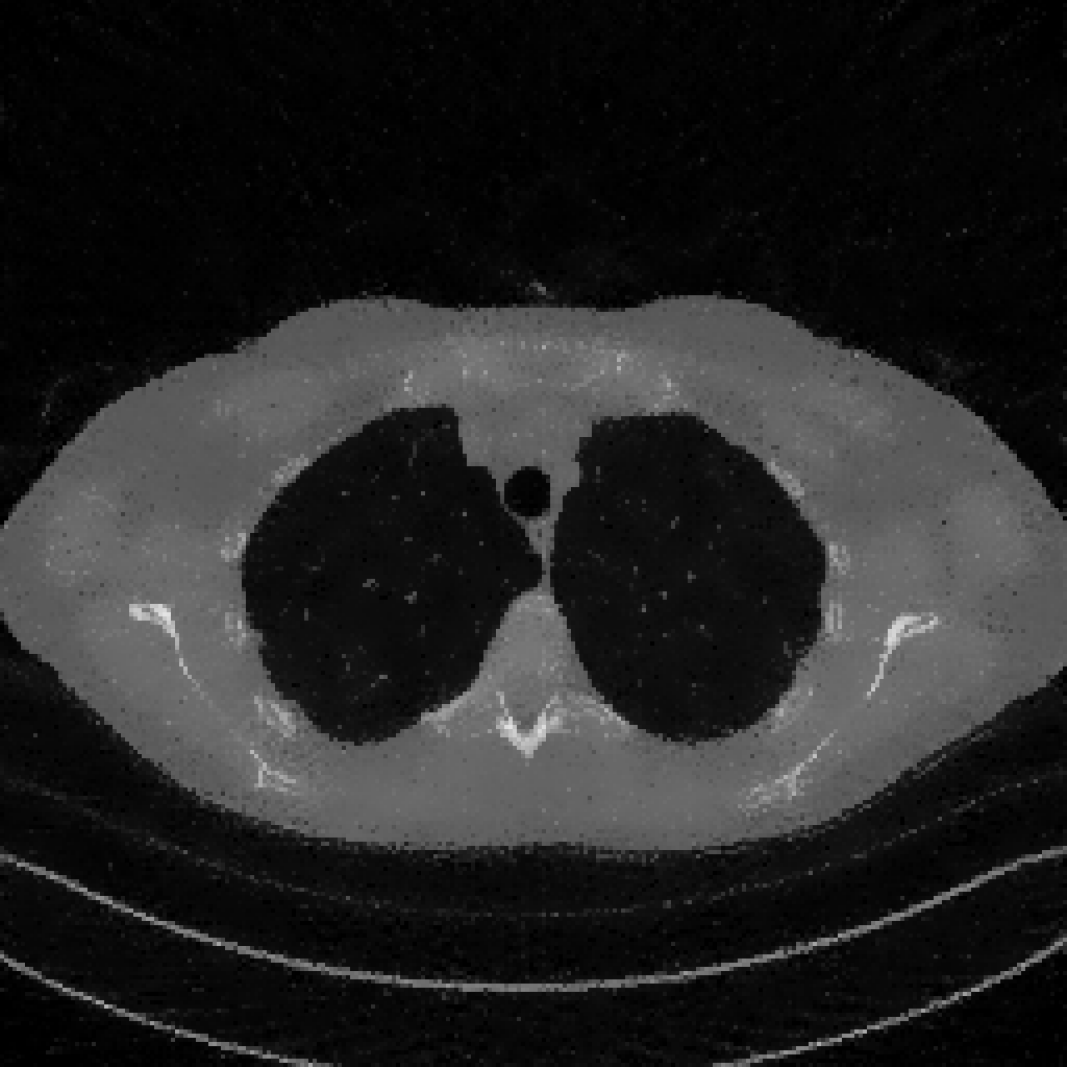}{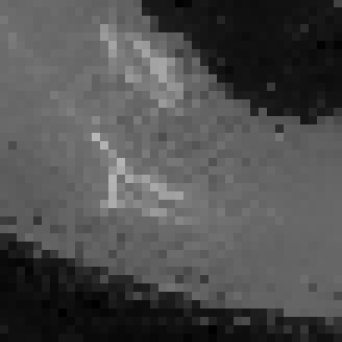}
			\caption{}
		\end{subfigure}\hfill
		\begin{subfigure}{0.215\textwidth}
			\ctpanel{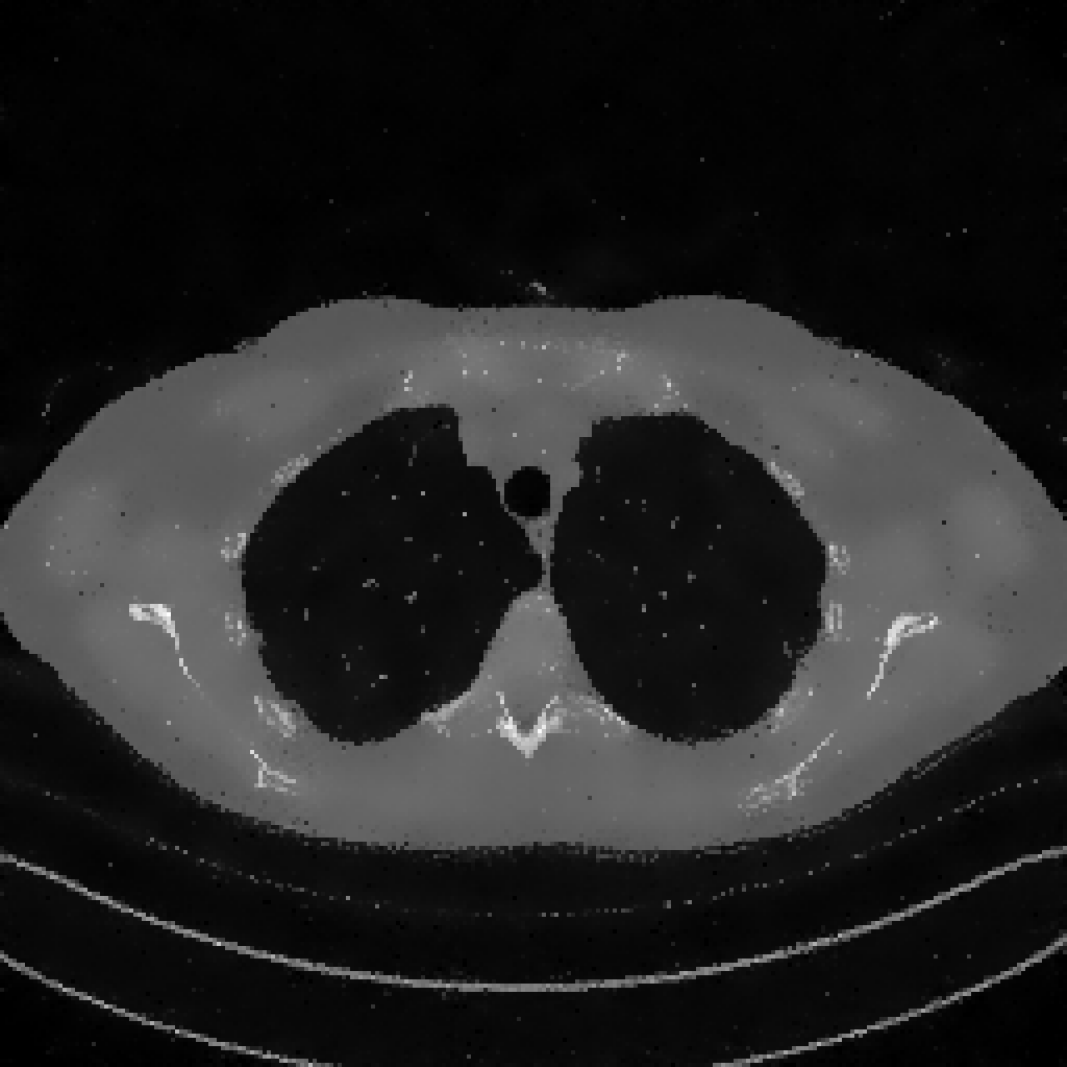}{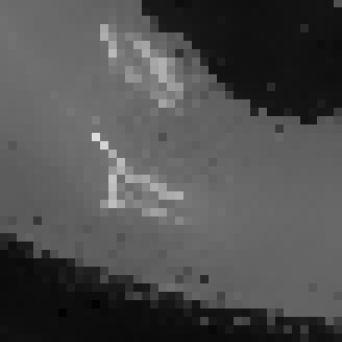}
			\caption{}
		\end{subfigure}\hfill
		\begin{subfigure}{0.215\textwidth}
			\ctpanel{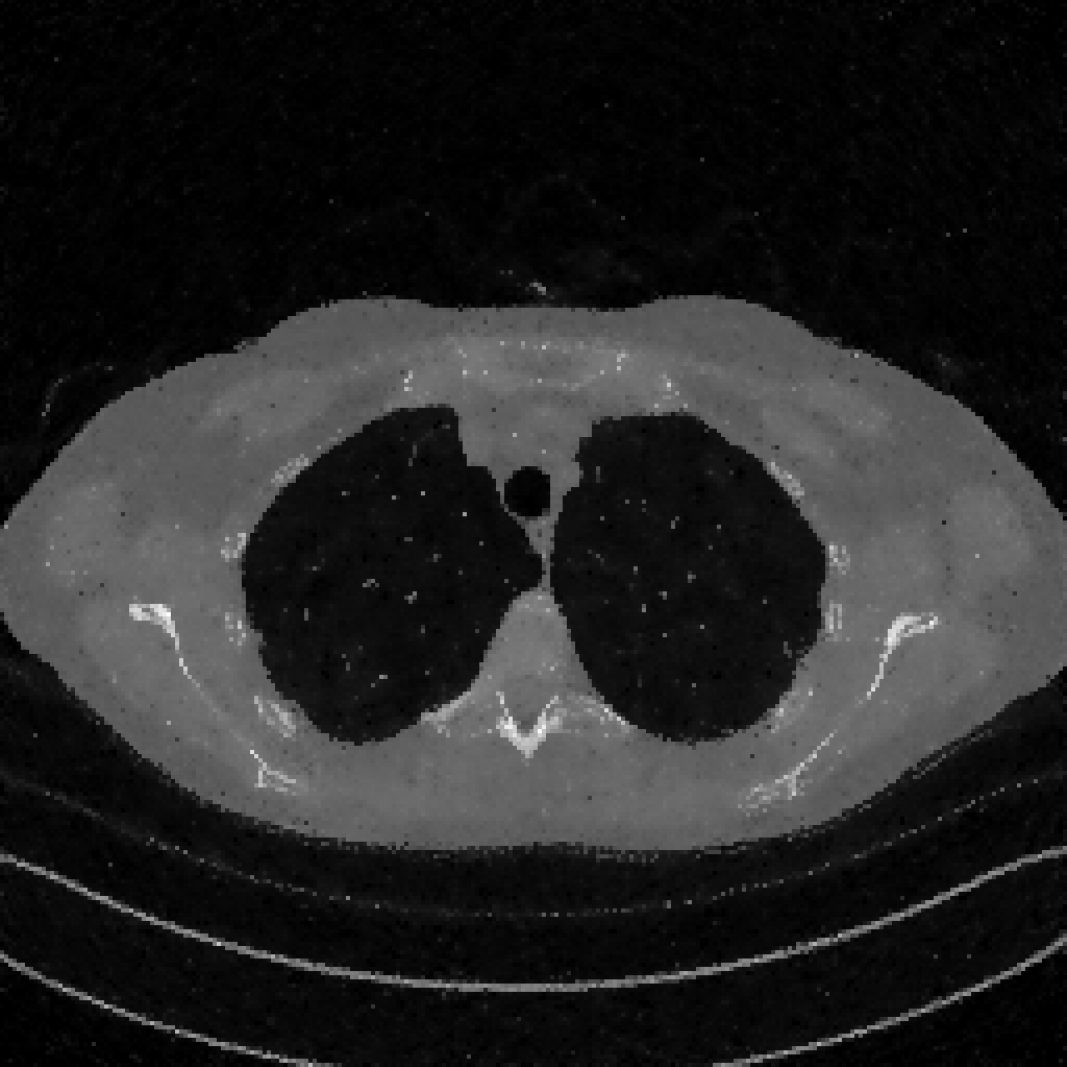}{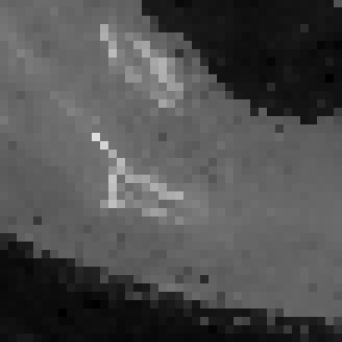}
			\caption{}
		\end{subfigure}
		\caption{Test 2: (a) original image, (b) observed sinogram corrupted by noise, (c) graphLa\texttt{+}FBP ($\bx_{1}^{\delta}$), (d) standard iterative scheme ($\bx_{n_{\mathrm{std}}}^{\delta}=\bx_{2}^{\delta}$), (e) final reconstruction with mixed iterative scheme ($\bx_{n_{\mathrm{std}}+1}^{\delta}$). In (a) and (c)--(e), the region in the yellow box is magnified in the inset. The detail scores, with ground truth in the inset of (a), are (c) PSNR = 25.7133, SSIM = 0.7033, GMSD = 0.0605; (d) PSNR = 25.7194, SSIM = 0.7273, GMSD = 0.0776; (e) PSNR = \textbf{26.2945}, SSIM = \textbf{0.7604}, GMSD~=~\textbf{0.0334}.}
		\label{fig:test_CT_REAL}
	\end{figure}

\begin{table}
\centering
	\begin{tabular}{c|c|c|c | c}
       iterate & RRE\(\downarrow \) & PSNR\(\uparrow\) & SSIM\(\uparrow\) & GMSD\(\downarrow \) \\ \hline
		 $\bx_{0}^{\delta}$ & 0.3373 & 22.41 & 0.3389 & 0.2155 \\
    	$\bx_{1}^{\delta}$  & 0.1414 & 29.96 & 0.7532 & 0.0538 \\ 
		$\bx_{n_{\mathrm{std}}}^{\delta}$  & 0.1356 & 30.32 & \textbf{0.7828} & 0.0586 \\ \hline
		$\bx_{n_{\mathrm{std}}+1}^{\delta}$  & \textbf{0.1352} & \textbf{30.35} & 0.7644 & \textbf{0.0393} \\
	\end{tabular}
	\caption{Performance metrics for Test 2. The upper part is the standard iterative scheme with $\bx_{0}^{\delta}$ computed by FBP. Below the horizontal line is the error-equation stage of the mixed iterative scheme, started from $\bx_{n_{\mathrm{std}}}^{\delta}=\bx_{2}^{\delta}$.}
    \label{Table: test_CT_REAL}
\end{table}

\subsection[Deblurring and denoising with neural network initialization]{Test 3: deblurring and denoising for $\bx_0^\delta$ computed by a neural network}

In this test, we use a deep Wiener deconvolution network \cite{dong2020deep} to compute the initial reconstruction $\bx_0^\delta$. We show that, even when using a state-of-the-art reconstruction method, the iterated graph Laplacian approach is still able to further improve the reconstruction quality. The results are compared with those of ADMM-BM3D and of IRMGL, started from $K^*\by^\delta$ as in Test~1, both with oracle parameter choices.

The test images are grayscale images with intensity values scaled to the interval $[0,1]$. They are blurred using a $17\times17$ Gaussian kernel with standard deviation $5$, and additive Gaussian noise with standard deviation $0.01$ is subsequently added. For all iterations, we set $R=4$ and $\sigma=0.002$.

The mixed iterative scheme is then applied, consisting of $n_{\mathrm{std}}=2$ standard iterations followed by $3$ error-equation iterations.

The original image, the blurred image, $\bx_0^\delta$ (the initial approximation), $\bx_{n_{\mathrm{std}}}^{\delta}$ (the reconstruction after the standard iterative scheme), and $\bx_{n_{\mathrm{std}}+3}^{\delta}$ (the final reconstruction after the mixed scheme) are shown in \Cref{fig:network_test_zoom_house} (``House front'' image) and \cref{fig:network_test_zoom_window} (``Window with flowers'' image), together with the reconstructions computed by IRMGL and ADMM-BM3D and with the final reconstruction obtained when our scheme is initialized by ADMM-BM3D. In each panel, a zoomed-in region is shown in the inset for a more detailed comparison. The corresponding performance metrics of all methods are reported in \Cref{Table:network_test}.

The iterated graph Laplacian improves all four metrics (RRE, PSNR, SSIM, and GMSD) for both images compared to the initial neural network reconstruction. Although the mixed scheme provides only modest gains for the House front image, it leads to a substantial improvement for the Window with flowers image. The benefits are particularly pronounced in the recovery of structural details, as shown in \Cref{fig:network_test_zoom_window} and reflected by the metrics reported in the caption.

The IRMGL reconstruction attains about the same quality as our neural network initialization and thus is outperformed by our final reconstruction. ADMM-BM3D, on the other hand, reaches a reconstruction result far better than the initial guess. The performance metrics indicate that our final reconstruction and ADMM-BM3D are on the same level for the House front image and ADMM-BM3D is slightly better on the Window with flowers image. However, this result was achieved by choosing the optimal parameters for ADMM-BM3D under oracle knowledge of the ground truth, which cannot be achieved in real-world applications. Thus our proposed approach offers an alternative with a straightforward parameter choice. In cases where a good parameter choice for ADMM-BM3D is given, the iterated graph Laplacian can still be used to improve on the returned reconstruction, as the whole-image metrics in \cref{Table:network_test} (last row) show.  Here, the ADMM-BM3D reconstruction was used as the initial reconstruction of the mixed scheme. For the parameters $R$, $\sigma$, and the number of iterations we simply applied the same setup as before without any optimization.

\begin{figure}
    \centering
    \newcommand{\netpanel}[7]{%
      \setlength{\unitlength}{\linewidth}%
      \begin{picture}(1,0.6664)
        \put(0,0){\includegraphics[width=\unitlength,height=0.6664\unitlength]{figures/Numerical_Tests/Network_Deblurring/#1.png}}%
        \put(#2,#3){\color{yellow}\linethickness{0.5pt}\framebox(#4,#5){}}%
        \put(#6,#7){{\color{yellow}\setlength{\fboxsep}{0pt}\setlength{\fboxrule}{0.5pt}\fbox{\includegraphics[width=0.45\unitlength]{figures/Numerical_Tests/Network_Deblurring/#1_zoom.png}}}}%
      \end{picture}}
    \newcommand{\housepanel}[1]{\netpanel{#1}{0.5983}{0.1202}{0.1562}{0.1429}{0.01}{0.244}}
    \begin{subfigure}{0.24\textwidth}
    \housepanel{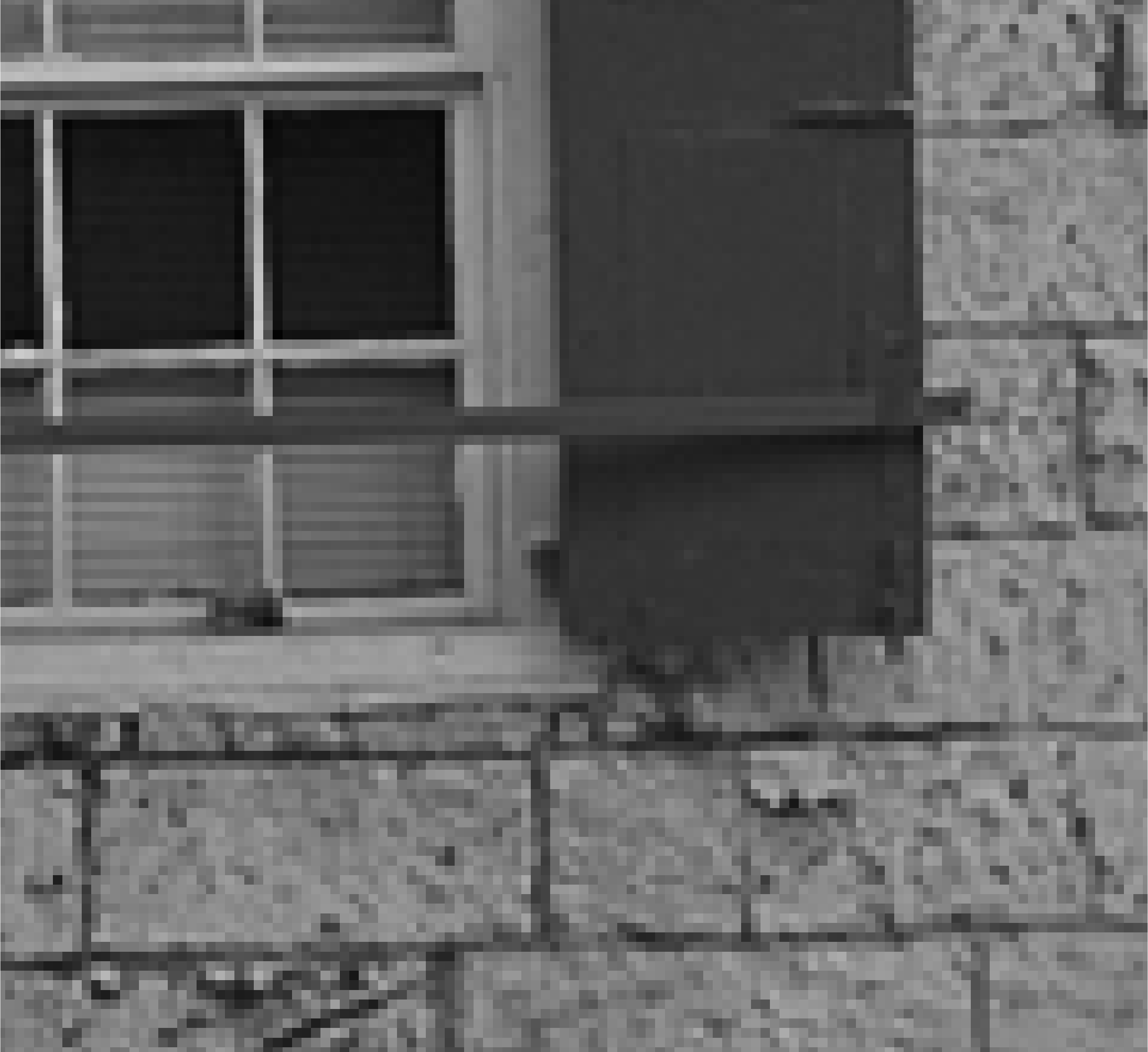}
    \caption{}
    \end{subfigure}
    \begin{subfigure}{0.24\textwidth}
    \housepanel{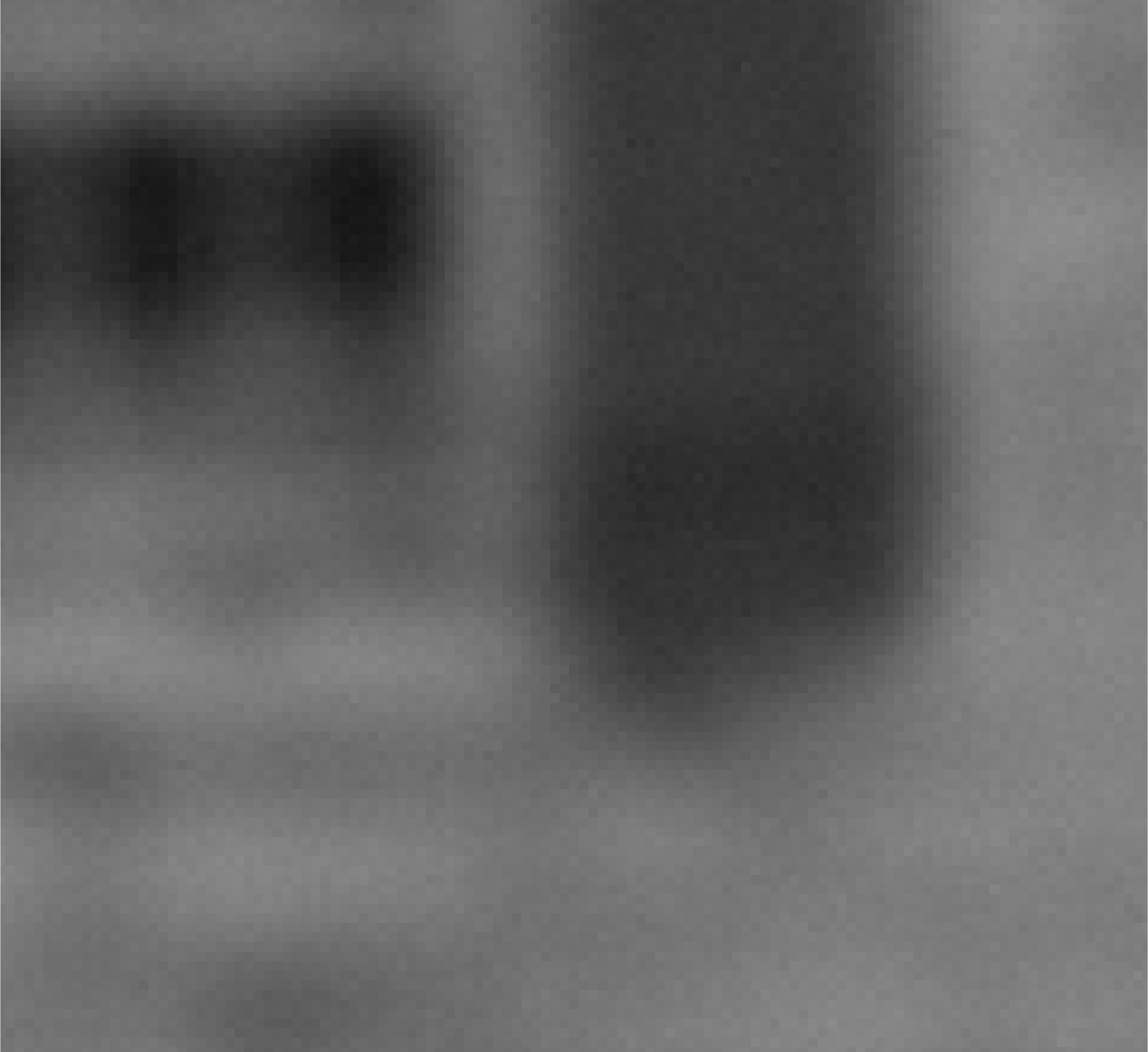}
    \caption{}
    \end{subfigure}
    \begin{subfigure}{0.24\textwidth}
    \housepanel{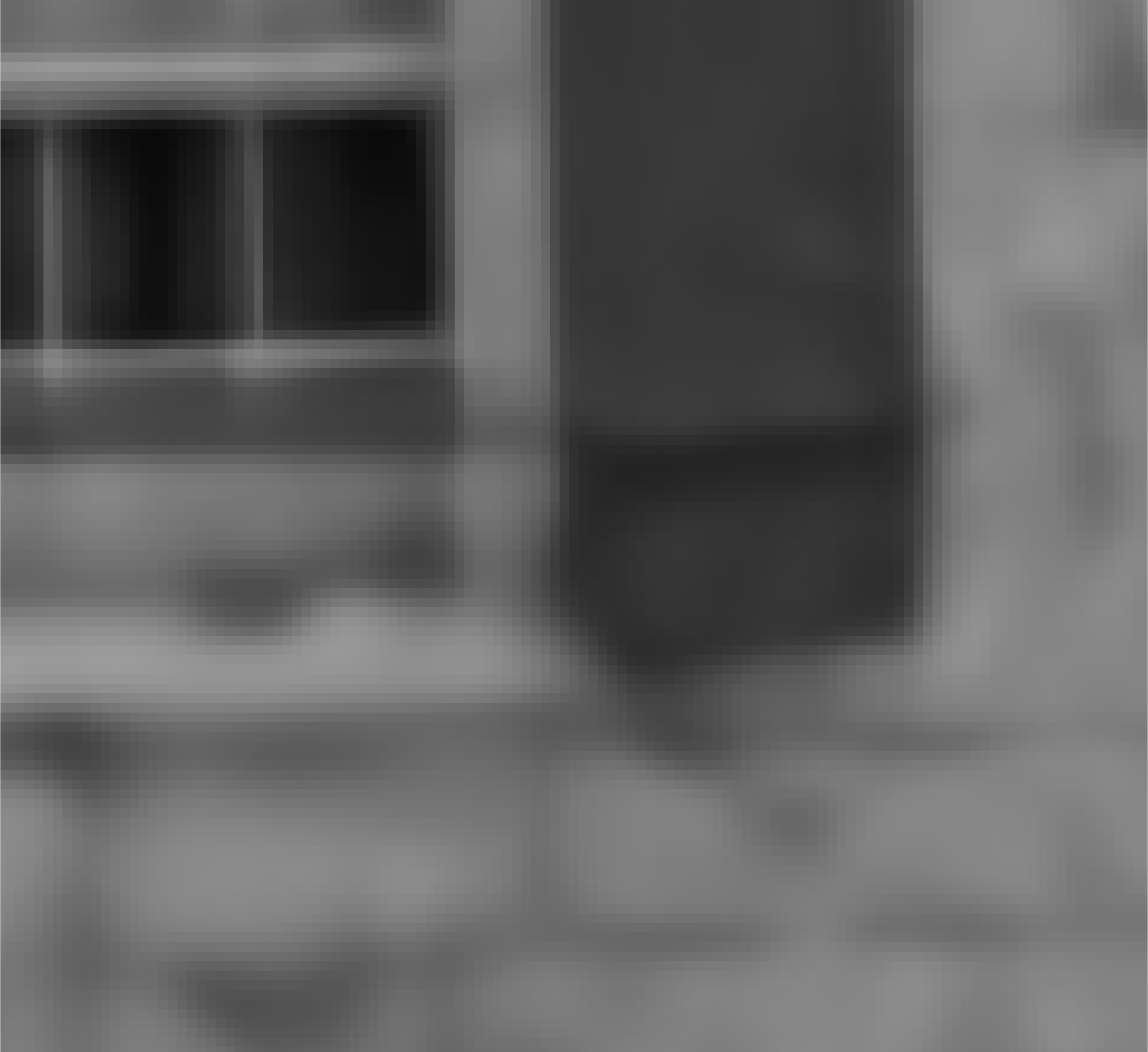}
    \caption{}
    \end{subfigure}
    \begin{subfigure}{0.24\textwidth}
    \housepanel{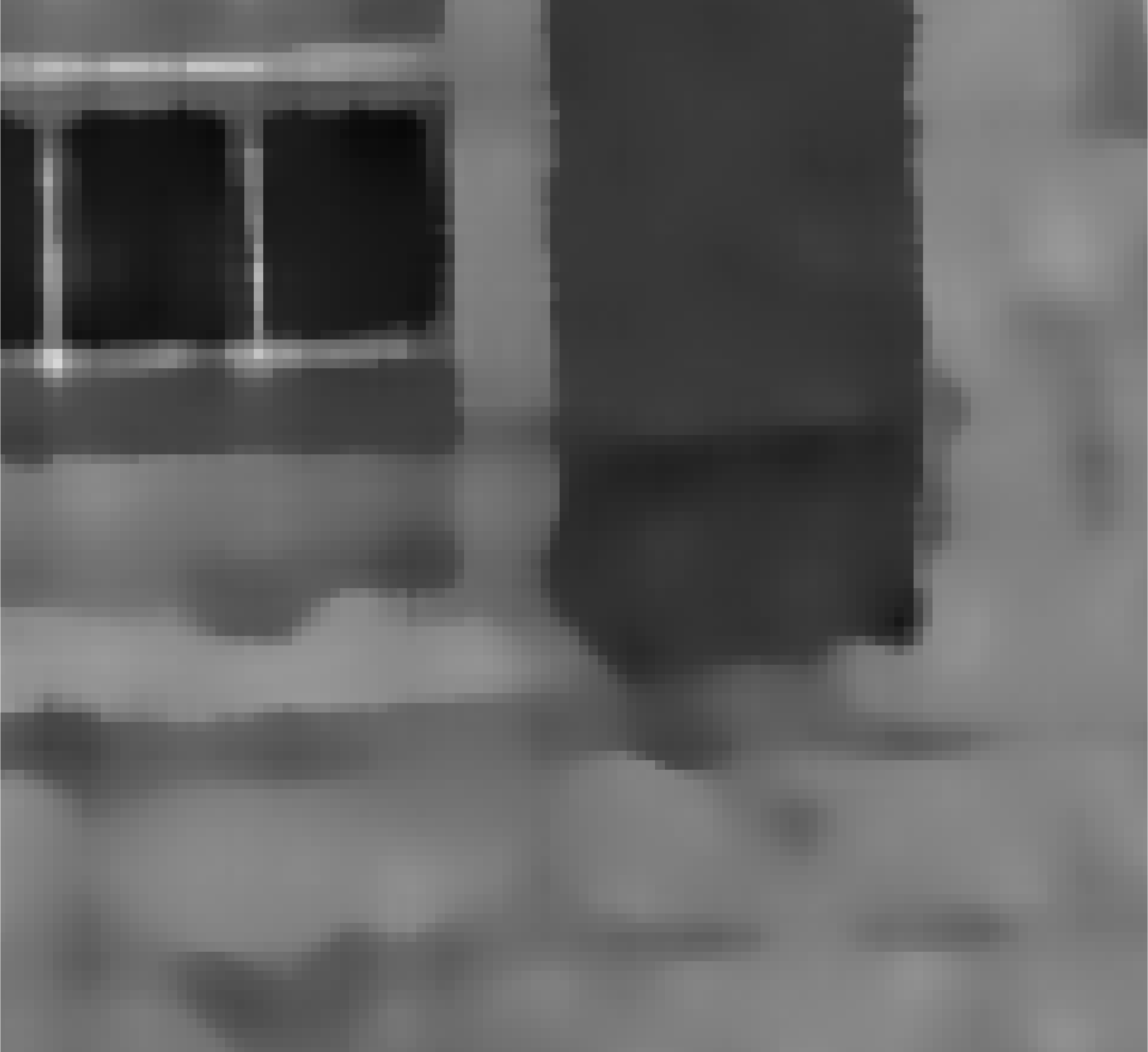}
    \caption{}
    \end{subfigure}
    \begin{subfigure}{0.24\textwidth}
    \housepanel{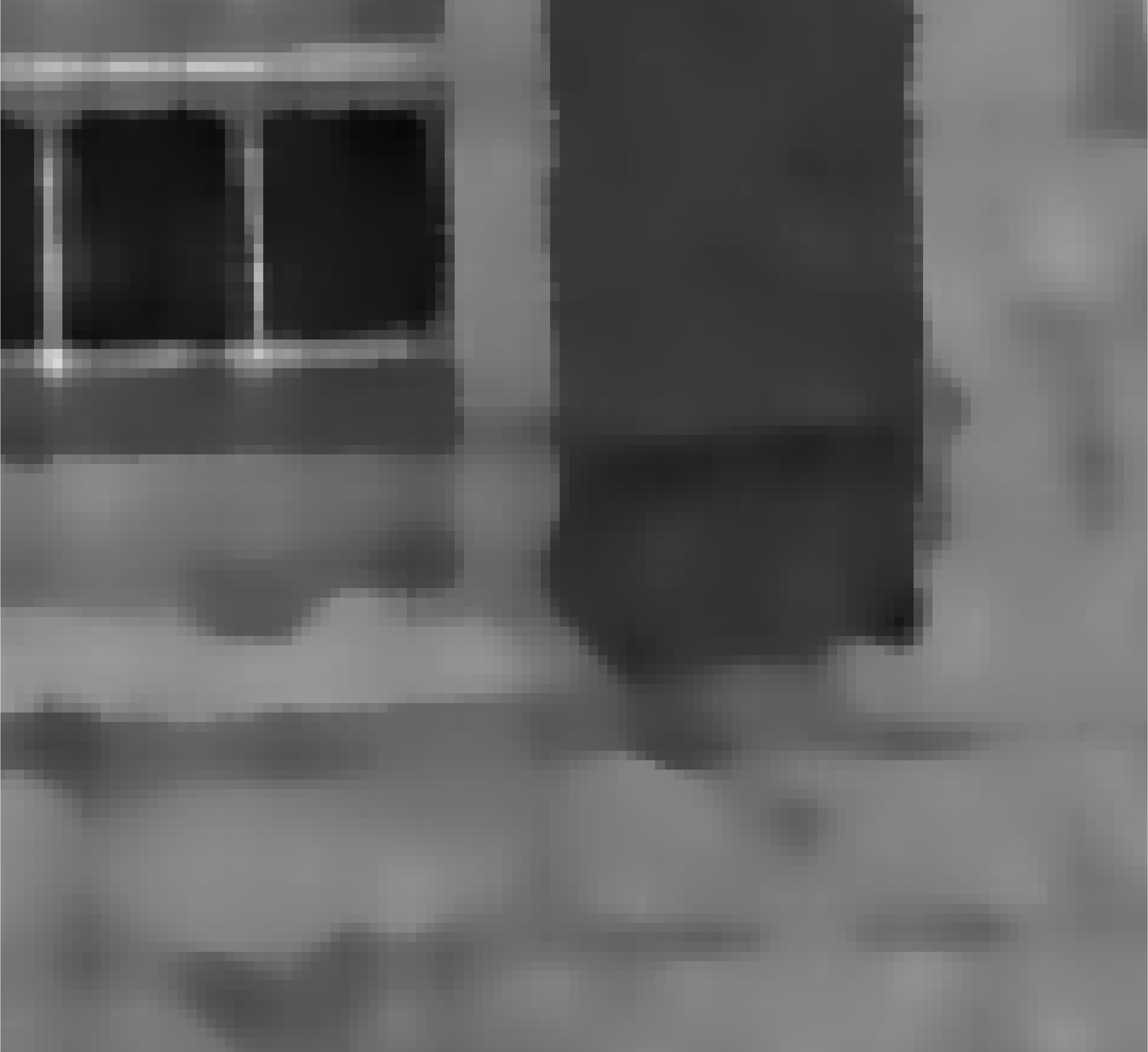}
    \caption{}
    \end{subfigure}
    \begin{subfigure}{0.24\textwidth}
    \housepanel{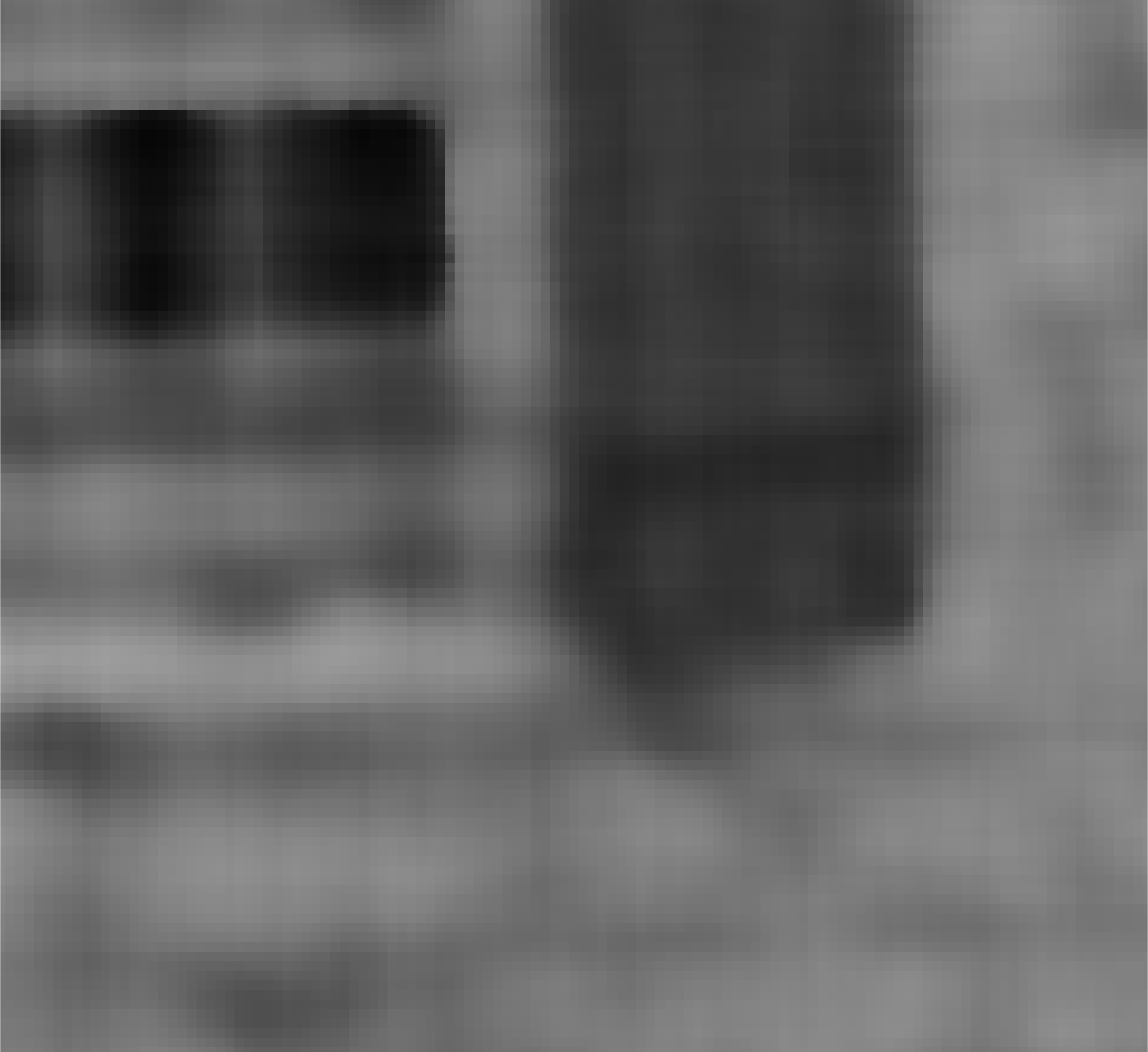}
    \caption{}
    \end{subfigure}
    \begin{subfigure}{0.24\textwidth}
    \housepanel{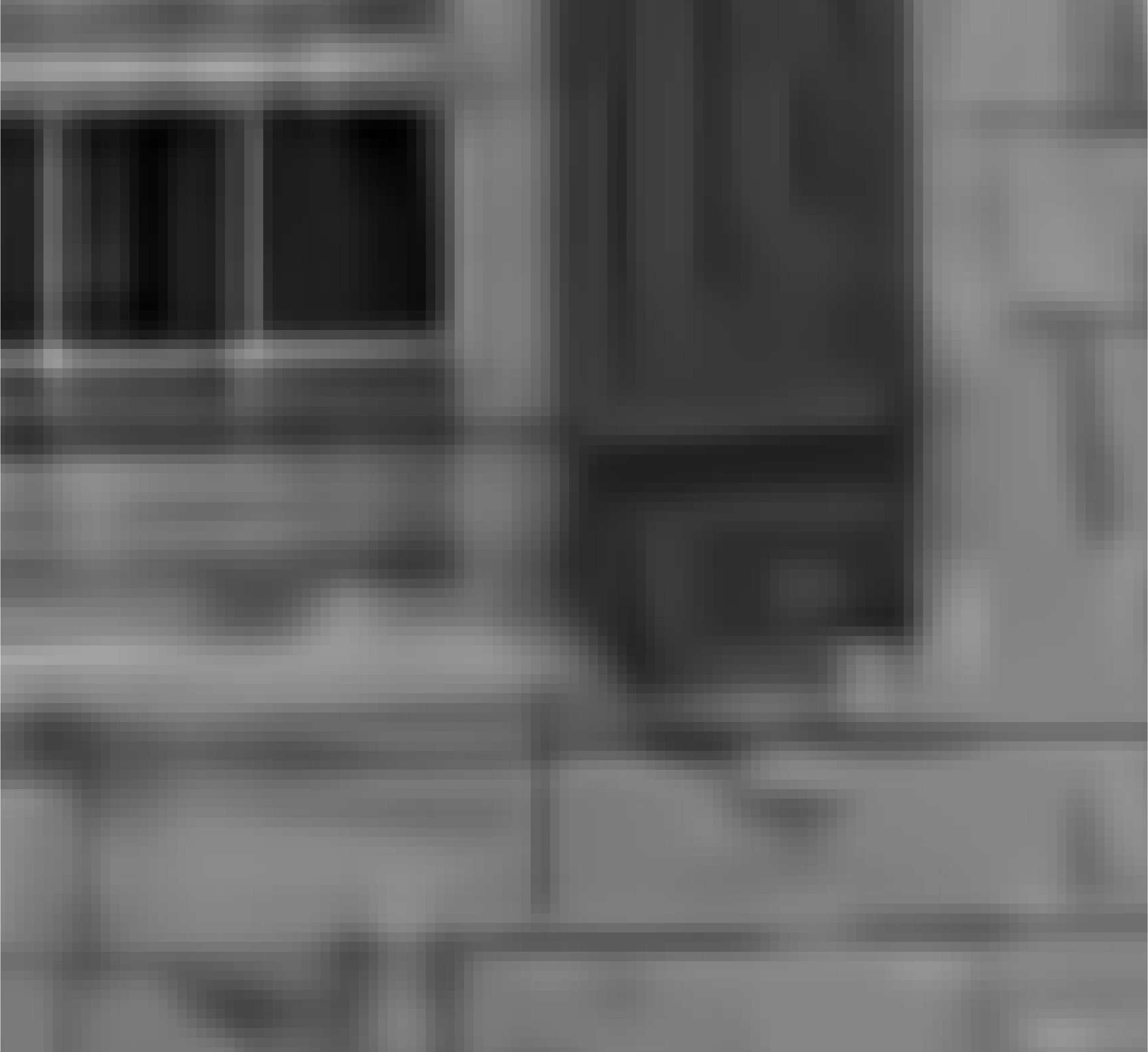}
    \caption{}
    \end{subfigure}
    \begin{subfigure}{0.24\textwidth}
    \housepanel{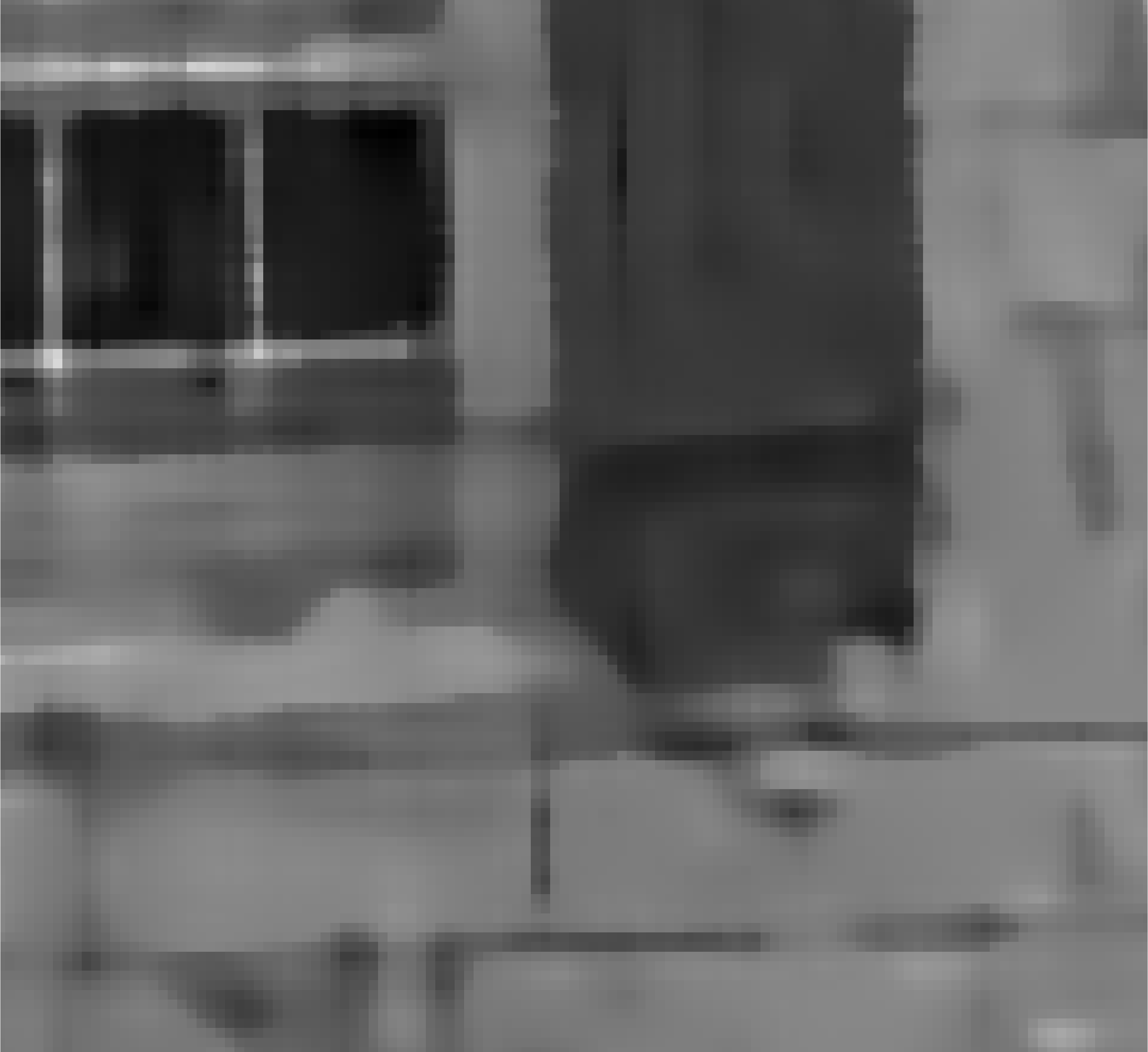}
    \caption{}
    \end{subfigure}
\caption{Image deblurring and denoising of ``House front'' with neural network initialization: (a) original image, (b) blurred and noisy image (PSNR 19.97, SSIM 0.3512), (c) network reconstruction (PSNR 21.91, SSIM 0.4860), (d) standard iterative scheme (PSNR 22.70, SSIM 0.5313), (e) final reconstruction (PSNR 22.71, SSIM 0.5325), (f) IRMGL (PSNR 21.31, SSIM 0.4477), (g) ADMM-BM3D (PSNR 22.35, SSIM 0.5249), (h) final reconstruction of the mixed scheme initialized by ADMM-BM3D (PSNR \textbf{23.05}, SSIM \textbf{0.5699}). In each panel, the region in the yellow box is magnified in the inset;  the values of PSNR and SSIM refer to this region.}
\label{fig:network_test_zoom_house}
\end{figure}

\begin{figure}
    \centering
    \newcommand{\netpanel}[7]{%
      \setlength{\unitlength}{\linewidth}%
      \begin{picture}(1,0.6664)
        \put(0,0){\includegraphics[width=\unitlength,height=0.6664\unitlength]{figures/Numerical_Tests/Network_Deblurring/#1.png}}%
        \put(#2,#3){\color{yellow}\linethickness{0.5pt}\framebox(#4,#5){}}%
        \put(#6,#7){{\color{yellow}\setlength{\fboxsep}{0pt}\setlength{\fboxrule}{0.5pt}\fbox{\includegraphics[width=0.45\unitlength]{figures/Numerical_Tests/Network_Deblurring/#1_zoom.png}}}}%
      \end{picture}}
    \newcommand{\windowpanel}[1]{\netpanel{#1}{0.2077}{0.2774}{0.1562}{0.1429}{0.54}{0.244}}
    \begin{subfigure}{0.24\textwidth}
    \windowpanel{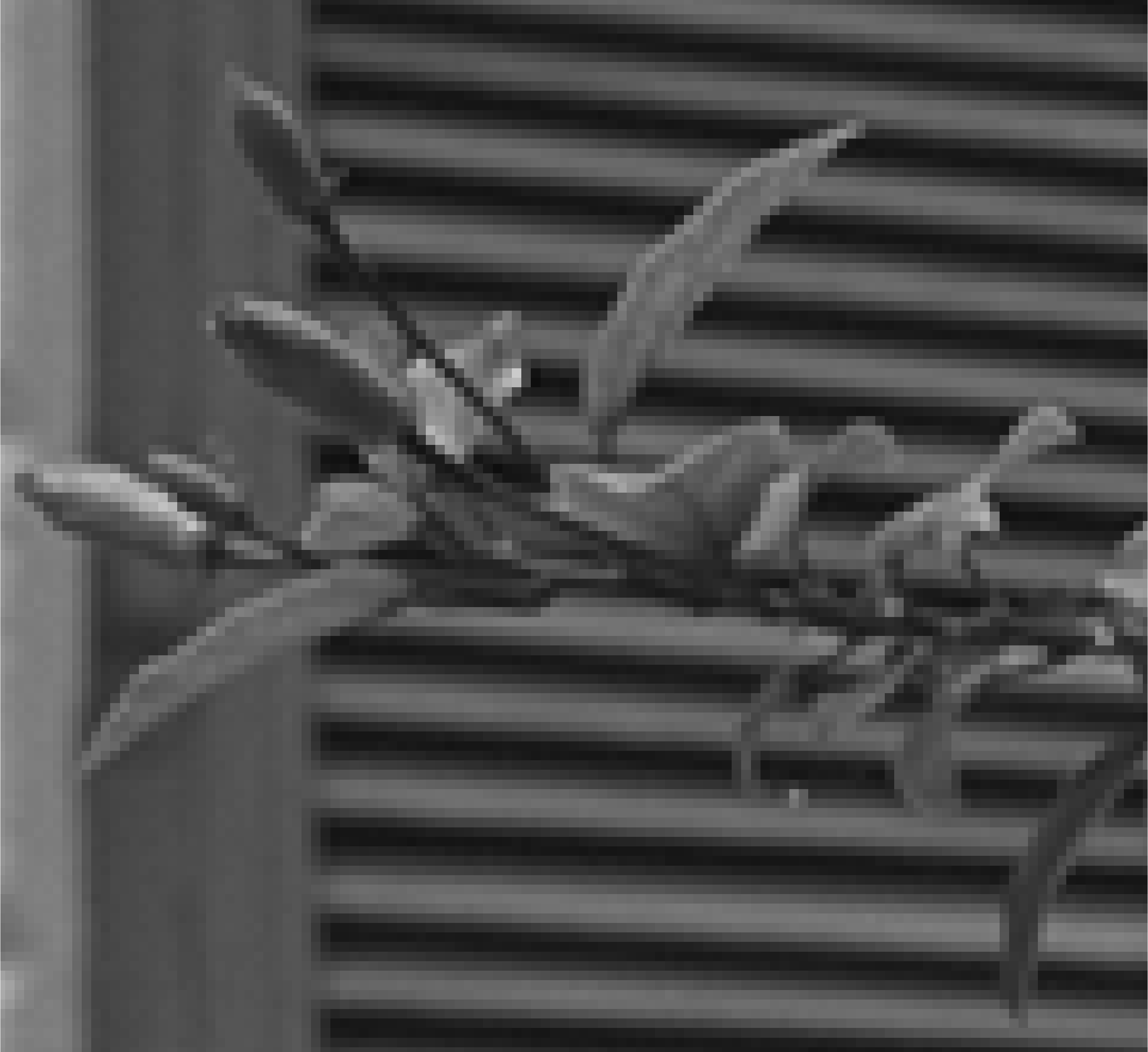}
    \caption{}
    \end{subfigure}
    \begin{subfigure}{0.24\textwidth}
    \windowpanel{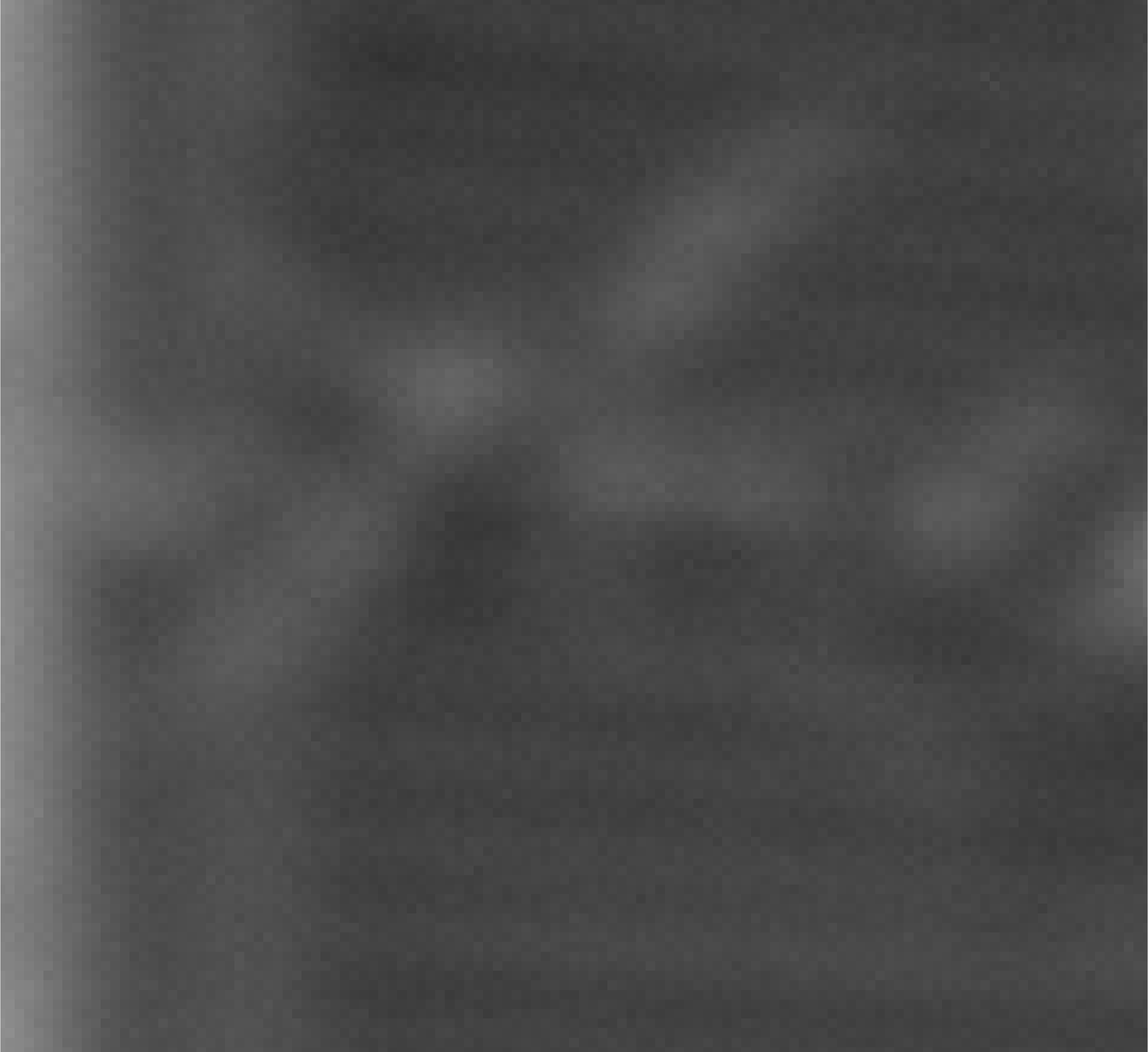}
    \caption{}
    \end{subfigure}
    \begin{subfigure}{0.24\textwidth}
    \windowpanel{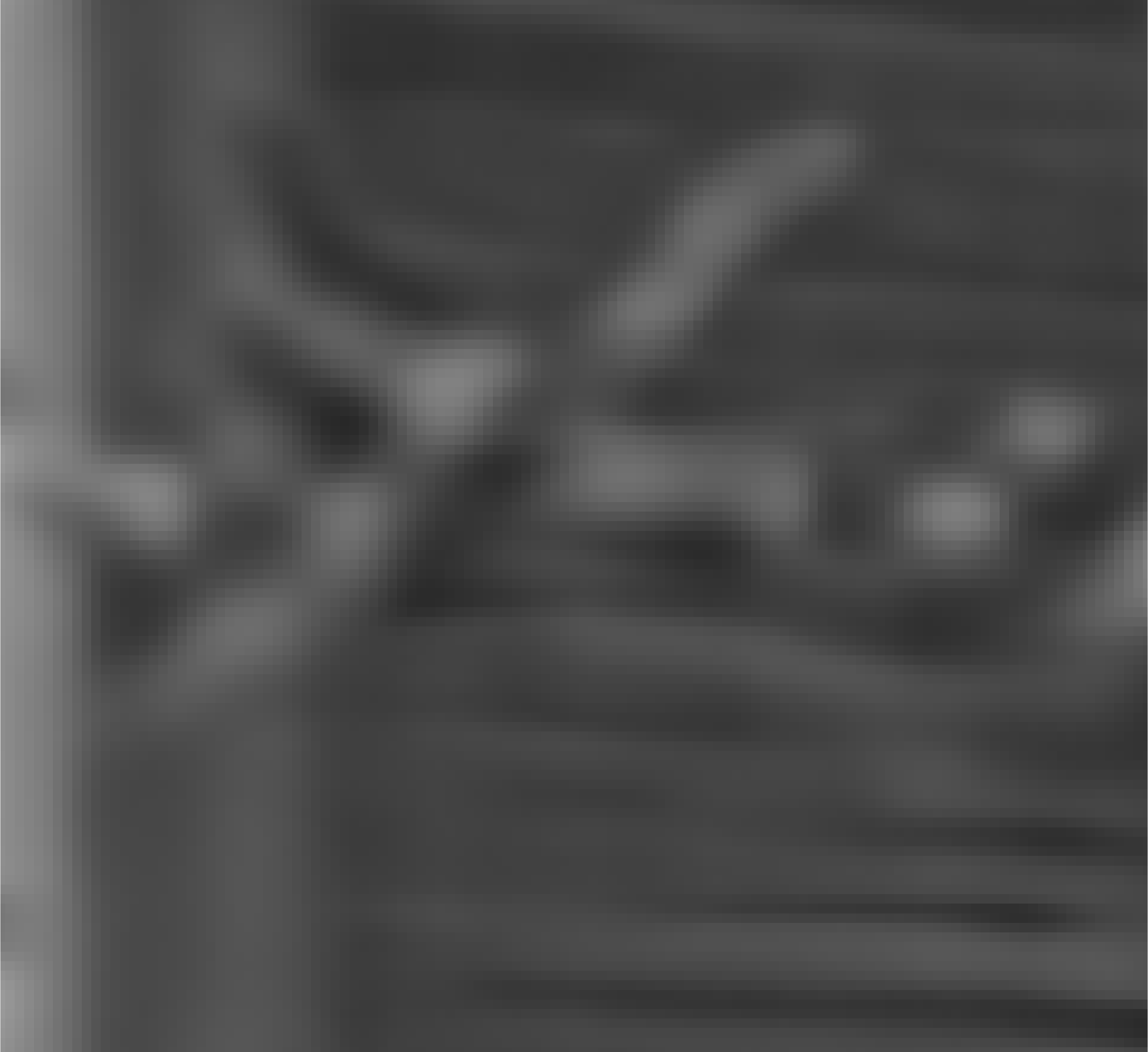}
    \caption{}
    \end{subfigure}
    \begin{subfigure}{0.24\textwidth}
    \windowpanel{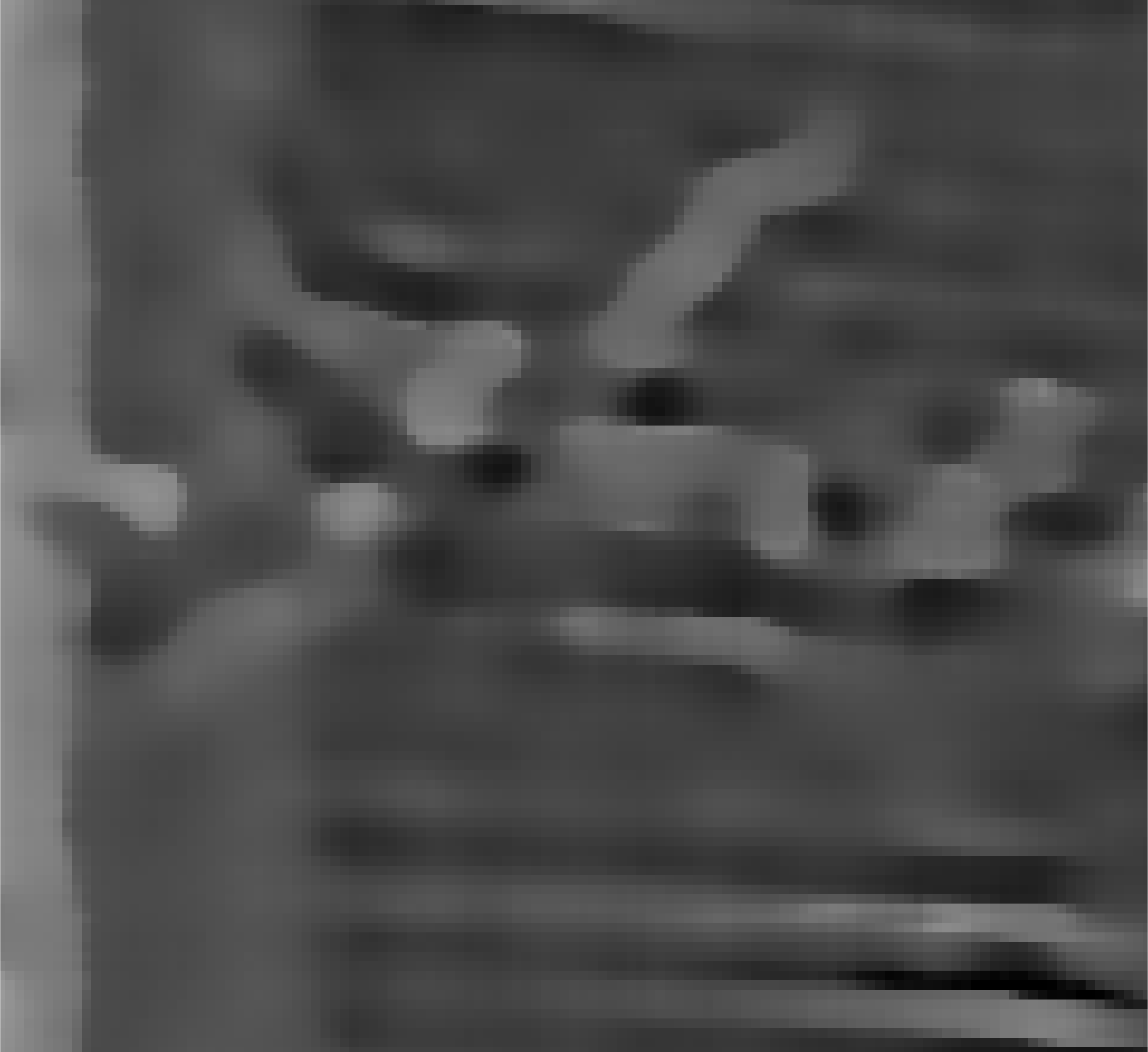}
    \caption{}
    \end{subfigure}
    \begin{subfigure}{0.24\textwidth}
    \windowpanel{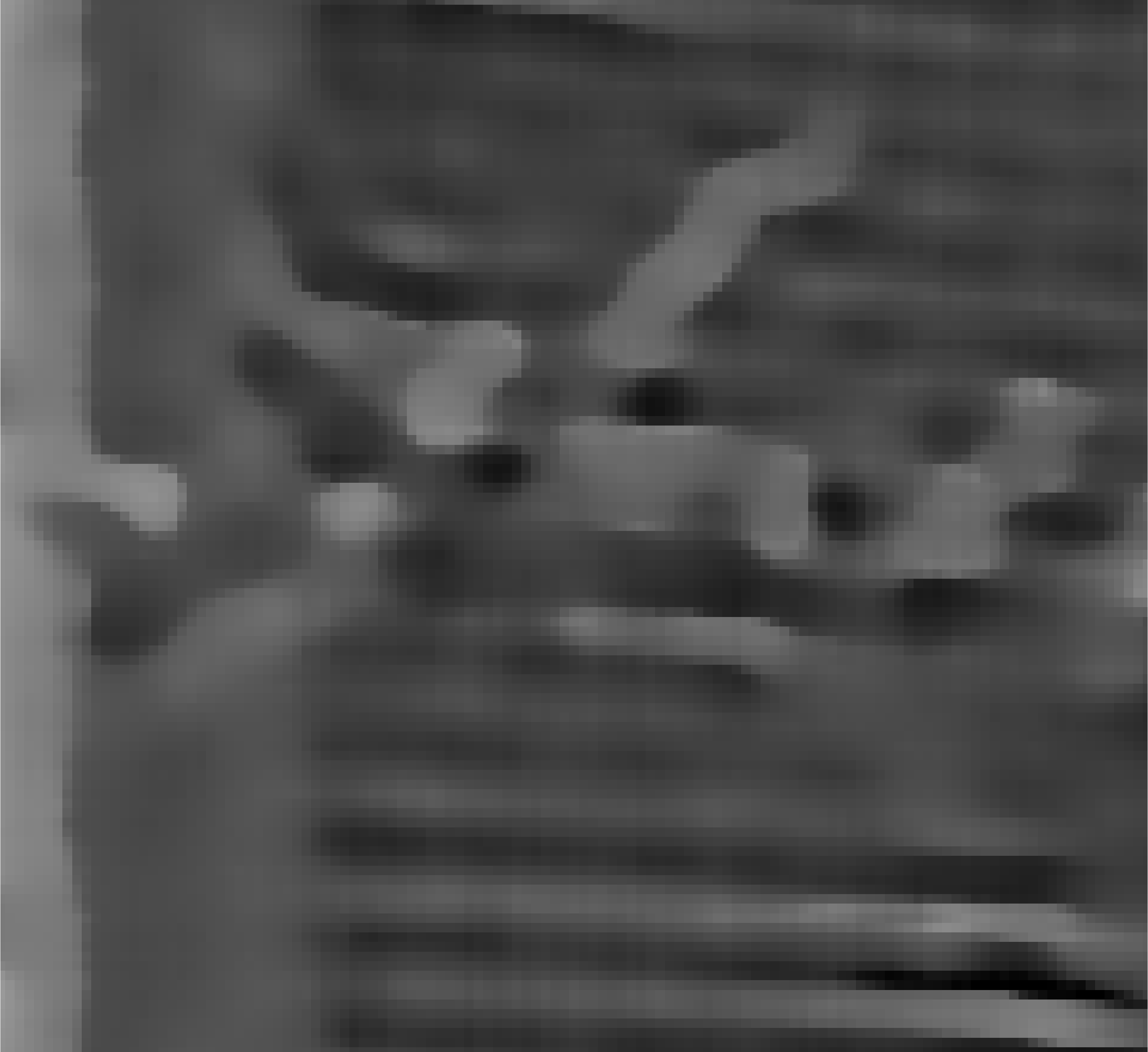}
    \caption{}
    \end{subfigure}
    \begin{subfigure}{0.24\textwidth}
    \windowpanel{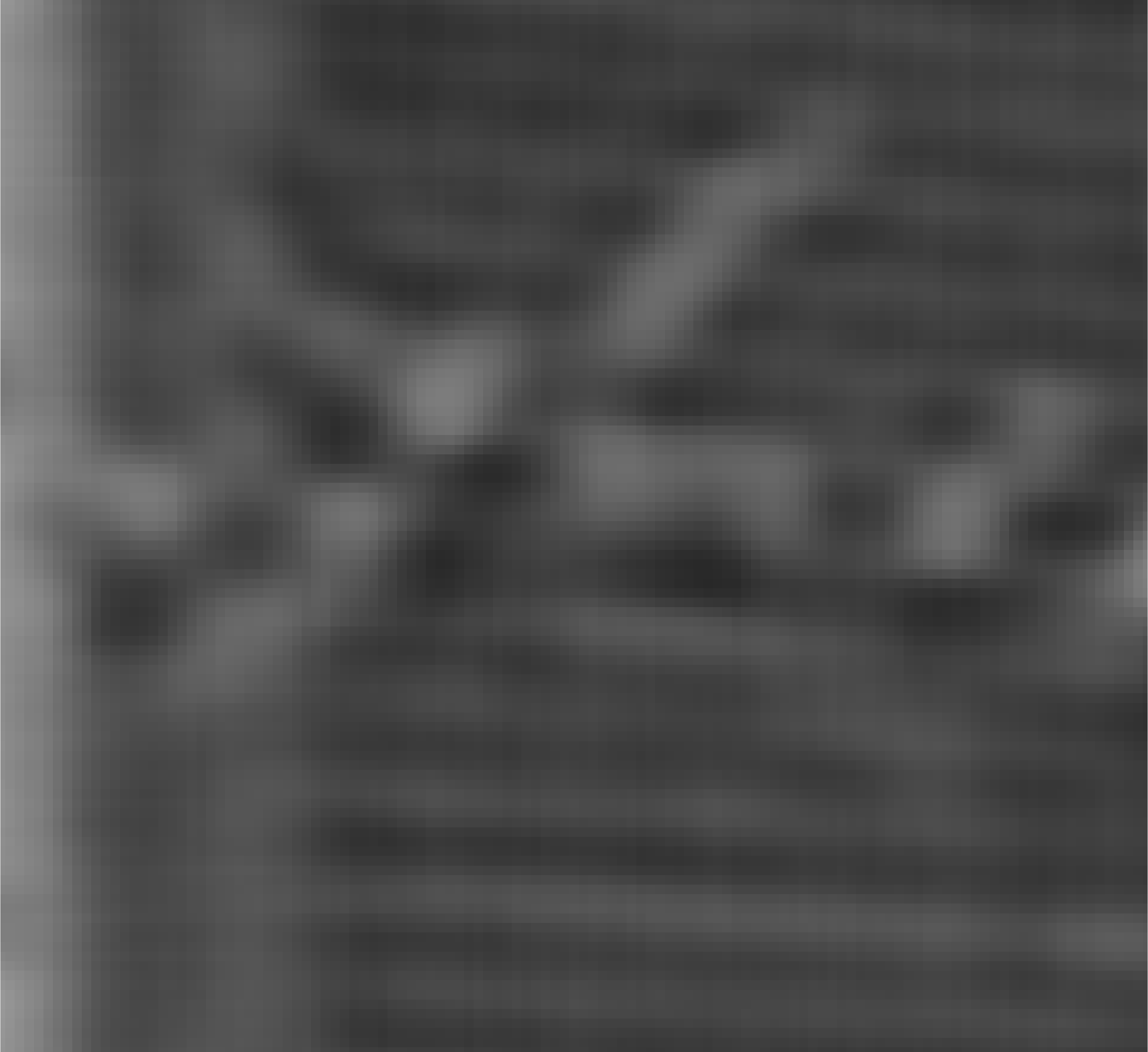}
    \caption{}
    \end{subfigure}
    \begin{subfigure}{0.24\textwidth}
    \windowpanel{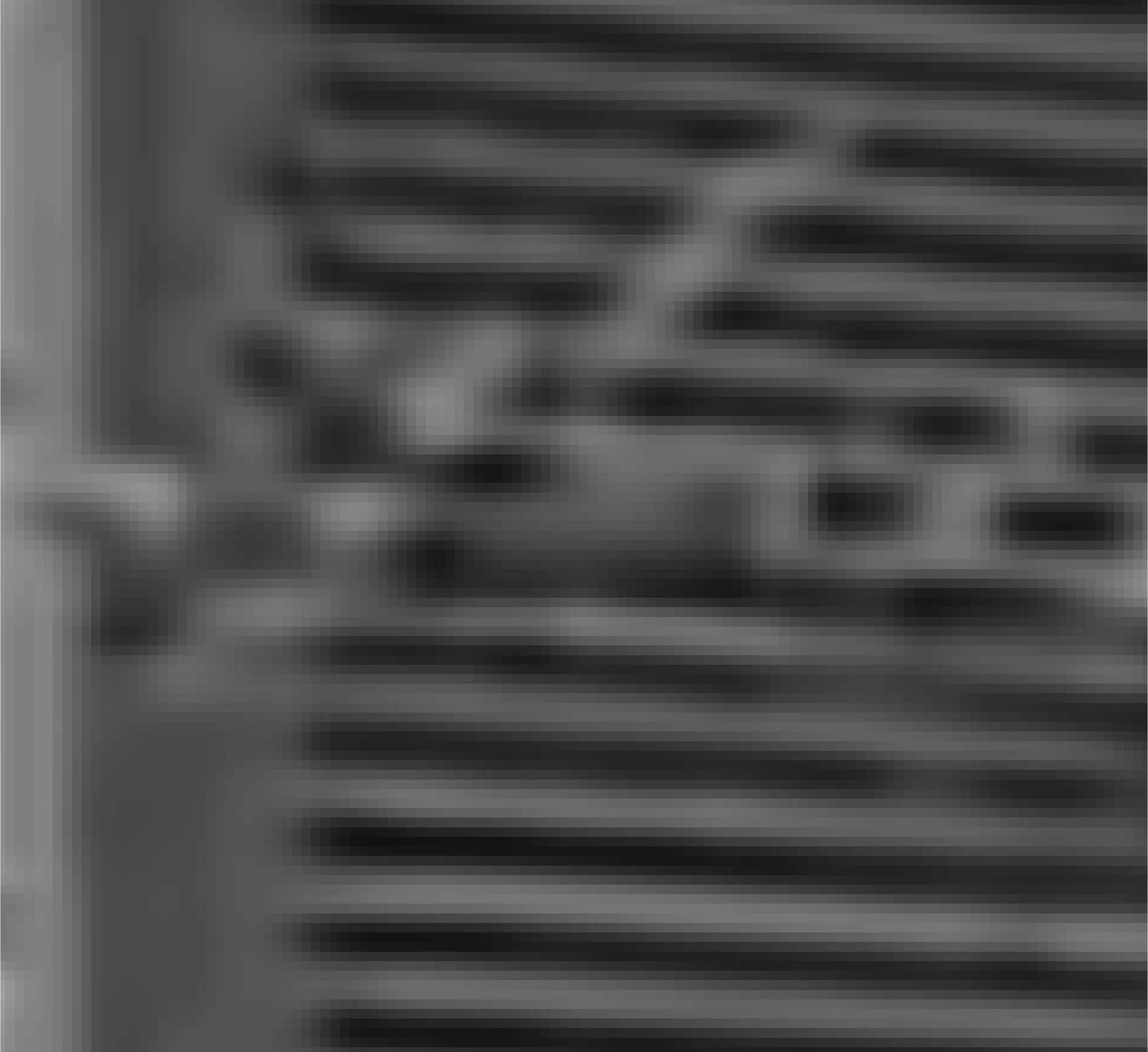}
    \caption{}
    \end{subfigure}
    \begin{subfigure}{0.24\textwidth}
    \windowpanel{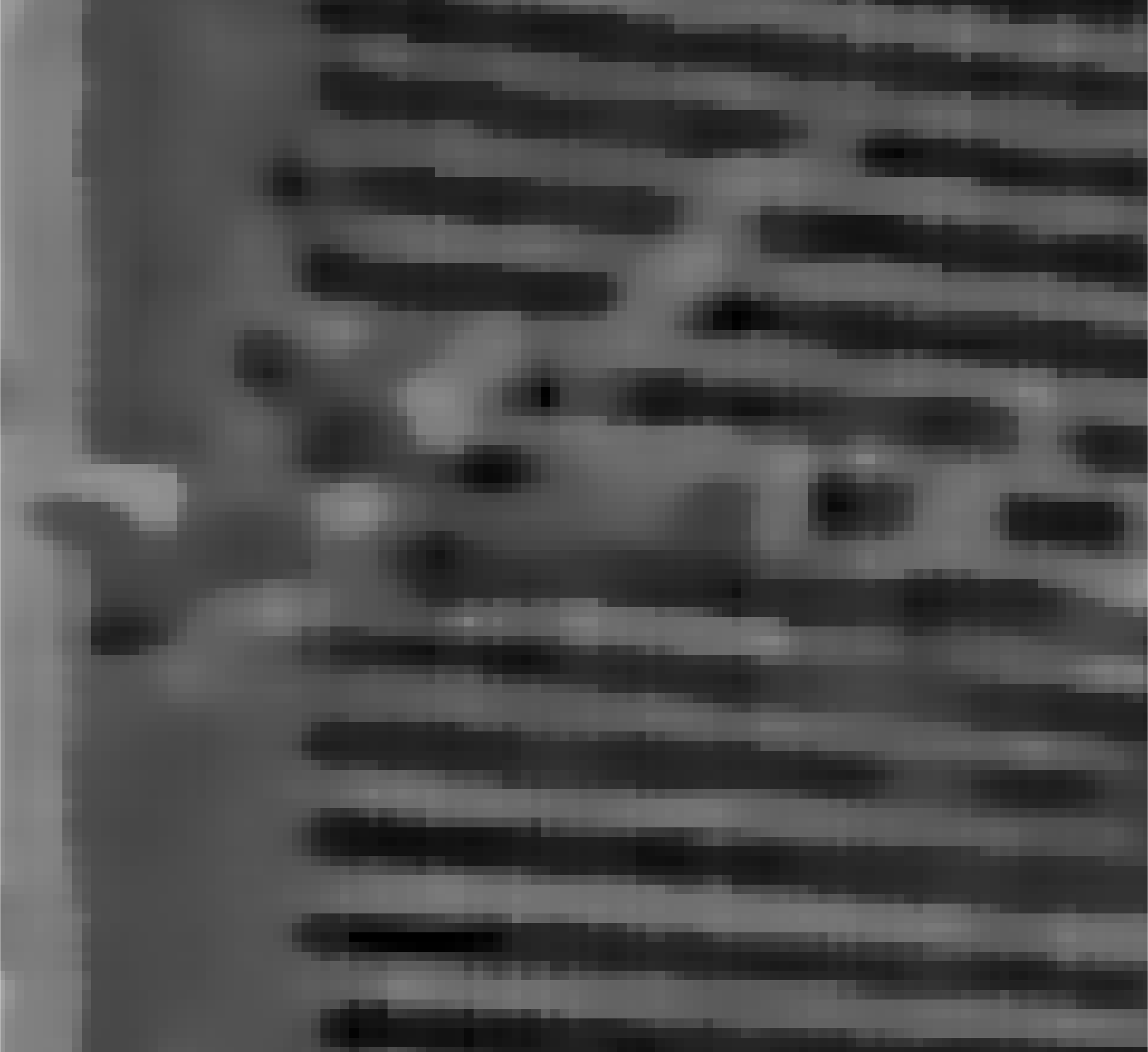}
    \caption{}
    \end{subfigure}
\caption{Image deblurring and denoising of ``Window with flowers'' with neural network initialization: (a) original image, (b) blurred and noisy image (PSNR 19.51, SSIM 0.3023), (c) network reconstruction (PSNR 19.88, SSIM 0.3952), (d) standard iterative scheme (PSNR 21.53, SSIM 0.5169), (e) final reconstruction (PSNR 22.09, SSIM 0.5727), (f) IRMGL (PSNR 21.22, SSIM 0.4525), (g) ADMM-BM3D (PSNR \textbf{25.39}, SSIM \textbf{0.7874}), (h) final reconstruction of the mixed scheme initialized by ADMM-BM3D (PSNR 25.06, SSIM 0.7825). In each panel, the region in the yellow box is magnified in the inset; the values of PSNR and SSIM refer to this region.}
\label{fig:network_test_zoom_window}
\end{figure}

\begin{table}
    \centering
    \setlength{\tabcolsep}{4pt}
    \begin{tabular}{c|cccc|cccc}
        & \multicolumn{4}{c|}{House front} & \multicolumn{4}{c}{Window with flowers}\\
        & RRE\(\downarrow\) & PSNR\(\uparrow\) & SSIM\(\uparrow\) & GMSD\(\downarrow\)
        & RRE\(\downarrow\) & PSNR\(\uparrow\) & SSIM\(\uparrow\) & GMSD\(\downarrow\) \\
        \hline \hline
        blurred & 0.2154 & 20.12 & 0.3489 & 0.2733  
                & 0.1538 & 23.11 & 0.6398 & 0.2520 \\
		  $\bx_{0}^{\delta}$ & 0.1719 & 22.08 & 0.4741 & 0.1878
                           & 0.1239 & 24.98 & 0.7369 & 0.2046 \\
        $\bx_{1}^{\delta}$ & 0.1636 & 22.51 & 0.4993 & 0.1802
                           & 0.1155 & 25.59 & 0.7549 & 0.1944 \\
		$\bx_{n_{\mathrm{std}}}^{\delta}$  & 0.1600 & 22.70 & 0.5123 & 0.1792
                                           & 0.1111 & 25.93 & 0.7670 & 0.1812  \\
        \hline
		  $\bx_{n_{\mathrm{std}}+1}^{\delta}$ & 0.1598 & 22.71 & 0.5129 & 0.1781
                                            & 0.1101 & 26.01 & 0.7697 & 0.1752 \\
        $\bx_{n_{\mathrm{std}}+2}^{\delta}$ & 0.1597 & 22.71 & 0.5134 & 0.1774
                                            & 0.1093 & 26.08 & 0.7718 & 0.1704 \\
        $\bx_{n_{\mathrm{std}}+3}^{\delta}$ & 0.1597 & 22.72 & 0.5137 & 0.1769
                                            & 0.1086 & 26.13 & 0.7736 & 0.1667 \\
        \hline
        \hline
        IRMGL
            & 0.1856 & 21.41 & 0.4216 & 0.2013
            & 0.1208 & 25.20 & 0.7210 & 0.1902 \\
        ADMM-BM3D
            & 0.1619 & 22.60 & 0.5162 & 0.1682
            & 0.1015 & 26.72 & 0.7916 & 0.1356 \\
        $\bx_{n_{\mathrm{std}}+3}^{\delta}$ (ADMM)
            & \textbf{0.1547} & \textbf{22.99} & \textbf{0.5446} & \textbf{0.1663}
            & \textbf{0.0988} & \textbf{26.95} & \textbf{0.8015} & \textbf{0.1295} \\
	\end{tabular}%
	\caption{Performance metrics for deblurring and denoising with neural network initialization: RRE, PSNR, SSIM, and GMSD of the blurred image, the network reconstruction, and the iterated graph Laplacian reconstructions, where the standard iteration is used above the horizontal line and the error-equation stage of the mixed scheme is used below the line. The results of IRMGL, ADMM-BM3D as well as the iterated graph Laplacian using ADMM-BM3D as initialization are shown in the bottom lines for comparison.}
    \label{Table:network_test}
\end{table}

\section*{Conclusions}
We have introduced three iterated graph Laplacian regularization schemes for linear image restoration problems: a standard iteration, an error-equation iteration, and a mixed strategy combining the two. In all cases, the graph Laplacian is updated throughout the reconstruction process, allowing the regularizer to progressively incorporate more accurate structural information about the unknown image.

We proved finite-iteration stability and convergence in the vanishing-noise regime under a priori parameter and stopping rules. The numerical experiments on deblurring and computed tomography confirm that updating the graph improves the single-step graphLa\texttt{+}$\Psi$ reconstruction. In all the tests, the schemes also improved the initial reconstruction, whether it was computed by Tikhonov regularization, total variation, filtered backprojection, IRMGL, ADMM-BM3D, or a trained network. The standard iterations mainly improve the overall reconstruction quality, whereas the error-equation stage can further enhance local details once the standard scheme has reached a plateau.

\appendix

\section{Joint coercivity assumption}\label{sec:appendix-coercivity}

In this section, we verify that the uniform joint coercivity
\cref{eq:joint-coercivity} of \Cref{hyp:joint-coercivity} holds for two
representative model operators: a periodic deblurring model and a full-angle CT
model. The numerical experiments of \Cref{ssec:CT} instead use a
sparse-angle CT model, which is more challenging and better showcases the
reconstruction properties of the iterative methods introduced here. The
uniform-in-$N$ analysis of the sparse-angle operator used in the experiments
is outside the scope of this appendix; we therefore restrict the CT analysis
to the full-angle model under the discrete stability assumptions stated below.

By the contrast cap in \cref{eq:gaussian-weights}, every graph input
$\bz\in X$ and every spatial edge satisfy
\[
0<\omega_*\le \omega_\bz(i,j)\le 1,
\qquad
\omega_*\coloneqq \exp\!\left(-\frac{B_{\rm c}^2}{\sigma^2}\right).
\]
Thus the estimates below are uniform over all graph inputs. When the graph input takes values in an interval
of length at most $B_{\rm c}$, the contrast cap has no effect.

\begin{remark}\label{rem:finite-dim-coercivity}
For a fixed discretization level $N$, the verification of the joint coercivity is elementary. If the spatial graph is connected, then the positive edge weights imply
\[
\ker(\Delta_\bz)=\operatorname{span}\{\bone\}
\qquad \text{for every }\bz\in X.
\]
Here $\bone\in X$ denotes the constant image with all pixel values equal to $1$. Therefore, if $K\bone\ne0$, then $\ker K\cap\ker\Delta_\bz=\{0\}$ for every $\bz$.
Consequently there exists $C_J = C_{J,N}>0$, independent of $\bz$ but possibly
dependent on $N$, such that
\[
\norm{\bu}_2\le C_{J,N}\bigl(\norm{K\bu}_Y+\norm{\Delta_\bz\bu}_q\bigr).
\]
The independence of $\bz$ follows from a compactness argument. By the bounds $\omega_*\le\omega_\bz(i,j)\le1$, the weights of all graph inputs lie in the compact set $[\omega_*,1]^{E}$, where $E$ is the set of spatial edges. The function $(\omega,\bu)\mapsto\norm{K\bu}_Y+\norm{\Delta_\omega\bu}_q$, with $\Delta_\omega$ the graph Laplacian of the weights $\omega$, is continuous on the compact set $[\omega_*,1]^{E}\times\{\bu\in X\colon\norm{\bu}_2=1\}$ and positive there, since the kernel condition holds for every choice of positive weights. Its minimum is therefore positive, and $C_{J,N}$ is the inverse of this minimum.
For each matrix operator used in the experiments (the deblurring and CT
matrices), it is straightforward to check that $K\bone\ne0$. This alone,
however, does not yield the dimension-free constant of \Cref{hyp:joint-coercivity},
since the constant $C_{J,N}$ may still diverge as $N\to\infty$.
\end{remark}

\subsection{Dimension-free estimate for a periodic blur model}\label{sec:appendix-deblurring}

We next give a genuinely joint estimate for a simple blurring model operator. Low frequencies are controlled by the blur, while the graph Laplacian controls the remaining frequencies. The result is stated for $q=2$.

Let $n\in\N$, $P_N=(\mathbb Z/n\mathbb Z)^2$, $N=n^2$, $X_N\coloneqq\R^{P_N}$, and let $K_N\colon X_N\to X_N$ be the periodic convolution operator generated by a fixed pixel-scale PSF $h\in\ell^1(\mathbb Z^2)$:
\[
K_N\bu=h*\bu.
\]
In this subsection, both the neighbor relation and the spatial distance in the
graph construction are understood periodically on $P_N$.

Let
\[
\widehat h(\theta)\coloneqq \sum_{r\in\mathbb Z^2}h(r)\,e^{-ir\cdot\theta},
\qquad \theta\in[-\pi,\pi]^2.
\]
We assume only that
\[
\widehat h(0)=\sum_r h(r)\ne0.
\]
In particular, every normalized PSF satisfies this assumption. The symbol $\widehat h$ may vanish away from the origin.

Let $L_N$ be the periodic nearest-neighbor grid Laplacian,
\[
(L_N\bu)(i)=\sum_{j\sim i}(\bu(i)-\bu(j)),
\]
where $j\sim i$ denotes the four nearest neighbors. Its Fourier symbol is
\[
\lambda(\theta)=4-2\cos\theta_1-2\cos\theta_2.
\]
We assume that the graph used to build $\Delta_\bz$ contains the nearest-neighbor grid edges. Since the corresponding graph weights are bounded below by $\omega_*$, the Dirichlet forms satisfy
\begin{equation}\label{eq:graph-form-comparison}
\langle \Delta_\bz\bu,\bu\rangle
\ge \omega_*\langle L_N\bu,\bu\rangle,
\end{equation}
where the inner product is the normalized Euclidean inner product on $X_N$.

\begin{proposition}\label{prop:dim-free-deblurring}
Under the assumptions above, there exists a constant $C_J>0$, independent of $N$ and of $\bz$, such that
\[
\norm{\bu}_2\le C_J\bigl(\norm{K_N\bu}_2+\norm{\Delta_\bz\bu}_2\bigr)
\qquad\text{for every }\bu\in X_N.
\]
\end{proposition}

\begin{proof}
Since $\widehat h$ is continuous and $\widehat h(0)\ne0$, there exist $\tau>0$ and $\kappa>0$, independent of $N$, such that
\[
|\widehat h(\theta)|\ge \kappa
\qquad\text{whenever }\lambda(\theta)\le\tau.
\]
Let $P_\tau$ be the Fourier projection onto the modes with $\lambda(\theta)\le\tau$, and write
\[
\bu_\ell=P_\tau\bu,
\qquad
\bu_h=(I-P_\tau)\bu.
\]
On the low-frequency space, $K_N$ is uniformly bounded below:
\[
\norm{\bu_\ell}_2\le \kappa^{-1}\norm{K_N\bu_\ell}_2.
\]
By Young's inequality, $\norm{K_N\bv}_2\le M_K\norm{\bv}_2$ with $M_K=\norm{h}_{\ell^1}$, so
\begin{equation}\label{eq:low-bound-blur}
\norm{\bu_\ell}_2
\le \kappa^{-1}\norm{K_N\bu}_2+\kappa^{-1}M_K\norm{\bu_h}_2.
\end{equation}
For the high-frequency part, the definition of $P_\tau$ gives
\[
\tau\norm{\bu_h}_2^2\le \langle L_N\bu,\bu\rangle.
\]
Using \cref{eq:graph-form-comparison} and Cauchy--Schwarz,
\begin{equation}\label{eq:high-bound-blur}
\norm{\bu_h}_2^2
\le (\tau\omega_*)^{-1}\langle\Delta_\bz\bu,\bu\rangle
\le (\tau\omega_*)^{-1}\norm{\Delta_\bz\bu}_2\norm{\bu}_2.
\end{equation}
Combining \cref{eq:low-bound-blur} with $\norm{\bu}_2\le\norm{\bu_\ell}_2+\norm{\bu_h}_2$ gives
\[
\norm{\bu}_2
\le \kappa^{-1}\norm{K_N\bu}_2
+\bigl(1+\kappa^{-1}M_K\bigr)\norm{\bu_h}_2.
\]
Insert \cref{eq:high-bound-blur}. With
\[
a\coloneqq \norm{\bu}_2,
\qquad b\coloneqq \norm{K_N\bu}_2,
\qquad c\coloneqq \norm{\Delta_\bz\bu}_2,
\]
we obtain
\[
a\le \kappa^{-1}b+
\bigl(1+\kappa^{-1}M_K\bigr)(\tau\omega_*)^{-1/2}(ca)^{1/2}.
\]
Set
\[
A\coloneqq \bigl(1+\kappa^{-1}M_K\bigr)(\tau\omega_*)^{-1/2}.
\]
Young's inequality gives $A(ca)^{1/2}\le \frac12 a+\frac12 A^2c$. Therefore
\[
a\le 2\kappa^{-1}b+A^2c.
\]
Equivalently,
\[
\norm{\bu}_2\le C_J\bigl(\norm{K_N\bu}_2+\norm{\Delta_\bz\bu}_2\bigr),
\]
with $C_J=\max\{2\kappa^{-1},A^2\}$, which is independent of $N$ and $\bz$.
\end{proof}

\subsection{Dimension-free estimate for full-angle CT}\label{sec:appendix-CT}

For CT, the cleanest dimension-free example is full-angle tomography equipped
with the natural filtered data norm. At the continuous level, the relevant
estimate is the Plancherel identity for the Radon transform. More precisely, the
Fourier slice theorem gives
\[
\mathcal F_s(\mathcal R f)(\theta,\rho)
=
\widehat f(\rho\cos\theta,\rho\sin\theta),
\]
where $\mathcal R$ is the Radon transform and $\mathcal F_s$ denotes the
one-dimensional Fourier transform in the detector variable; moreover, $\theta\in[0,\pi)$ is the projection angle, $\rho\in\R$ is the frequency variable dual to the detector coordinate, and $\widehat f$ is the two-dimensional Fourier transform of $f\in L^2(\R^2)$. Hence, by
Plancherel's theorem and the polar-coordinate change of variables in frequency
space,
\[
\norm{f}_{L^2(\R^2)}^2
=
c_{\mathcal F}
\int_0^\pi\!\int_\R
|\rho|\,|\mathcal F_s(\mathcal R f)(\theta,\rho)|^2\,d\rho\,d\theta ,
\]
where $c_{\mathcal F}>0$ depends only on the Fourier-transform normalization; see,
for example, \cite{natterer2001}. Absorbing this
constant into the definition of the data norm, set
\[
\norm{g}_Y^2\coloneqq 
c_{\mathcal F}
\int_0^\pi\!\int_\R
|\rho|\,|\mathcal F_s g(\theta,\rho)|^2\,d\rho\,d\theta .
\]
Then
\[
\norm{f}_{L^2(\R^2)}=\norm{\mathcal R f}_Y .
\]

At the continuous level, the Radon transform
is therefore an isometry with respect to this data norm; in particular, the coercivity bound
$\norm{f}_{L^2(\R^2)}\le\norm{\mathcal R f}_Y$ holds with constant $1$. Let $X_N$ denote the image space at discretization level $N$. The
finite-dimensional operator $K_N$ and the discrete data norm $\norm{\cdot}_Y$ are
obtained by discretizing $\mathcal R$ and by applying quadrature to the filtered
integral over $(\theta,\rho)$. For the discretized CT model, we assume the
following uniform analogue of the continuous lower bound:
\begin{equation}\label{eq:ct-filtered-lower-bound}
\norm{\bu}_2\le C_R\norm{K_N\bu}_Y,
\end{equation}
with $C_R$ independent of $N$.

\begin{proposition}\label{prop:dim-free-CT}
Assume that $K_N$ is a full-angle CT discretization satisfying
\cref{eq:ct-filtered-lower-bound}, and that
\[
\norm{K_N\bu}_Y\le M_K\norm{\bu}_2
\]
with $M_K$ independent of $N$. Then, for every $q\in[1,2]$, the joint coercivity
condition holds on $X_N$ with a constant independent of $N$ and $\bz$:
\[
\norm{\bu}_2
\le
C_R\norm{K_N\bu}_Y
\le
C_R\bigl(\norm{K_N\bu}_Y+\norm{\Delta_\bz\bu}_q\bigr).
\]
\end{proposition}

\begin{proof}
The claim is immediate from 
\cref{eq:ct-filtered-lower-bound}.
\end{proof}

\bibliographystyle{plain}
\bibliography{iterated_graph_Laplacian.bib}
\end{document}